\documentclass[11pt]{amsart}
\usepackage{graphicx}
\usepackage{amsmath,amsthm}
\usepackage{amssymb}
\usepackage{pinlabel}
\usepackage[left = 3cm, right = 3 cm , top = 3 cm , bottom = 3cm ]{geometry}
\usepackage[utf8]{inputenc}
\usepackage{todonotes}
\usepackage{enumerate}
\usepackage{tikz-cd}

\newtheorem{thm}{Theorem}[section]
\newtheorem{prop}[thm]{Proposition}
\newtheorem{lem}[thm]{Lemma}
\newtheorem{cor}[thm]{Corollary}

\newtheorem{question}[thm]{Question}
\newtheorem{construction}[thm]{Construction}

\newtheorem{notation}[thm]{Notation}
\newtheorem{convention}[thm]{Convention}
\newtheorem{example}[thm]{Example}

\theoremstyle{remark}
\newtheorem{remark}[thm]{Remark}

\theoremstyle{definition}
\newtheorem{definition}[thm]{Definition}

\begin{document}
	
\title[Bound for distance in the pants graph]{Bounds for the number of moves between pants decompositions, and between triangulations}
\author{Marc Lackenby, Mehdi Yazdi}
\thanks{ML was partially supported by EPSRC grant EP/Y004256/1. For the purpose of open access, the authors have applied a CC BY public copyright licence to any author accepted manuscript arising from this submission.}
%
%
%
\maketitle

\begin{abstract}
	Given two pants decompositions of a compact orientable surface $S$, we give an upper bound for their distance in the pants graph that depends logarithmically on their intersection number and polynomially on the Euler characteristic of $S$. As a consequence, we find an upper bound on the volume of the convex core of a maximal cusp (which is a hyperbolic structures on $S \times \mathbb{R}$ where given pants decompositions of the conformal boundary are pinched to annular cusps). As a further application, we give an upper bound for the Weil--Petersson distance between two points in the Teichm\"uller space of $S$ in terms of their corresponding short pants decompositions. Similarly, given two one-vertex triangulations of $S$, we give an upper bound for the number of flips and twist maps needed to convert one triangulation into the other.  The proofs rely on using pre-triangulations, train tracks, and an algorithm of Agol, Hass, and Thurston. 
\end{abstract}

\section{Introduction}	

\subsection{Pants decompositions}
Let $S$ be a compact orientable surface. An \emph{essential} curve on $S$ is a homotopically non-trivial and non-boundary-parallel simple closed curve. The \emph{curve graph} $\mathcal{C}(S)$ of $S$, defined by Harvey \cite{harvey1981boundary}, is a simplicial graph whose vertices are isotopy classes of essential simple closed curves on $S$, and with two vertices joined by an edge if their simple closed curves are disjoint up to isotopy. Hempel's Lemma \cite{hempel20013} states that the distance between two simple closed curves $\alpha$ and $\beta$ in the curve graph of $S$ is at most a logarithmic function of their geometric intersection number, namely $2 \log_2 i(\alpha, \beta)+2$. 

A \emph{pants decomposition} of $S$ is a maximal collection of disjoint, pairwise non-parallel, essential simple closed curves on $S$ up to isotopy. The \emph{pants graph} $\mathcal{P}(S)$ of $S$, defined by Hatcher and Thurston \cite{hatcher1980presentation}, is a simplicial graph whose vertices are pants decompositions of $S$, with two vertices joined by an edge if their pants decompositions are related by one \emph{simple} or \emph{associativity} move as in Figure \ref{fig:pants graph moves}. Hatcher and Thurston showed that the pants graph is connected. The pants graph is intimately related to hyperbolic geometry in dimensions two and three. For example, Brock \cite{brock2003weil} showed that the Teichm\"{u}ller space of $S$ equipped with the Weil--Petersson metric is quasi-isometric to the pants graph of $S$, with the quasi-isometry map sending a marked hyperbolic metric to one of its \emph{short} pants decompositions (i.e. each curve in the pants decomposition being shorter than the \emph{Bers constant}). 

\begin{figure}
	
	\includegraphics*[width= 3 in]{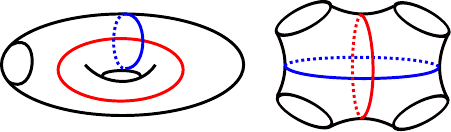}
	
	\caption{The edges of the pants graph correspond to simple moves (left) and associativity moves (right). In each case, the red curve is replaced with the blue curve, and the remaining curves are unchanged.}
	\label{fig:pants graph moves}

\end{figure}

In light of Hempel's Lemma, a natural question is whether there is an upper bound for the distance between two arbitrary pants decompositions $P$ and $P'$ in the pants graph that depends logarithmically on the geometric intersection number $i(P, P')$. The following example shows that this is too optimistic if we do not take into account the topological complexity of the underlying surface.
 
\begin{example}
	The complexity of a compact orientable surface $S$ of genus $g$ with $b$ boundary components is defined as $\xi(S) = 3g-3 +b$. It is easy to see that $\xi(S)$ is the number of curves in a pants decomposition of $S$. Given an essential multicurve $\gamma$ in $S$, define the subsurface $S_\gamma$ of $S$ to be the components of $S\setminus N(\gamma)$ that are not pairs of pants, where $N(\gamma)$ is a regular neighbourhood of $\gamma$.  Define $\mathcal{P}_\gamma(S)$ to be the full subgraph of the pants graph $\mathcal{P}(S)$ consisting of pants decompositions containing $\gamma$. Taylor and Zupan \cite{taylor2016products} showed that if $\gamma$ is an essential multicurve such that $S_\gamma$ is a disjoint union of surfaces of complexity $1$ then $\mathcal{P}_{\gamma}(S)$ is totally geodesic in $\mathcal{P}(S)$. Note that in this case $\mathcal{P}_\gamma(S)$ is a product of copies of the Farey graph. 
	
	Now suppose that $S$ is closed 	and take $\gamma$ to be a multicurve such that $S_\gamma$ is a disjoint union of $n$ surfaces of complexity $1$ where $n \geq g$. If $P$ and $P'$ are pants decompositions of $S$ both containing $\gamma$, then the distance between them is equal to the sum over distances between the restrictions of $P$ and $P'$ to the components of $S_\gamma$. For example if the restriction of $P $ and $P'$ to each component of $S_\gamma$ consists of two curves intersecting either once or twice then $i(P, P') \leq 2n $ while $d_{\mathcal{P}_S}(P, P') \geq n$. Recall that $n$ is comparable to $g$ since we assumed that $g \leq n $, and also for Euler characteristic reasons $n \leq 2g-2$. This shows that there is no general upper bound of the form 
	\[ A (\log i(P, P'))^B + C\]
	for distance in the pants graph, where $A, B, C$ are positive universal constants independent of $g$.
	\label{ex: optimality of bound}
\end{example}

Given the above example, it is natural to ask if there is an upper bound for the distance in the pants graph that depends logarithmically on the geometric intersection number of the two pants decompositions, and polynomially on the Euler characteristic of the underlying surface. A first attempt would be to try to adapt the proof of Hempel's Lemma to the context of pants decompositions. The proof that we have in mind uses a surgery trick, due to Lickorish \cite{lickorish1962representation}, to construct a curve $\gamma$ such that both $i(\gamma, \alpha)$ and $i(\gamma, \beta)$ are at most one half of $i(\alpha, \beta)$, and then it proceeds inductively. The reader familiar with the proof would soon realise that if one starts with multicurves $\alpha$ and $\beta$ instead, the produced multicurve $\gamma$ might have more or less components. In particular, even if $\alpha$ and $\beta$ are pants decompositions, there is no guarantee that $\gamma$ would be a pants decomposition too. A second attempt would be to apply Hempel's Lemma successively and $|P|$ times, with $|P|$ the number of curves in the pants decomposition $P$, to convert one curve of $P$ at a time to one curve of $P'$. The issue is that applying the proof of Hempel's Lemma to the first curve in $P$ makes the remaining curves more complicated rapidly, and its successive application does not seem to produce a good bound for the distance between $P$ and $P'$ in the pants graph. Nevertheless, using different techniques, we show that such a bound exists. Throughout the paper, the implicit constants in the big $O$ notation are universal; in particular they do not depend on the underlying surface. Note that if $P$ and $P'$ are two pants decompositions of $S$ that intersect each other exactly once, then the distance between $P$ and $P'$ in the pants graph is exactly $1$. Therefore, the hypothesis $i(P, P') \geq 2$ in the following theorem is not restrictive.

\begin{thm}
	Let $S$ be a compact connected orientable surface, and $P$ and $P'$ be pants decompositions of $S$ with $i(P, P') \geq 2$. The distance between $P$ and $P'$ in the pants graph is
	\[ O(|\chi(S)|^2) \hspace{1mm}  \log (i(P,P')). \] 
	
	Moreover, there is an algorithm that constructs a (not necessarily geodesic) path $P = P_0, P_1, \cdots, P_n = P'$ of length $ n = O(|\chi(S)|^2) \hspace{1mm} \log (i(P,P'))$ connecting $P$ and $P'$ in the pants graph, in time that is a polynomial function of $|\chi(S)|$ and $\log(i(P,P'))$. Here, for the input, $P$ is given as a union of pairs of pants with gluing instructions, and $P'$ is given in normal form with respect to $P$. For the output, each $P_i$ is given as a union of pairs of pants with gluing instructions together with the pants move from $P_{i-1}$ to $P_{i}$. 
	\label{thm:upper bound for distance in pants graph}  
\end{thm}

\begin{remark}
	Let $\chi = |\chi(S)|$. Example \ref{ex: optimality of bound} shows that any general upper bound for distance in the pants graph should be at least of order
	\[ \frac{\chi}{\log \chi} \log i(P, P'). \]
	Therefore our bound and an optimal bound differ by a multiplicative factor of at most $\chi \log \chi$. 
\end{remark}

We now briefly explain what we mean by normal form for $P'$ with respect to $P$ and refer the reader to Section \ref{sec: normal form for pants decompositions} and Figure \ref{fig: Dehn coordiantes} for precise definitions. The \emph{Dehn coordinates} express the isotopy class of any multicurve $\gamma$ in terms of 
\begin{enumerate}
	\item the geometric \emph{intersection numbers} of $\gamma$ with each curve in a fixed pants decomposition $P$, and 
	\item the \emph{twist numbers} about the curves in $P$. 
\end{enumerate}

Both $i(P,P')$ and the distance in $\mathcal{P}(S)$ between $P$ and $P'$ are unchanged by the action of the mapping class group of $S - P$, which is generated by Dehn twists about the curves in $P$. Dehn twisting $P'$ about a curve $\alpha$ of $P$ does not change the intersection numbers, and changes the twist number about $\alpha$ by $m_\alpha$, where $m_\alpha$ is the intersection number with $\alpha$.  Therefore, for the purpose of Theorem \ref{thm:upper bound for distance in pants graph}, it is enough to know the intersection numbers as in (1), and the twisting numbers modulo the corresponding intersection numbers. This is essentially the data that is referred to as normal form for $P'$ with respect to $P$. 

Our proof of Theorem \ref{thm:upper bound for distance in pants graph} uses pre-triangulations (Definition \ref{def:pre-triangulation}), train tracks, as well as an algorithm of Agol, Hass, and Thurston \cite{agol2006computational} for counting the number of orbits for the action of a pseudogroup generated by \emph{isometries} between subintervals of a (discrete) \emph{interval} $[1,N]: = \{ 1, \cdots, N\}$. See the Background Section. 

\begin{convention}
	To simplify the statements, throughout the paper we allow a 1-complex to have components that are simple closed curves with no vertices on them. Any such component would be considered as an edge of the 1-complex as well. 
	\label{convention: 1-complex}
\end{convention}

\begin{definition}[Pre-triangulation]
A \emph{pre-triangulation} of a surface $S$ is an embedded finite $1$-complex $\mathcal{T}$ in $S$ such that 

\begin{enumerate}
	\item every complementary region is either a disc, or an annulus, or a pair of pants; and 
	\item each component of $\partial S$ either lies inside $\mathcal{T}$ or is disjoint from $\mathcal{T}$. 
\end{enumerate}	
\label{def:pre-triangulation}
\end{definition} 
 
 With Convention \ref{convention: 1-complex}, both triangulations and pants decompositions are special cases of pre-triangulations. The connection between pants decompositions and pre-triangulations is as follows; see Section \ref{sec: pre-triangulations to pants} for the details. We describe a natural procedure that given a pre-triangulation $\mathcal{T}$ and an ordering $\mathrm{O}$ on the  set of edges of $\mathcal{T}$, constructs a pants decomposition $\mathrm{\textbf{P}}(\mathcal{T}, \mathrm{O})$ of $S$ (Construction \ref{const:triangulation-to-pants}). Conversely, given a pants decomposition $P$ and a choice of a suitable 1-complex $\gamma$ in $S$, called a \emph{nerve} (Definition \ref{def:nerve}), a connected pre-triangulation $\mathcal{T}=\mathrm{\textbf{T}}(P, \gamma)$ of $S$ can be constructed (Construction \ref{const:pants-to-triangulation}). It is shown that for any pre-triangulation $\mathcal{T}$ and orderings $\mathrm{O}$ and $\mathrm{O}'$ on the edges of $\mathcal{T}$, the distance between the pants decompositions $\mathrm{\textbf{P}}(\mathcal{T}, \mathrm{O})$ and $\mathrm{\textbf{P}}(\mathcal{T}, \mathrm{O}')$ in the pants graph is $O(|\mathcal{T}|^2)$, where $|\mathcal{T}|$ is the number of edges of $\mathcal{T}$ (Lemma \ref{changing the ordering}). Using this, we show in Lemma \ref{bound in terms of intersection} that for any two pre-triangulations $\mathcal{T}$ and $\mathcal{T}'$ and orderings $\mathrm{O}$ and $\mathrm{O}'$ on their edges, the distance between the pants decompositions $\mathrm{\textbf{P}}(\mathcal{T}, \mathrm{O})$ and $\mathrm{\textbf{P}}(\mathcal{T}', \mathrm{O}')$ is $O(|\mathcal{T} \cup \mathcal{T}'|^2)$, where $|\mathcal{T} \cup \mathcal{T}'|$ is the number of edges of the pre-triangulation $\mathcal{T} \cup \mathcal{T}'$ obtained by superimposing $\mathcal{T}$ and $\mathcal{T}'$.  In particular, this implies Corollary \ref{quadratic bound} stating that the distance between two pants decompositions $P$ and $P'$ in the pants graph is $O(i(P, P')^2)$. As far as we are aware, albeit much weaker than Theorem \ref{thm:upper bound for distance in pants graph}, even this quadratic upper bound for distance in the pants graph is new. 
 
 \subsection{One-vertex triangulations}
 \label{subsection: triangulations}
 The idea of the proof of Theorem \ref{thm:upper bound for distance in pants graph} can be used to give an upper bound for the number of \emph{flip and twist moves} between two triangulations of a surface. The \emph{flip graph} of a closed orientable surface $S$ is a simplicial graph whose vertices are one-vertex triangulations of $S$ up to isotopy, with two vertices joined by an edge if they differ by a \emph{flip} of a diagonal in a simplicial square. 
 It is well-known that the flip graph is connected; see for example \cite[page 36]{mosher1988tiling}. The mapping class group of $S$ acts by isometries on the flip graph, and the action is properly discontinuous and cocompact. By the Švarc–Milnor Lemma, the flip graph of $S$ is quasi-isometric to the mapping class group of $S$, where the metric on the mapping class group is the word metric with respect to any fixed finite generating set. Given two one-vertex triangulations $\mathcal{T}$ and $\mathcal{T}'$ of $S$ with the same vertex, one can convert $\mathcal{T}$ to $\mathcal{T}'$ using a sequence of at most $i(\mathcal{T}, \mathcal{T}')$ flips; see e.g. \cite[Lemma 2.1]{disarlo2019geometry} and below for the definition of $i(\mathcal{T}, \mathcal{T}')$. In general, the linear dependency on $i(\mathcal{T}, \mathcal{T}')$ cannot be improved, essentially because the flip graph is quasi-isometric to the mapping class group. For example, if $\mathcal{T}$ is a fixed one-vertex triangulation of $S$, $\alpha$ is a fixed essential
 simple closed curve in $S$, and $\mathcal{T}':= (T_\alpha)^n(\mathcal{T})$ where $T_\alpha$ is the Dehn twist about $\alpha$, then
 
 \begin{enumerate}
 	\item[-] $i(\mathcal{T}', \mathcal{T})$ is linear in $n$; and  
 	
 	\item[-] $\mathcal{T}'$ and $\mathcal{T}$ have distance at least a constant multiple of $n$ in the flip graph. This is because by the work of Farb, Lubotzky, and Minsky \cite{Frab2001rank}, the word length of $(T_\alpha)^n$ in the mapping class group is linear in $n$. 
 \end{enumerate}
 
 However, we show in Theorem \ref{thm: one-vertex triangulations, fixed vertices} that if we are allowed to use powers of Dehn twists (\emph{twist maps} for short) in addition to flips, then one can convert $\mathcal{T}$ to $\mathcal{T}'$ much more economically. In this case, the total number of flips and twist maps needed will be bounded above by a polynomial function of $\log(i(\mathcal{T}, \mathcal{T}'))$ and $|\chi(S)|$. 
 
 Given one-vertex (respectively ideal) triangulations $\mathcal{T}$ of a surface $S$ with the same vertex set $V$, let $i(\mathcal{T} , \mathcal{T}')$ be the minimum number of intersections between $\mathcal{T}$ and $\mathcal{T'}$ outside of $V$, up to isotopies fixing $V$. Therefore $i(\mathcal{T}, \mathcal{T}') = 0$ if and only if $\mathcal{T} = \mathcal{T}'$ up to isotopies keeping the vertex set fixed. Furthermore, if $i(\mathcal{T}, \mathcal{T}') = 1$ then $\mathcal{T}$ and $\mathcal{T'}$ are related by a flip. Therefore in the following theorem the assumption $i(\mathcal{T}, \mathcal{T}') \geq 2$ is not restrictive. 

 \begin{thm} 
 	Let $S$ be a compact orientable surface. When $S$ is closed (respectively, has non-empty boundary), let $\mathcal{T}$ and $\mathcal{T}'$ be one-vertex (respectively, ideal) triangulations of $S$. Assume that $i(\mathcal{T}, \mathcal{T}') \geq 2$. There is a sequence $\mathcal{T}= \mathcal{T}_0, \mathcal{T}_1, \cdots, \mathcal{T}_n = \mathcal{T}'$ of one-vertex (respectively, ideal) triangulations of $S$ such that: 
	
	\begin{enumerate}
 		\item Each $\mathcal{T}_{i+1}$ is obtained from $\mathcal{T}_i$ by either a flip or a power of a Dehn twist. When $\mathcal{T}_{i+1}$ is obtained from $\mathcal{T}_i$ by Dehn twisting $k$ times about a curve $\alpha$, then $\alpha$ is a normal curve intersecting each edge of $\mathcal{T}_i$ at most three times, and the absolute value of $k$ is bounded above by $i(\mathcal{T}, \mathcal{T}')$.
		
		\item $n = O(|\chi(S)|^3) \log(i(\mathcal{T}, \mathcal{T}'))$. 
		
	\end{enumerate}
	
	Moreover, there is an algorithm that constructs the sequence $\mathcal{T}_1, \cdots, \mathcal{T}_n$ in time that is a polynomial function of $\log(i(\mathcal{T}, \mathcal{T}'))$ and $|\chi(S)|$. Here we assume that $\mathcal{T}$ and $\mathcal{T}'$ have the same vertex (when $S$ is closed), $\mathcal{T}$ is given as a union of (possibly ideal) triangles with gluing instructions, and $\mathcal{T}'$ is given in terms of its normal coordinates with respect to $\mathcal{T}$. For the output, each $\mathcal{T}_i$ is given as a union of (possibly ideal) triangles with gluing instructions together with the flip or twist move from $\mathcal{T}_i$ to $\mathcal{T}_{i+1}$. 

  	\label{thm: one-vertex triangulations, fixed vertices} 
 \end{thm}
 
 See Definition \ref{def:normal form triangulation} for the definition of the normal form for a triangulation with respect to another triangulation. Theorem \ref{thm: one-vertex triangulations, fixed vertices} has been used by Baroni \cite{baroni2024uniformly} and Lackenby \cite{lackenby2024some} to study algorithmic problems about curves on surfaces and surface homeomorphisms, and it is likely to have further applications. Marc Bell and Richard Webb have communicated to us \cite{bell} that they, independently, were aware of an analogue of Theorem \ref{thm: one-vertex triangulations, fixed vertices}.
 
 In Theorem \ref{thm: one-vertex triangulations, variable vertices}, as a precursor to Theorem \ref{thm: one-vertex triangulations, fixed vertices}, we prove an analogous statement (Theorem \ref{thm: one-vertex triangulations, variable vertices}) for a similar set of moves between two \emph{polygonal decompositions} (Definition \ref{def: polygonal decomposition}) of a compact orientable surface. In Theorem \ref{thm: edge swaps on spines}, we also prove a similar result for a set of moves between two spines of a compact orientable surface. 
 
 \subsection{Volumes of hyperbolic 3-manifolds}
 
 A \emph{maximal cusp} is a geometrically finite hyperbolic 3-manifold such that each component of its conformal boundary is a thrice-punctured sphere \cite{canary2003approximation}. The \emph{convex core} of a non-compact hyperbolic manifold is the smallest geodesically convex subset that contains every closed geodesic. Theorem \ref{thm:upper bound for distance in pants graph} together with the work of Agol \cite{agol2003small} implies the following. 
 
 \begin{cor}
 	Let $\Sigma$ a closed orientable surface of genus $g \geq 2$, and $P$ and $P'$ be pants decompositions of $\Sigma$ with no curve in common. Assume that $M$ is a maximal cusp obtained from a quasi-Fuchsian 3-manifold homeomorphic to $\Sigma \times \mathbb{R}$ by pinching the multicurves $P$ and $P'$ to annular cusps on the two conformal boundary components of $M$. The volume of the convex core of $M$ is  
 	\[ O(g^2) \hspace{1mm} \log (i(P, P')) . \]
 	\label{upper bound for volume}
 \end{cor}

In Proposition \ref{sharpness of volume bound} we give examples demonstrating that the bound in Corollary \ref{upper bound for volume} is sharp up to a multiplicative factor of $g \log(g)$.

\subsection{Weil--Petersson distance in the Teichm\"{u}ller space}

Let $S$ be an orientable surface of finite type, and denote by $\mathrm{Teich}(S)$ the Teichm\"uller space of $S$ (in other words, the space of marked hyperbolic metrics possibly with cusps but with no boundary). Bers showed \cite{bers1985inequality} that there is a constant $C= C(S)$ only depending on the topology of $S$ such that for every $X \in \mathrm{Teich}(S)$ there is a pants decomposition $P_X$ in $S$ such that each curve in $P_X$ has length at most $C$ with respect to the marked hyperbolic metric $X$. In other words every hyperbolic metric on $S$ admits a `short' pants decomposition. By Parlier's \cite{parlier} quantified version of Bers's theorem, which is building on and improving the work of Buser \cite{buser1980riemannsche} and of Buser and Sepp\"al\"a \cite{buser1992symmetric}, every hyperbolic metric $X$ on $S$ has a pants decomposition $P$ in which every curve has length at most $2 \pi |\chi(S)|$. 

Denote the Weil--Petersson metric on $\mathrm{Teich}(S)$ by $d_{\mathrm{WP}}$. Using Theorem \ref{thm:upper bound for distance in pants graph}, the work of Wolpert \cite{wolpert2008lengthfunctions}, Cavendish and Parlier \cite{cavendishparlier}, and Parlier \cite{parlier} we show the following.

\begin{thm}
	Let $S$ be an orientable surface of finite type, and let $X, Y \in \mathrm{Teich}(S)$. Let $P_X$ and $P_Y$ be pants decompositions for $X$ and $Y$ respectively, in which each curve has length at most $2 \pi |\chi(S)|$, which exist by a theorem of Parlier \cite{parlier}. Then
	\[ d_{\mathrm{WP}}(X,Y) \leq O(|\chi(S)|^2) \hspace{1mm} ( 1 + \log (i(P_X, P_Y)+1) ). \]
	\label{thm: Weil-Petersson distance}
\end{thm}

\subsection{Idea of the proof of Theorem \ref{thm:upper bound for distance in pants graph}} Given pants decompositions $P$ and $P'$, we first put $P'$ in normal form with respect to $P$. Then we squeeze parallel normal arcs together to produce a train track $\tau$ such that $P'$ is carried by $\tau$. In other words, there is an integral weight $\mu$ on branches of $\tau$ such that the carried 1-complex $\mathcal{CC}(\tau, \mu)$ is equal to $P'$. 

Given an integrally weighted train track $(\tau, \mu)$, and making some extra choices $\mathcal{I}$, we can define an \emph{orbit counting problem} $\mathrm{OCP}(\tau,\mu, \mathcal{I})$ (Definition \ref{orbit counting problem}). The extra data $\mathcal{I}$ is called an \emph{initial interval identification} and can be safely ignored for now. Applying the Agol--Hass--Thurston algorithm, we can count the number of orbits $\mathrm{OCP}(\tau,\mu, \mathcal{I})$. While the number of orbits of $\mathrm{OCP}(\tau,\mu, \mathcal{I})$ is not any new information (it is always equal to the number of curves in $P'$, which is $(3|\chi(S)| - |\partial S|)/2$), we are interested in the way that the AHT algorithm calculates the number of orbits. In Proposition \ref{running AHT} and Lemma \ref{number of train tracks}, we show that the AHT algorithm does this count in $n$ steps with
\begin{eqnarray*}
	n-1 \leq E (1 + \log |\mu|),
\end{eqnarray*}
where $E = O(|\chi(S)|)$ is the number of branches of $\tau$ and $|\mu| = i(P, P')$ is the total weight of all branches of $\tau$. Moreover, each step of the the AHT algorithm can be seen as a geometric change on the underlying integrally weighted train track, namely as either a \emph{splitting} or a \emph{twirling} (Definition \ref{def:twirling} and Figure \ref{fig:twirling}). The algorithm proceeds until $(\tau, \mu)$ is completely unwound to $P'$, so intuitively the algorithm does the reverse process of squeezing parallel arcs together, in a precise and efficient way.  Hence, there is a sequence $(\tau_i,\mu_i)$ of integrally weighted train tracks for $i =1, \cdots, n$ starting with $(\tau_1, \mu_1)= (\tau, \mu)$ and terminating at $P'$ such that each $(\tau_{i+1}, \mu_{i+1})$ is obtained from $(\tau_i, \mu_i)$ by splitting or twirling, and
\begin{eqnarray}
	n -1 = O(|\chi(S)|) (1 + \log i(P, P')).
	\label{eq:upper bound for the number of train tracks}
\end{eqnarray}	
The underlying 1-complex of each $\tau_i$ is a pre-triangulation. By Lemma \ref{lem:local distance}, since $(\tau_{i+1},  \mu_{i+1})$ is obtained from $(\tau_i, \mu_i)$ by a simple geometric operation (splitting or twirling), the following holds:  there are orderings $\mathrm{O}_i$ on the edges of $\tau_i$ such that for most $i$ the distance between the pants decompositions $\mathrm{\textbf{P}}(\tau_i, \mathrm{O}_i)$ and $\mathrm{\textbf{P}}(\tau_{i+1}, \mathrm{O}_{i+1})$ is $O(|\chi(S)|)$. Moreover, by Lemma \ref{bound in terms of intersection} the distance between the pants decompositions $P$ and $\mathrm{\textbf{P}}(\tau_1, \mathrm{O}_1)$ is $O(|\chi(S)|^2)$, because $i(\tau_1, P) = O(|\chi(S)|)$. Note also that $\mathrm{\textbf{P}}(\tau_n, \mathrm{O}_n) = P'$. The triangle inequality now implies that the distance between $P$ and $P'$ is at most $O(|\chi(S)^2|)+(n-1) O(|\chi(S)|)$, which combined with equation (\ref{eq:upper bound for the number of train tracks}) gives the desired upper bound.

\subsection{Previous related work}

Several authors had previously used either train tracks and the Agol--Hass--Thurston algorithm, or the \emph{tracing algorithm} of Erickson and Nayyeri \cite{erickson2013tracing}, to simplify curves on surfaces efficiently.  

\begin{itemize}
	\item Motivated by the AHT algorithm, Dunfield and D. Thurston \cite{dunfield2006random} adapted Brown's algorithm in the context of splittings of train tracks to decide efficiently whether a random tunnel number one 3-manifold fibers over the circle. 
	
	\item Dynnikov and Wiest \cite{dynnikov2007complexity} introduced a specialised version of the AHT algorithm for curves on a punctured disc, called the \emph{transmission-relaxation algorithm}, to simplify a weighted train track associated with a braid in an $n$-times punctured disc. Using this, they proved that certain algebraic and geometric complexity measures for a braid are comparable up to multiplicative constants only depending on $n$. They showed an improved running time for their algorithm compared to the general version of the AHT algorithm. In the proof of Proposition \ref{running AHT}, we use their improved bound to analyse the time complexity of the AHT algorithm applied to a weighted train track.  
	
	\item Erickson and Nayyeri \cite{erickson2013tracing} gave an algorithm to \emph{trace} a curve $\gamma$ with respect to a triangulation $\mathcal{T}$ in time that is a polynomial function of $\log (i(\mathcal{T}, \gamma))$ and $|\chi(S)|$. 
	
	\item Bell \cite{bell2021simplifying} showed how to simplify a triangulation $\mathcal{T}$ of a surface $S$ with respect to a \emph{curve} $\gamma$ using flips and twist maps, in time that is a polynomial function of $\log (i(\mathcal{T}, \gamma))$ but a super-exponential function of $|\chi(S)|$. He has communicated to us that the source code of his software program \emph{flipper} does this simplification in time that is a polynomial function of both $\log (i(\mathcal{T}, \gamma))$ and $|\chi(S)|$. See also the work of Bell and Webb \cite{bell2016applications}.
	
\end{itemize}

Another novel feature of our work is to show that the AHT algorithm can be used to establish non-algorithmic results such as Corollary \ref{upper bound for volume} and Theorem \ref{thm: Weil-Petersson distance}.

\subsection{Plan of the paper} In Section 2, we discuss the background on the algorithm of Agol, Hass, and Thurston. We also introduce various normal forms needed for the statements of the theorems. These include 
\begin{enumerate}
	\item[-] normal form for a pants decomposition with respect to another pants decomposition. This is a slight modification of Dehn coordinates, and is used in Theorem \ref{thm:upper bound for distance in pants graph}. 
	\item[-] normal form for a triangulation with respect to another triangulation with common vertices. This is a mild generalisation of the usual notion of normal curve with respect to a triangulation, and is used in Theorem \ref{thm: one-vertex triangulations, fixed vertices}. 
	\item[-] normal form for a 1-complex with respect to a polygonal decomposition. This is a more relaxed notion than the usual notion of a normal curve with respect to a triangulation, since the 1-complex can be quite general. This is used as a bookkeeping tool in Theorems \ref{thm: one-vertex triangulations, variable vertices} and \ref{thm: edge swaps on spines}.
	 
\end{enumerate}

In Section 3, we discuss pre-triangulations and their relation to pants decompositions. Section 4 discusses based weighted train tracks and orbit counting problems associated to them. The effect of the AHT algorithm applied to such orbit counting problems is seen as a geometric move on the underlying train track (Proposition \ref{running AHT}). Section 5 contains the proof of Theorem \ref{thm:upper bound for distance in pants graph}. We first prove Theorems \ref{thm: one-vertex triangulations, variable vertices} and \ref{thm: edge swaps on spines} in Section 6, and then use the latter to prove Theorem \ref{thm: one-vertex triangulations, fixed vertices} in Section 7. Section 8 proves Corollary \ref{upper bound for volume} using only the statement of Theorem \ref{thm:upper bound for distance in pants graph} as well as the work of Agol \cite{agol2003small}. Theorem \ref{thm: Weil-Petersson distance} is proved in Section 9. Some open problems are discussed in Section 10.

\section{Background}

\subsection{Agol--Hass--Thurston (AHT) algorithm}

By an \emph{interval} $[M, N]$ with $M \leq N$ integers, we mean the set $\{ M, M+1, \cdots, N\}$. An \emph{isometry} or \emph{pairing} between two intervals is a bijection of the from $x \mapsto x + c$ or $x \mapsto -x+c$ for a constant integer $c$. Let $[1, N]$ be an interval and $\{ g_i \}$, $1 \leq i \leq k$, be a collection of pairings between subintervals of $[1,N]$. Two pairings can be composed if the range of the first lies in the domain of the second. The pairings generate a pseudogroup $\mathcal{G}$ under the operations of composition where defined, inverses, and restriction to subintervals. Two points $x, y \in [1, N]$ are in the same orbit if there is a sequence $h_1, \cdots , h_s$ of (not necessarily distinct) elements of $\mathcal{G}$ such that $h_s \circ \cdots \circ h_1(x) = y$. This forms an equivalence relation on the interval $[1,N]$ and the equivalence classes are called \emph{orbits}. Given an interval $[1,N]$ and isometric pairings $g_1, \cdots, g_k$ as above, the associated \emph{orbit counting problem} is the problem of counting the number of orbits for the action of the pseudogroup $\mathcal{G}$ (generated by $\{ g_i \}$ for $1 \leq i \leq k$) on the interval $[1,N]$. Agol, Hass, and Thurston gave an algorithm for counting the number of such orbits, in time that is a polynomial function of $k$ and $\log(N)$. This logarithmic dependence on $N$ will be crucial for the arguments in this paper.  

An important example of an orbit counting problem comes from the theory of train tracks (see Definition \ref{def:based train track}). For example suppose $S$ is a surface of finite type, $\tau \subset S$ is a train track, and $\gamma$ is a multicurve carried by $\tau$. Then $\gamma$ can be described using integral weights on branches of $\tau$. A question that arises in practice is to determine the number of connected components of $\gamma$ given the weights associated to $\gamma$.  Agol, Hass, and Thurston used their algorithm to count the number of connected components of $\gamma$, in time that is bounded above by a polynomial function of $E \log(N)$, where $E$ is the number of branches of $\tau$ and $N$ is the total weight of $\gamma$ with respect to $\tau$. 

Similarly, Agol, Hass, and Thurston used their algorithm to count the number of connected components of a normal surface $S$ in a triangulated 3-manifold $S$ given by its normal coordinates with respect to the triangulation, and used it to verify an upper bound on the genus of a knot $K \subset M$ in non-deterministic polynomial time. 
 
We now explain further terminology and the AHT algorithm briefly; the reader should see \cite{agol2006computational} for more details. A pairing $g \colon [a, b] \rightarrow [c, d]$ is \emph{orientation-preserving} if $g(a) = c$ and $g(b)=d$, and otherwise \emph{orientation-reversing}. If $a \leq c$ we refer to $[a, b]$ as the \emph{domain} and $[c, d]$ as the \emph{range} of the pairing. If the pairing preserves orientation, its \emph{translation distance} $t$ is defined as $t = c-a = d-b$.  An interval is called \emph{static} if it is in neither the domain nor the range of any pairing. Given a collection of pairings acting on $[1, N ]$, a pairing is \emph{maximal} if its range contains both $N$ and the range of any other pairing containing $N$. A pairing $g \colon  [a, b] \rightarrow [c, d]$ is \emph{periodic} with period $t$ if it is orientation-preserving with translation distance $t$ and $a < c = a+t \leq  b+1$, so there is no gap between the domain and range. The combined interval $[a, d]$ is then called a \emph{periodic interval} of period t.

The orbit-counting algorithm applies a series of the following modifications, in a particular order, to a collection of pairings. 

\begin{itemize}
	\item \emph{Periodic merger}: Given two periodic pairings $g_1$ and $g_2$ with periodic intervals $R_1$ and $R_2$ and periods $t_1$ and $t_2$ such that $\text{width}(R_1 \cap R_2) \geq t_1 + t_2$, the periodic merger replaces $g_1$ and $g_2$ with a single periodic pairing of period $\mathrm{GCD}(t_1, t_2)$. Here $\mathrm{GCD}$ stands for the greatest common divisor.
	
	\begin{remark}
	Assume that a periodic merger is done. Then before the merging, there is a point which lies in at least 3 domains or ranges of pairings, i.e. 
	\[ \exists p \in [1, N] \text{ such that } \sharp\{ i : p \in \mathrm{domain}(g_i) \} + \sharp\{ i : p \in \mathrm{range}(g_i)\} \geq 3. \] 
	In our applications in this paper, every point will be in the domain or range of at most two pairings, hence periodic mergers will not occur. For this reason we will not elaborate on the properties of periodic merger further. 
	\end{remark}	

	\item \emph{Contraction}:
	Contraction can be performed on a static interval $[r, s]$. We eliminate this interval, replace $[1, N ]$ by $[1, N - (s -r + 1)]$, and change each $g_j$ by replacing any point $x$ in a domain or range which lies entirely to the right of $s$ by $x - (s - r + 1)$. We will then decrease the number of orbits by $s -r + 1$, since the eliminated points are each unique representatives of an orbit.
	
	\item \emph{Trimming}:
	The trimming operation simplifies an orientation-reversing pairing whose
	domain and range overlap. 
	
	\begin{remark}
		In our applications in this paper, the orientability of the underlying surface implies that there will be no trimming.
	\end{remark}

	\item \emph{Truncation}:
	Given an interval that lies in the domain or range of exactly one pairing, we can remove it without changing the orbit structure. This operation can be applied to remove points from the right of the interval $[1,N]$.
	Assume that there is a pairing $g \colon [a,b] \rightarrow [c,N]$ and a value $N'$ with $c \leq N'+1 \leq N$, such that all points in the interval $[N'+1,N]$ are in the range of only $g$. Truncating $g$ shortens the interval $[1,N]$ to the interval $[1,N']$, and similarly shortens the domain and range of $g$. If $g$ is orientation-reversing, truncation is applied only when $g$ has disjoint domain and range (i.e. after trimming).
	
	\item \emph{Transmission}: In transmission, a pairing $g_1$ is used to shift the domain and range of a second pairing $g_2$ to the left as much as possible. Once the pairings are shifted leftwards, we can subsequently apply truncation. Transmissions allow significant simplifications of the orbit counting problem in one step.
	
	If $g_1$ is orientation-reversing and has overlapping domain and range, then as a first step in transmission we trim $g_1$. Now consider a pairing $g_1$, and a second pairing $g_2$ whose range is contained in the range of $g_1$. Consider two cases:
	
	\begin{enumerate}
		\item[a)] If $\text{domain}(g_2) \nsubseteq \text{range}(g_1)$: then define the map $g'_2 = g_1^{-r} \circ g_2$ , where $r = 1$ if $g_1$ is orientation-reversing and otherwise $r \geq  1$ is the largest integer such that $g_1^{-r+1}([c_2 , d_2 ])$ is contained in the range of $g_1$. 
		\item[b)] If $\text{domain}(g_2) \subseteq \text{range}(g_1)$: then define the map $g'_2 = g_1^r \circ g_2 \circ g_1^s \colon g_1^{-s}([a_2 ,b_2 ]) \rightarrow 
		g_1^r([c_2 ,d_2 ])$, where $r$ is as above, $s = 1$ if $g$ is orientation reversing, and 
		otherwise $s \geq 1$ is the largest integer such that $g_1^{-s+1}([a _2, b_2 ])$ is contained in the range of $g_1$. The operation of replacing $g_2$ by $g_1^{-r} \circ g_2 \circ g_1^s$ is called a transmission of $g_2$ by $g_1$.
	\end{enumerate}

\end{itemize}

The AHT algorithm repeatedly applies the following steps (1)--(6) to the orbit counting problem.

\begin{enumerate}
	\item Delete any pairings that are restrictions of the identity.
	
	\item Make any possible contractions and, if any exist, increment the orbit counter by the sum of the number of points deleted by the contractions. If the number of pairings remaining is zero, output the number of orbits and stop.
	
	\item Trim all orientation-reversing pairings whose domain and range overlap.
	
	\item Search for pairs of periodic pairings $g_i$ and $g_j$ with periods $t_1$ and $t_2$ and with overlapping periodic intervals $R_1$ and $R_2$ such that $\text{width}(R_1 \cap R_2) \geq t_1+t_2$. If any such pair exists, then perform a merger replacing $g_i$ and $g_j$ by a single periodic pairing. Repeat until no mergers can be performed.
	
	\item Find a maximal $g_i$. For each $g_j \neq g_i$ whose range is contained in $[c_i,N']$, transmit $g_j$ by $g_i$.
	
	\item Find the smallest value of $c$ such that the interval $[c,N']$ intersects the range of exactly one pairing. Truncate the pairing whose range contains the interval $[c,N']$.
	
\end{enumerate}	

A \emph{cycle} of the AHT algorithm consists of applying steps (1)--(6) above. 

\begin{example}[Euclid's algorithm for GCD]
	Agol, Hass, and Thurston's algorithm is in a sense a generalisation of Euclid's algorithm for computing the greatest common divisor (GCD) of two natural numbers $a$ and $b$. Let $N = a+b$, and consider the interval $[1, N]=[1, a+b]$ with pairings 
	\begin{align*}
		&f(x)= x+a , \hspace{6mm} &f \colon [1, b] \rightarrow [a+1, a+b ], \\
		&g(x) = x+b, \hspace{6mm} &g \colon [1, a] \rightarrow [b+1, a+b].
	\end{align*}
	The number of orbits for the action of the pseudogroup generated by $f$ and $g$ is equal to $d: = \mathrm{GCD}(a, b)$. This will be a consequence of the upcoming discussion (or one can show it directly).
	
	
	Let us apply one cycle of AHT algorithm to this orbit counting problem. Steps (1)--(4) are not applicable for the first cycle, and so we apply Step (5). Assume that $a<b$. Then $f$ is maximal, and we transmit $g$ by $f$. Let $b = na + r$ where $n$ is an integer and $0 \leq r < a$. The transmission replaces $g$ by $g' = f^{-n} \circ g$ where
	\[ g'(x) = x+ b - na = x+r \hspace{6mm} g' \colon [1, a] \rightarrow [r+1, r+a]. \]
	Then Step (6) truncates $f$ and replaces $N$ with $N' = r+a$ and $f$ with $f'$ where 
	\[ f'(x) = x+a, \hspace{6mm} f' \colon [1, r] \rightarrow [ a+1, r+a]. \]
	Hence, after one cycle of the AHT algorithm we have replaced the pair $(a, b)$ with the pair $(r, a)$. This is essentially reflecting the equality $\mathrm{GCD}(a,b) = \mathrm{GCD}(r, a)$ in Euclid's algorithm. Repeating this procedure, we obtain the number of orbits. Note that this inductively shows that the number of orbits is equal to $\mathrm{GCD}(a, b)$, where the base of the induction corresponds to the case $a= 0 $.
	\label{ex:Euclid's algorithm}
\end{example}

In Proposition \ref{running AHT}, we will be interested in the effect of applying the AHT algorithm to certain orbit counting problems associated with train tracks. In particular, we would like to see the effect of one cycle as a geometric change in the underlying train track. In order to see a meaningful geometric change, we will shift the steps in one cycle of the AHT algorithm as follows.

\begin{definition}[Shifted cycle of the AHT algorithm]
	The first \emph{shifted cycle} of the AHT algorithm applies steps (1)--(4) of the orbit counting algorithm. Afterward, every shifted cycle applies steps (5)--(6) followed by steps (1)--(4).  Clearly the AHT algorithm is the successive application of shifted cycles as well.
	\label{def:shifted OCA}
\end{definition}

\subsection{Normal form for a simple curve with respect to a triangulation}

Let $\mathcal{T}$ be a triangulation of a compact surface. For a triangle $\Delta$ of $\mathcal{T}$, a \emph{normal arc} is a simple arc with endpoints lying on the interior of edges of $\Delta$ and joining two different sides of $\Delta$. Two normal arcs $\alpha$ and $\beta$ in $\Delta$ are of the same \emph{type} if there is an isotopy of $\Delta$, preserving the edges and vertices of $\Delta$ throughout, taking one arc to the other. Each triangle contains $3$ types of normal arcs, and so there are in total $3t$ types of normal arcs, where $t$ is the number of triangles in $\mathcal{T}$; see Figure \ref{fig: normal arcs}. Let $\mathcal{\gamma}$ be a properly embedded simple multicurve on $S$. We allow $\gamma$ to consist of both simple arcs and simple closed curves. We say $\gamma$ is a \emph{normal curve} with respect to $\mathcal{T}$ if for each triangle $\Delta$ of $\mathcal{T}$, the intersection $\gamma \cap \Delta$ is a union of normal arcs. Fix a bijection between the set of normal arc types and $\{1 , 2 , \cdots, 3t \}$. The \emph{normal vector} or the \emph{normal coordinates} $(\gamma) \in \mathbb{Z}^{3t}$ is the vector recording the number of normal arcs of each type for $\gamma$. 

\begin{figure}
	\includegraphics[width = 1.5 in]{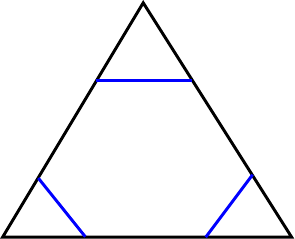}
	\caption{The three types of normal arcs in a triangle.}
	\label{fig: normal arcs}
\end{figure}

Any properly embedded simple multicurve on $S$, with no homotopically trivial simple closed curve component, can be isotoped through isotopies fixing $\partial S$ setwise to a normal curve. 

In the next few subsections, we extend the definition of normal form to other more general situations, which will be used in this paper.

\subsection{Complexity of integral vectors}

Given a vector $v = (v_1, \cdots, v_n) \in \mathbb{Z}^n$ define the \emph{$\ell^1$-norm} of $v$ as 
\[ |v|_1 := \sum_{i=1}^n |v_i|. \]
Given an integer $a$, let $\mathrm{dig}(a)$ be the number of digits of $a$ in binary. 
In the case where $a$ is negative, we view the minus sign in front as an extra digit.
The \emph{bit-sized complexity} of $v$ is defined as 
\[ |v|_{\mathrm{bit}} := \sum_{i=1}^n \mathrm{dig}(v_i).   \]

If $\gamma$ is a normal curve with respect to a triangulation and $v = (\gamma)$ is the normal vector for $\gamma$, then we can speak of the \emph{$\ell^1$-norm of $\gamma$} as $|(\gamma)|_1$, and of the \emph{bit-sized complexity of $\gamma$} as $|(\gamma)|_\mathrm{bit}$.

\label{complexity of vectors}

\subsection{Normal form for pants decompositions }

\label{sec: normal form for pants decompositions}

We define the Dehn parametrisation of isotopy classes of multicurves on a surface, following Penner and Harer \cite[page 13]{penner1992combinatorics}. We then modify Dehn coordinates slightly to define a normal form for a pants decomposition with respect to another pants decomposition. The motivation for this modification is that in Theorem \ref{thm:upper bound for distance in pants graph}, the isotopy class of $P'$ is needed only up to the action of the mapping class group of $S - P$. 

Let $S$ be a compact orientable surface of genus $g$ with $b$ boundary components and let $P$ be a pants decomposition of $S$. Denote the components of $P$ by $\alpha_r$ where $1 \leq r \leq |P|= 3g-3+b$. For each pants curve $\alpha_i$, choose a closed arc $w_i \subset \alpha_i$ called a \emph{window}. For each pair of pants $F$ in $P$, choose a collection of properly embedded arcs $\ell_{i,j}$ ($1 \leq i, j \leq 3$) in $F$ with endpoints contained in the windows as in Figure \ref{fig: Dehn coordiantes}. Let $(m_r)$ be non-negative integers and $(t_r)$ be integers for $1 \leq r \leq  |P|$ such that 

\begin{enumerate}
	\item[-] if $m_r = 0$ then $t_r \geq 0$; 
	\item[-] for each embedded pair of pants $F$, the sum $m_{i_1}+m_{i_2} +m_{i_3}$ corresponding to the three boundary components of $F$ is even; and
	\item[-] the number $m_r$ corresponding to a pants curve that bounds a torus minus a disc or a twice punctured disc is even. 
\end{enumerate} 

\begin{figure}
	
	\includegraphics*[width = 4 in]{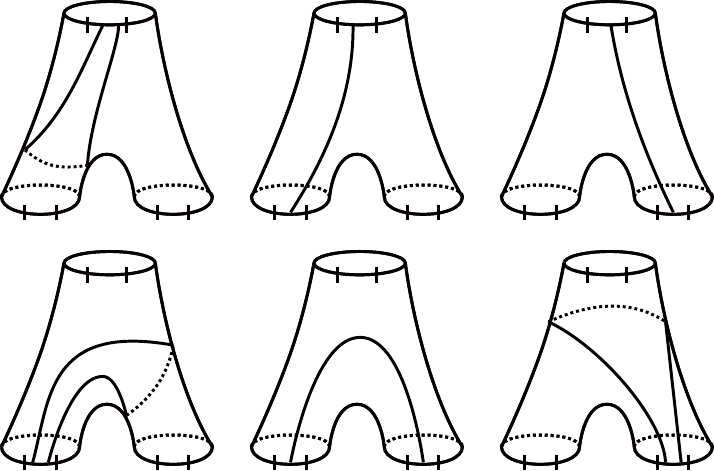}
	\caption{The windows $w_i$ and the arcs $\ell_{i,j}$. }
	\label{fig: Dehn coordiantes}
	
\end{figure}

Then we can construct a multicurve $\gamma$ as follows: take a number of parallel copies of the arcs $\ell_{i,j}$ such that the intersection number with $\alpha_r$ is exactly $m_r$. These parallel arcs can be naturally pasted together to produce a multicurve, and this corresponds to the case of $t_r = 0$ and $m_r \neq 0$. If $m_r, t_r \neq 0$, then we paste these parallel arcs on the two sides of $\alpha_r$ by a $(\frac{t_t}{m_r})$-fractional Dehn twist, where the twist is right-handed if $t_r >0$. If $t_r \neq 0$ and $m_r = 0$ then we add $t_r$ parallel copies of $\alpha_r$ to the multicurve. The numbers $m_r$ are called the \emph{intersection numbers}, and $t_r$ are called the \emph{twisting numbers}. The isotopy class of every multicurve, with all components essential, can be uniquely represented in this way. This is called the \emph{Dehn coordinates} of the multicurve. 

In our applications, we only need the twisting numbers $t_r$ modulo the integers $m_r$, since a full twist (that is a Dehn twist) about $\alpha_r$ lies in the mapping class group of $S - P$.
We say that a pants decomposition $P'$ is given in normal form with respect to a pants decomposition $P$ if the coordinates $(m_r, t_r \hspace{2mm} \mathrm{mod} \hspace{2mm} m_r)$ of $P'$ with respect to $P$ are given ($1 \leq r \leq |P|$). 

\subsection{Normal form for triangulations with common vertices }

\begin{definition}[Normal form for a triangulation]
	Let $\mathcal{T}$ and $\mathcal{T}'$ be triangulations of a compact surface $S$ with the same set of vertices. We say that $\mathcal{T}'$ is in \emph{normal form with respect to $\mathcal{T}$} if for each triangle $\Delta$ of $\mathcal{T}$, the intersection $\gamma \cap \Delta$ consist of a union of arcs of the following types 
	\begin{enumerate}
		\item[-] an arc connecting different sides of $\Delta$; or 
		\item[-] an arc connecting a common vertex of $\Delta$ and $\mathcal{T}'$ to its opposite side; or 
		\item[-] a side of $\Delta$ whose endpoints are vertices of $\mathcal{T}'$.
	\end{enumerate}
	\label{def:normal form triangulation}
\end{definition}
See Figure \ref{fig:normal form for triangulations} left for some examples of arcs that are normal, and Figure \ref{fig:normal form for triangulations} right for some non-examples. 

\begin{figure}
	\labellist
	\pinlabel $A$ at -8  0
	\pinlabel $B$ at 45 90
	\pinlabel $B'$ at 185 90
	\pinlabel $C$ at 112 0
	\pinlabel $A'$ at 135 0
	\pinlabel $C'$ at 257 0
	\pinlabel $D$ at 80 60
	\pinlabel $F$ at 95 35
	\pinlabel $E$ at 60 10
	\pinlabel $D'$ at 170 60
	\pinlabel $E'$ at 190 -8
	\pinlabel $F'$ at 220 -8
	\endlabellist
	\includegraphics[width = 3 in]{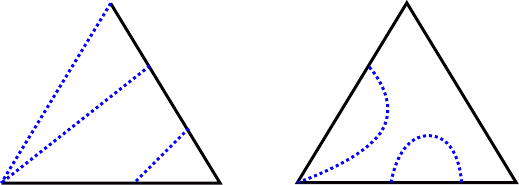}
	\caption{Here $ABC$ and $A'B'C'$ show two triangles of $\mathcal{T}$. Left: The three (blue dashed) arcs $AB$, $AD$, and $EF$ of $\mathcal{T}'$ are normal with respect to $\mathcal{T}$. Right: The two (blue dashed) arcs $A'D'$ and $E'F'$ of $\mathcal{T'}$ are not normal with respect to $\mathcal{T}$.}
	\label{fig:normal form for triangulations}
\end{figure}

\begin{prop}
	Let $\mathcal{T}$ and $\mathcal{T}'$ be triangulations of a compact surface $S$ with the same set of vertices. Assume that every edge of $\mathcal{T}'$ with the same endpoints is homotopically essential. Then $\mathcal{T}'$ can be isotoped relative to its vertices to be in normal form with respect to $\mathcal{T}$.
	\label{normal form exists}
\end{prop}

\begin{proof}
	First, isotope $\mathcal{T}'$ relative to its vertices to be in general position with respect to $\mathcal{T}$. Isotope $\mathcal{T}'$ relative to its vertices and remove any bigons with one side in $\mathcal{T}$ and one side in $\mathcal{T}'$. This reduces the weight $|\mathcal{T}' \cap \mathcal{T}|$ until no bigons are left. At this point, for any triangle $\Delta$ of $\mathcal{T}$, any component of $\mathcal{T}' \cap \Delta $ is a normal arc or a simple closed curve lying entirely in $\Delta$. The latter possibility is ruled out by the assumption that every edge of $\mathcal{T}'$ with the same endpoints is homotopically essential.
\end{proof}

\subsection{Normal form for a 1-complex with respect to a polygonal decomposition} 


\begin{definition}[Polygonal decomposition]
	A \emph{polygonal decomposition} $\mathcal{D}$ of a compact surface $S$ is an embedded connected finite $1$-complex such that 
	\begin{enumerate}
		\item[-] each complementary region adjacent to a boundary component of $S$ is topologically either a disc or an annular neighbourhood of the boundary component; and
		\item[-] every other complementary region is a topological disc.
	\end{enumerate} 
	\label{def: polygonal decomposition}
\end{definition}

The next definition is a generalisation of the usual notion of normal form for multicurves.

\begin{definition}[Normal form for a 1-complex with respect to a polygonal decomposition]
	Let $\mathcal{T}$ be a polygonal decomposition of a compact surface $S$, and $\gamma$ be a finite 1-complex embedded in $S$. We say that $\gamma$ is in \emph{normal form with respect to $\mathcal{T}$} if 
	\begin{enumerate}
		\item $\gamma$ has no homotopically trivial simple closed curve component that lies in a 2-cell of $\mathcal{T}$; and 
		\item there are no bigons between $\mathcal{T}$ and $\gamma$ except possibly those bigons that were already present in $\gamma$; i.e. there is no embedded disc $D \subset S$ whose interior is disjoint from $\gamma \cup \mathcal{T}$ and with $\partial D = \alpha \cup \beta$ where $\alpha$ and $\beta$ are arcs with $\partial \alpha = \partial \beta$, and $\alpha$ lies in an edge of $\mathcal{T}$ but not in an edge of $\gamma$, and $\beta$ lies in an edge of $\gamma$.
	\end{enumerate}
	\label{def:normal form1-complex}
\end{definition}

\begin{figure}
	\includegraphics[width = 3 in]{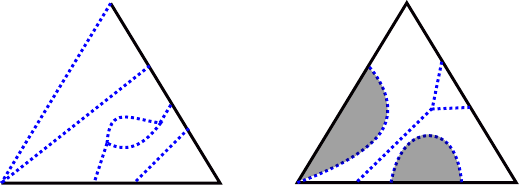}
	\caption{Here the two triangles belong to the polygonal decomposition $\mathcal{T}$, and the blue dashed 1-complex is the intersection of the 1-complex $\gamma$ with the corresponding triangles. On the left the restriction of $\gamma$ to the triangle is in normal form with respect to $\mathcal{T}$. On the right $\gamma$ in not in normal form with respect to $\mathcal{T}$ because of either of the two shaded bigons between $\mathcal{T} $ and $\gamma$.  }
	\label{fig:normal form for 1-complexes}
\end{figure}

See Figure \ref{fig:normal form for 1-complexes} left for an example of a 1-complex in normal form with respect to a polygonal decomposition, and Figure \ref{fig:normal form for 1-complexes} right for a non-example. 

\begin{prop}
	Let $\mathcal{T}$ be a polygonal decomposition of a compact surface $S$, and $\gamma$ be a finite 1-complex embedded in $S$. Then $\gamma$ can be isotoped relative to its vertices to be in normal form with respect to $\mathcal{T}$.
	\label{normal form exists}
\end{prop}

\begin{proof}
	Denote the $i$-skeleton of $\mathcal{T}$ and $\gamma$ by respectively $\mathcal{T}^i$ and $\gamma^i$ for $i=0,1$. First, isotope $\gamma$ relative to its vertices to be in general position with respect to $\mathcal{T}$; i.e. such that $\gamma^1 - \gamma^0$ is transverse to $\mathcal{T}^1 - \mathcal{T}^0$. All isotopes we consider will be relative to the vertices of $\gamma$. By a bigon we mean a bigon between $\mathcal{T}$ and $\gamma$. Denote by $i(\gamma, \mathcal{T})$ the number of transverse intersections of $\gamma^1 - \gamma^0$ with $\mathcal{T}^1 - \mathcal{T}^0$. If $B$ is a bigon such that at least one vertex of $B$ is not in $\gamma^0$, then we can isotope $\gamma$ relative to its vertices and remove $B$ while reducing the intersection number $i(\gamma, \mathcal{T})$. Repeat until every bigon has both vertices in $\gamma^0$. For any remaining bigon $B$ with $\partial B = \alpha \cup \beta$, if one side of $\partial B$, say $\alpha$, lies in an edge of $\mathcal{T}$ but not in an edge of $\gamma$, then isotope the other side $\beta$ of $\partial B$ into $\alpha$, and increase the number of edges of $\gamma$ that lie in an edge of $\mathcal{T}$. Repeating this process puts $\gamma$ in normal form with respect to $\mathcal{T}$.
\end{proof}

\begin{definition}[Normal coordinates of a 1-complex with respect to a polygonal decomposition]
Let $\mathcal{T}$ be a polygonal decomposition of a compact surface $S$, and let $\gamma$ be a 1-complex that is in normal form (Definition \ref{def:normal form1-complex}) with respect to $\mathcal{T}$. Then $\mathcal{T}$ cuts $\gamma$ into connected components, each of which is either  
\begin{enumerate}
	\item a graph with at least one vertex that is embedded in a closed 2-cell of $\mathcal{T}$; or
	\item an arc properly embedded in a 2-cell of $\mathcal{T}$ whose endpoints lie on distinct edges of the 2-cell. 
\end{enumerate}

Let $\Delta$ be a 2-cell of $\mathcal{T}$ and $a,b$ be two distinct sides of $\Delta$. Given two disjoint simple properly embedded arcs $\gamma_1, \gamma_2 \subset \Delta$ as in (2) above and joining $a$ and $b$, there is a rectangle with two opposite sides $\gamma_1$ and $\gamma_2$ and the other two sides lying on $a$ and $b$. We say that $\gamma_1$ and $\gamma_2$ are of the same \emph{type} if any other piece of $\gamma$ in $R$ is a properly embedded arc joining $a$ and $b$ as well. We can specify $\mathcal{T}$ by specifying a weighted 1-complex $\mathcal{N}$ in each 2-cell of $\mathcal{T}$ where  
\begin{enumerate}
	\item[-] $\mathcal{N}$ contains all components as in (1) above, and
	\item[-] for each arc type appearing in (2), $\mathcal{N}$ contains one copy of the arc together with a positive integer weight that counts the number of arcs of that type. 
\end{enumerate}

The data of the weighted 1-complex $\mathcal{N}$ over all 2-cells of $\mathcal{T}$ is called the \emph{normal coordinates of $\gamma$ with respect to $\mathcal{T}$}.
\label{def: normal coordinates for 1-complex}
\end{definition}

Note that if $\gamma$ is a multicurve, then only components as in (2) above (i.e. arcs) appear, and we recover the usual notion of normal coordinates for a multicurve.

\begin{example}
	In Figure \ref{fig:normal coordinates for a 1-complex} left, the intersection of $\gamma$ with a triangle of $\mathcal{T}$ is shown, which can be seen to be in normal form. On the right the normal coordinates of $\gamma$ with respect to $\mathcal{T}$ are shown.
\end{example}

\begin{figure}
	\labellist
	\pinlabel $2$ at 175 10
	\pinlabel $2$ at 220 10
	\endlabellist
	
	\includegraphics[width= 3 in]{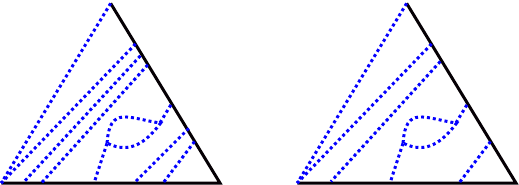}
	\caption{Left: the triangle belongs to the polygonal decomposition $\mathcal{T}$, and the dashed blue lines are the intersection of the 1-complex $\gamma$ with $\mathcal{T}$. Right: the normal coordinates of $\gamma$ with respect to $\mathcal{T}$ are shown.}
	\label{fig:normal coordinates for a 1-complex}
\end{figure}

\section{From pants decompositions to pre-triangulations, and vice versa}
\label{sec: pre-triangulations to pants}
\subsection{From pants decompositions to pre-triangulations}
\label{sec:pants decomposition-to-triangulation}

\begin{definition}[Bipartite adjacency graph of a pants decomposition]
Let $P$ be a pants decomposition of a compact connected orientable surface $S$. Define the \emph{bipartite adjacency graph} $\Gamma = \Gamma(P)$ with vertices $V = B \cup W$ partitioned into black $B$ and white $W$ colours as follows: black vertices correspond to pairs of pants in $S \setminus P$, white vertices correspond to simple closed curves in $P$, and each black vertex is connected to the (not necessarily distinct) white vertices corresponding to the essential boundary components of the pair of pants. 
\label{def:adjacency graph}
\end{definition}

\begin{definition}[Projection of a path to the adjacency graph]
	Let $S$, $P$, and $\Gamma = \Gamma(P)$ be as above. Let $\gamma \colon [0,1] \rightarrow S$ be a smooth path such that 
	\begin{enumerate}
		\item $\gamma$ is transverse to $P$, and 
		\item $\gamma$ has no backtrack with respect to the pants curves in $P$, i.e. there is no embedded bigon $D$ in $S$ with $\partial D = \alpha \cup \beta$ where $\alpha$ is a subarc of $\gamma$ and $\beta \subset P$ is an arc. 
	\end{enumerate}
	The \emph{projection of $\gamma$ to $\Gamma$} is a path defined as follows. Each component $\delta$ of $\gamma \setminus (\gamma \cap P)$ is an arc lying in a pair of pants $P_\delta$ such that one or both of its endpoints lie on $\partial P_\delta$. If both endpoints of $\delta$ lie on $\partial P_\delta$, then the projection of $\delta$ is the path of length 2 in $\Gamma$ joining the white vertices corresponding to $\partial P_\delta$. If $\delta$ is an arc with one endpoint on $\partial P_\delta$ and one endpoint in the interior of $P_\delta$ then the projection of $\delta$ is defined as the path of length 1 joining the the white vertex corresponding to $\delta \cap \partial P_\delta$ to the black vertex corresponding to $P_\delta$. Finally the projection of $\gamma$ is obtained by concatenating the projections of components $\delta$ of $\gamma \setminus (\gamma \cap P)$. Note that the projection of $\gamma$ is a path in the graph $\Gamma$, which does not depend on the parametrisation of $\gamma$. 
\end{definition}

\begin{definition}[A nerve for a pants decomposition]
Let $S$, $P$, and $\Gamma = \Gamma(P)$ be as above. Pick a base point $b \in S \setminus P$, and let $b_0$ be the black vertex of $\Gamma$ corresponding to the pair of pants containing $b$. Let $T$ be a maximal tree (also called spanning tree) for $\Gamma$. Denote the white vertices of $\Gamma$ by $w_i$ where $1 \leq i \leq m$, and let $\alpha_i \in P$ be the simple closed curve corresponding to $w_i$. For every vertex $w_i$, let $p_i$ be the unique shortest path in $T$ from $b_0$ to $w_i$. 

A \emph{nerve of $(P,b)$ compatible with $T$} is a collection $\gamma= \cup \gamma_i$ of simple arcs $\gamma_i \subset S$ such that
\begin{enumerate}
	\item $\gamma_i$ starts at $b$ and ends on $\alpha_i$; 
	\item $\gamma_i$ is transverse to $P$ and its projection to $\Gamma$ is the path $p_i$; and 
	\item  if the last white vertex of the path $p_i \cap p_j$ is $w_\ell$, then $\gamma_i \cap \gamma_j = \gamma_\ell$. In the special case that $p_i \cap p_j = \{ b_0 \}$, then $\gamma_i \cap \gamma_j = \{b \}$.
\end{enumerate}
A \emph{nerve of $(P,b)$} is a nerve of $(P,b)$ compatible with some maximal tree $T \subset \Gamma(P)$.
\label{def:nerve}
\end{definition}

\begin{example}
	Let $P$ be the pants decomposition of a surface of genus two as in Figure \ref{fig:adjacency-graph-pants-decomposition} top-left. Then the adjacency graph $\Gamma(P)$ is as in the top-right of the same figure. Let $b$ be a base point lying in the interior of the left pair of pants, and $T$ be the maximal tree as in bottom-right hand side of the figure. Then a nerve for $(P,b)$ compatible with $T$ is shown in the bottom-left side of the figure. 
\end{example}

\begin{figure}
	\labellist
	\pinlabel $b$ at 20 78
	\pinlabel $b_0$ at 305 40
	\pinlabel $\alpha_1$ at 55 140
	\pinlabel $\alpha_3$ at 200 140
	\pinlabel $\alpha_2$ at 145 145
	\pinlabel $w_1$ at 275 50
	\pinlabel $w_2$ at 325 58
	\pinlabel $w_3$ at 380 50
	\pinlabel $\gamma_1$ at 37 75
	\pinlabel $\gamma_2$ at 80 74
	\pinlabel $\gamma_3$ at 150 65
	\endlabellist
	\includegraphics[width = 4 in]{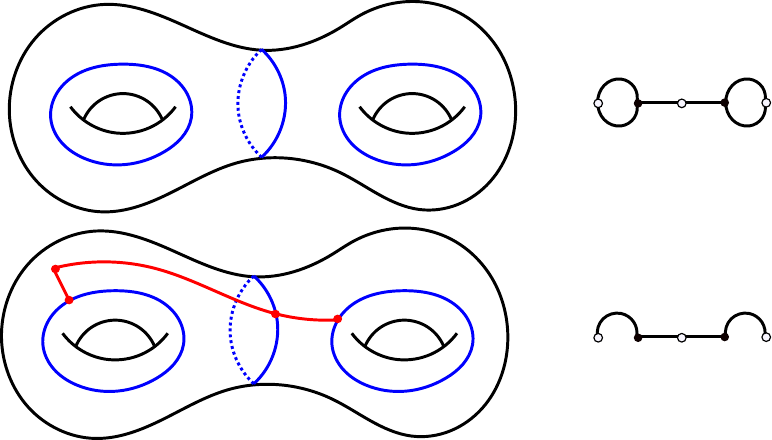}
	\caption{A pants decomposition $P$ (top-left), its adjacency graph (top-right), a maximal tree $T$ for the adjacency graph (bottom-right), a base point $b$ and a nerve for $(P,b)$ compatible with $T$ (bottom-left). }
	\label{fig:adjacency-graph-pants-decomposition}
\end{figure}

\begin{lem}
Given $S$, $P$, $T$, and $b$ as above, there is a nerve for $(P, b)$ compatible with $T$.
\end{lem}

\begin{proof}
	Let $d_i$ be the length (number of edges) of the path $p_i$. Define the arcs $\gamma_i$ inductively based on the length $d_i$. The base case is when $d_i =1$. There are either 1, 2, or 3 paths $p_i$ with $d_i=1$, and the corresponding curves $\alpha_i$ will be boundary components of the pair of pants $P_0$ containing $b_0$. In this case we take $\gamma_i$ to be simple arcs joining $b$ to the corresponding curves in $\partial P_0$ that intersect only in $b$. Now assume that $\gamma_i$ is not defined yet, and $d_i$ is minimal with this property. Let $w_j$ be the white vertex that is distance 2 from $w_i$ along $p_i$. Let $b_k$ be the common neighbour of $w_i$ and $w_j$ in $p_i$, and $w_r$ ($r \neq i , j$) be the third neighbour of $b_k$. Then $\gamma_j$ is defined by hypothesis. Note that $\gamma_j \subset \gamma_i$ should hold. Similarly if $p_j \subset p_r$, then $\gamma_j \subset \gamma_r$ should hold. Now if $p_j \subset p_r$ then we define $\gamma_i \setminus \gamma_j$ and $\gamma_r \setminus \gamma_j$ simultaneously such that $\gamma_i \cap \gamma_r = \gamma_j$ as in Figure \ref{fig:nerve} right, and otherwise define $\gamma_i \setminus \gamma_j$ as in Figure \ref{fig:nerve} left. 
\end{proof}

\begin{figure}
	\labellist
	\pinlabel $\gamma_j$ at 89 25
	\pinlabel $\gamma_i$ at 86 120
	\pinlabel $\gamma_j$ at 320 25
	\pinlabel $\gamma_i$ at 320 120
	\pinlabel $\gamma_r$ at 353 120
	\endlabellist

	\includegraphics[width = 3 in]{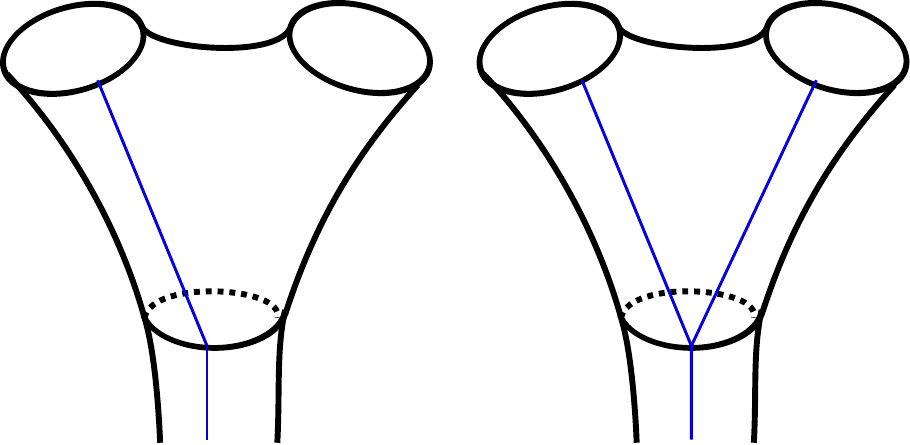}
	\caption{Constructing a nerve}
	\label{fig:nerve}
\end{figure}

\begin{remark}
	Assume that a triangulation $\mathcal{H}$ of $S$ is given such that $P$ is contained in the 1-skeleton of $\mathcal{H}$. Then we can construct a nerve $\gamma$ for $(P,b)$ compatible with $T$ such that each arc $\gamma_i$ passes through each face of $\mathcal{H}$ at most once, and it is disjoint from the vertices of the triangulation. To see this let $q_i$ be a shortest path in the dual graph of $\mathcal{H}$ connecting $\gamma_j \cap \alpha_j$ to $\alpha_i$ in the above proof. Now if $p_j \nsubseteq p_r$, then define $\gamma_i \setminus \gamma_j$ such that its projection to the dual graph of $\mathcal{H}$ is equal to $q_i$. In particular, $\gamma_i \setminus \gamma_j$ passes through each face of $\mathcal{H}$ at most once.  If $p_j \subset p_r$, define $q_r$ similar to $p_i$. Then after possibly replacing $q_r$ by another path of the same length and the same endpoints, $q_i$ and $q_r$ intersect each other in a connected interval; this follows from the shortest length hypothesis. Hence we can define $\gamma_i \setminus \gamma_j$ and $\gamma_r \setminus \gamma_j$ simultaneously such that their projections to the dual graph of $\mathcal{H}$ gives respectively $q_i$ and $q_r$, and they only intersect in a single point.
	\label{nerve intersect triangulation}  
\end{remark}

A pants decomposition is a pre-triangulation as well. The next construction shows how to produce a \emph{connected} pre-triangulation from a pants decomposition equipped with a nerve.

\begin{construction}[From a pants decomposition to a connected pre-triangulation]
Let $S$ be a compact connected orientable surface, and $P$ be a pants decomposition of $S$. Let $T$ be a maximal tree in the bipartite adjacency graph $\Gamma= \Gamma(P)$, and $b \in S \setminus P$ be a base point. Pick a nerve $\gamma$ for $(P,b)$ compatible with $T$. We construct a connected pre-triangulation $\mathcal{T} = \mathrm{\textbf{T}}(P, \gamma)$ of $S$. Assume the notations of Definitions \ref{def:adjacency graph} and \ref{def:nerve}. The only vertex of $\mathcal{T}$ is $b$. Moreover, $\mathcal{T}$ has an edge $e_i$ for each curve $\alpha_i$ in the pants decomposition, constructed inductively as follows
\begin{enumerate}
	\item First, for every white vertex $w_i$ adjacent to $b_0$, add an edge $e_i$ that runs from $b$ to $\alpha_i$ and remains close and almost parallel to $\gamma_i$, then goes around $\alpha_i$, and finally comes back to $b$ close and almost parallel to $\gamma_i$. 
	
	\item From now on at each step one or two edges are added. Pick a white vertex $w_i$ for which $\gamma_i$ is not defined yet and such that the distance between $w_i$ and $b_0$ in $T$ is minimal. Consider the shortest path $p_i$ in $T$ joining $b_0$ to $w_i$, and let $w_j $ be the white vertex along $p_i$ that is closest to $w_i$ ($j \neq i$). By hypothesis, the edge $e_j$ is already defined. Let $b_k$ be the common neighbour of $w_i$ and $w_j$ in $p_i$, and let $w_r$ ($r \neq i , j$) be the other neighbour of $b_k$ in $\Gamma(P)$. If the edge joining $b_k$ to $w_r$ appears in $p_r$, then add $e_i$ and $e_r$ as in Figure \ref{fig:pants-to-triangulation} Right; otherwise add $e_i$ as in Figure \ref{fig:pants-to-triangulation} Left. Again $e_i$ runs from $b$ to $\alpha_i$ remaining almost parallel to $\gamma_i$, then goes around $\alpha_i$, and finally comes back to $b$ almost parallel to $\gamma_i$
	
\end{enumerate}

Intuitively, $e_i$ is obtained by pulling $\alpha_i$ to $b$ along the path $\gamma_i$; the inductive process for constructing $e_i$ is to make sure that the edges $e_i$ only intersect each other at the common vertex $b$.

\label{const:pants-to-triangulation}
\end{construction}

\begin{figure}
\labellist
\pinlabel $\alpha_j$ at 50 50
\pinlabel $\alpha_j$ at 280 50 
\pinlabel $\alpha_i$ at 15 220
\pinlabel $\alpha_i$ at 250 220
\pinlabel $\alpha_r$ at 185 220
\pinlabel $\alpha_r$ at 420 220
\pinlabel $e_i$ at 80 100
\pinlabel $e_i$ at 305 100
\pinlabel $e_r$ at 360 100 
\pinlabel $e_j$ at 120 10
\pinlabel $e_j$ at 350 10

\endlabellist	
	
\includegraphics[width=3 in]{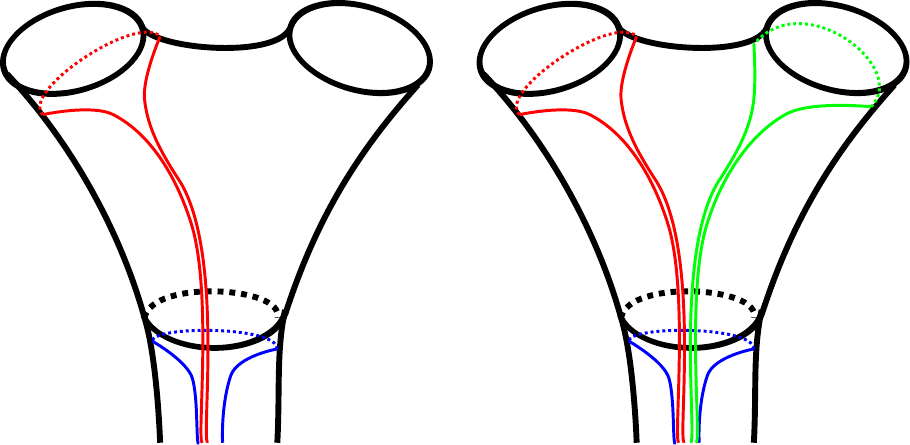}
\caption{From a pants decomposition to a connected pre-triangulation}
\label{fig:pants-to-triangulation}
\end{figure}

\begin{lem}
$\mathrm{\textbf{T}}(P, \gamma)$ is a pre-triangulation of $S$.
\end{lem}

\begin{proof}
	The nerve $\gamma$ is an embedded tree in $S$, and there is a collapsing map $r \colon S \rightarrow S$ that collapses $\gamma$ to the base point $b$, and the image of $P \cup \gamma$ under $r$ is the 1-complex $\mathrm{\textbf{T}}(P, \gamma)$. Each complementary component of $\mathrm{\textbf{T}}(P, \gamma)$ in $S$ is homeomorphic to a corresponding complementary component of $P \cup \gamma$ in $S$. Since $P \cup \gamma$ is a pre-triangulation, so is $\mathrm{\textbf{T}}(P, \gamma)$. 
	
\end{proof}

\begin{example}
	Let $P$ be the pants decomposition as in Figure \ref{fig:adjacency-graph-pants-decomposition} top-left, and $\gamma$ be the nerve as in Figure \ref{fig:adjacency-graph-pants-decomposition} bottom-left. Then the pre-triangulation $\mathrm{\textbf{T}}(P, \gamma)$ is shown in Figure \ref{fig:pants-decomposition-from-nerve}.
\end{example}

\begin{figure}
	\includegraphics[width= 2 in]{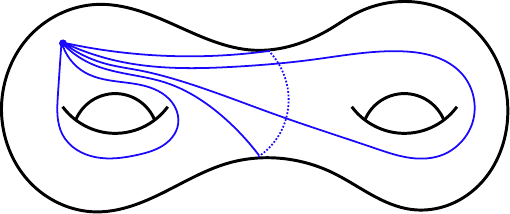}
	\caption{The pants decomposition associated with the nerve shown in Figure \ref{fig:adjacency-graph-pants-decomposition} bottom-left.}
	\label{fig:pants-decomposition-from-nerve}
\end{figure}

\subsection{From pre-triangulations to pants decompositions}
\label{sec:triangulation-to-pants decomposition}

\begin{construction}[From a pre-triangulation to a pants decomposition]
	Let $\mathcal{T}$ be a pre-triangulation of a compact connected orientable surface $S$. Pick a forest $T$ in the $1$-skeleton $\mathcal{T}^1$ of $\mathcal{T}$ such that each component of $T$ has at most one point of intersection with $ \partial S$, and choose an ordering $\mathrm{O}$ on the edges of $\mathcal{T} \setminus (T \cup \partial S)$ as $e_1, \cdots, e_\ell$. First define the subsurfaces $N_i$ inductively: $N_1$ is a regular neighbourhood of $T \cup \partial S \cup e_1$
	 in $S$, and $N_{i+1}$ is a regular neighbourhood of $N_{i} \cup e_{i+1}$ for each $1 \leq i < \ell$. Let $F_i$ be the subsurface obtained from $N_i$ by first capping off any boundary component that bounds a disc in $S \setminus N_i$ and then deleting any disc components. In particular, $F_i$ is isotopic to the subsurface that $e_1 \cup \cdots \cup e_i$ \emph{fills} minus disc components. Define the multicurve $P$ as the union of $\partial F_i$ for $1 \leq i \leq \ell$ after discarding any boundary parallel or repeated simple closed curves. We will show in Lemma \ref{lem: pre-triangulation-to-pants-decomposition} that $P =\mathrm{\textbf{P}}(\mathcal{T}, T , \mathrm{O})$ is a pants decomposition. If $T = \emptyset$, we abbreviate $\mathrm{\textbf{P}}(\mathcal{T}, T, \mathrm{O})$ to $\mathrm{\textbf{P}}(\mathcal{T},\mathrm{O})$.
	\label{const:triangulation-to-pants}
\end{construction}

\begin{example}
	Let $\mathcal{T}$ be the pre-triangulation shown in Figure \ref{fig:pants-decomposition-from-pre-triangulation}, and $\mathrm{O} = \{ e_1, e_2, e_3 \}$ be the ordering on its edges. Then the associated pants decomposition $\mathrm{\textbf{P}}(\mathcal{T} , \mathrm{O})$ is the one in Figure \ref{fig:adjacency-graph-pants-decomposition} top-left. 
\end{example}

\begin{figure}
	\labellist 
	\pinlabel $e_1$ at 50 67
	\pinlabel $e_2$ at 147 67
	\pinlabel $e_3$ at 200 67
	\endlabellist
	
	\includegraphics[width =2 in]{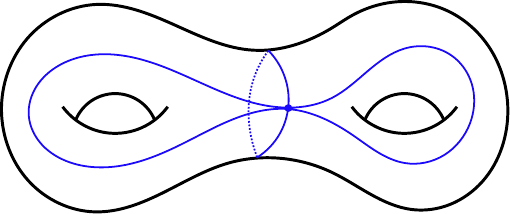}
	\caption{A pre-triangulation $\mathcal{T}$ with an ordering $\mathrm{O} = \{ e_1, e_2, e_3 \}$ on its edges. The associated pants decomposition $\mathrm{\textbf{P}}(\mathcal{T}, \mathrm{O})$ is as in Figure \ref{fig:adjacency-graph-pants-decomposition} top-left.} 
	\label{fig:pants-decomposition-from-pre-triangulation}
\end{figure}


A connected embedded subsurface $F \subset S$ is called \emph{essential} if it is $\pi_1$-injective. An embedded subsurface is \emph{essential} if each of its connected components is essential. When $F \subset S$ is an annulus or a pair of pants, then it is easy to see that $F$ is essential if and only if each curve in $\partial F$ is 
$\pi_1$-injective
in $S$.

\begin{lem}
	Each surface $F_{j}$ is obtained from $F_{j-1}$ by thickening the boundary components to the outside and then either attaching a $1$-handle, filling in a complementary annulus, or taking a disjoint union with at most one essential annulus ($2 \leq j \leq \ell$). In particular $F_j - F_{j-1}$ is a disjoint union of essential annuli and at most one essential pair of pants, and
	\[ \chi(F_j - F_{j-1}) =  \chi(F_j) - \chi(F_{j-1}) \geq -1. \]
	\label{lem:attaching a band}
\end{lem}

\begin{proof}
	The surface $F_{j-1}$ is obtained by taking a regular neighbourhood $N_{j-1}$ of $\partial S \cup T \cup e_1 \cup \cdots \cup e_{j-1}$, capping off any complementary discs, and then discarding any disc components. Consider three cases: 
	\begin{enumerate}
		\item The edge $e_j$ is a loop disjoint from $\partial S \cup T \cup e_1\cup \cdots \cup e_{j-1}$. In this case $N_j$ is the disjoint union of a regular neighbourhood of $N_{j-1}$ and an annulus. If the core curve of the annulus is essential then $F_j$ is the disjoint union of a regular neighbourhood of $F_{j-1}$ and an essential annulus, and otherwise $F_j$ is equal to a regular neighbourhood of $F_{j-1}$. Therefore $F_j - F_{j-1}$ is a disjoint union of essential annuli. 
		
		\item The edge $e_j$ has only one endpoint on $\partial S \cup T \cup e_1 \cup \cdots \cup e_{j-1}$. In this case $N_j$ is a regular neighbourhood of $N_{j-1}$, and $F_j$ is a regular neighbourhood of $F_{j-1}$. 
		
		\item The edge $e_j$ has both endpoints on $\partial S \cup T \cup e_1 \cup \cdots \cup e_{j-1}$. Then $N_j$ is obtained by taking a regular neighbourhood of $N_{j-1}$ and then attaching a band. In particular, $N_j - N_{j-1}$ is a union of annuli and one pair of pants. The pair of pants could be attached to $N_{j-1}$ along either one or two of its boundary components, and $N_j$ will have respectively either two or one new boundary components compared to $N_{j-1}$. Denote the boundary components of the attached pair of pants $R$ by $\{ \alpha, \beta , \gamma \}$. Now the lemma follows by considering various cases depending on the subset of $\{ \alpha, \beta, \gamma\}$ along which $R$ is attached to $N_{j-1}$, and the subset of $\{ \alpha, \beta,  \gamma \}$ that are homotopically trivial in $S$. For example if $R$ is attached to $N_{j-1}$ along $\alpha$ and $\beta$, and furthermore $\alpha$ and $\beta$ are homotopically essential but $\gamma$ is homotopically trivial, then the disc $D$ bounding $\gamma$ cannot contain $\alpha$ and $\beta$ in its interior, and in this case $F_j$ is obtained from $F_{j-1}$ by attaching the annulus $R \cup D$. The remaining cases are similar and we leave them to the reader. 
	\end{enumerate}
\end{proof}

\begin{lem}
$P = \mathrm{\textbf{P}}(\mathcal{T}, T, \mathrm{O})$ is a pants decomposition.


\label{lem: pre-triangulation-to-pants-decomposition}	
\end{lem}

\begin{proof}
	Assume that $\alpha$ is an essential simple closed curve in $S \setminus P$. We need to show that $\alpha$ is isotopic to a curve in $P$. Since $\mathcal{T}$ is a pre-triangulation, $F_\ell$ is the entire surface $S$ minus a (possibly empty) union of disjoint essential annuli and pairs of pants. Setting $F_0 = \emptyset$ and $F_{\ell +1} = S$, the surface $S$ can be written as the disjoint union of $F_{i+1}-F_i$ for $0 \leq i \leq \ell $. Since $\alpha$ can be isotoped to be disjoint from $P = \cup_{i=1}^{\ell} \partial F_i$, there is an index $i$ such that $\alpha \subset F_{i+1}-F_i$. However, by Lemma \ref{lem: pre-triangulation-to-pants-decomposition} every essential simple closed curve in $F_{i+1} - F_i$ is parallel to a component of $\partial F_{i+1}$ or $\partial F_i$ where $1 \leq i < \ell$. Moreover, since we assume that every component of $T$ has at most one point of intersection with $\partial S$, it follows that $F_1$ is a union of annuli and at most one pair of pants and $\partial S \subset \partial F_1$. This implies that $\alpha$ is isotopic to a curve in $P = \cup \partial F_i $. 


\end{proof}

\begin{remark}
	Let $\mathrm{O}= (e_1, \cdots, e_\ell)$ be an ordering on the edges of $\mathcal{T} \setminus T$. Let $e'_1, \cdots, e'_m$ be any ordering on the edges of $T$. Define the ordering $\mathrm{O}'$ on the edges of $\mathcal{T}$ as $\mathrm{O}' = (e'_1, \cdots, e'_m, e_1, \cdots, e_\ell)$. Then the pants decompositions $\mathrm{\textbf{P}}(\mathcal{T}, T, \mathrm{O})$ and $\mathrm{\textbf{P}}(\mathcal{T}, \emptyset, \mathrm{O})$ are equal. 
\end{remark}

\begin{lem}[Naturality]
Let $S$ be a compact connected orientable surface, and $P$ be a pants decomposition of $S$. Let $b \in S \setminus P$ be a base point, and $\gamma$ be a nerve for $(P,b)$. Let $\mathcal{T}$ be a pre-triangulation obtained by possibly adding extra edges to the pre-triangulation $\mathrm{\textbf{T}}(P, \gamma)$. There is an ordering $\mathrm{O}$ on the edges of $\mathcal{T}$ such that $\mathrm{\textbf{P}}(\mathcal{T}, \mathrm{O})= P$.
	
	
\label{naturality}	
	
\end{lem}

\begin{proof}
	
	A pants decomposition of $S$ is a maximal collection of disjoint non-isotopic essential simple closed curves on $S$. Therefore it is enough to show that for a suitable choice of an ordering $\mathrm{O}$ we have $P \subset \mathrm{\textbf{P}}(\mathcal{T}, \mathrm{O}) $. Choose $\mathrm{O}$ such that the edges of $\mathrm{\textbf{T}}(P, \gamma)$ appear in the same order that they were constructed in the inductive recipe of Construction \ref{const:pants-to-triangulation}; if two (or three) edges were constructed simultaneously, the two (or three) edges appear consecutively in the ordering $\mathrm{O}$. Finally, any edge of $\mathcal{T} \setminus \mathrm{\textbf{T}}(P, \gamma)$ appears after the edges of $\mathrm{\textbf{T}}(P, \gamma)$ in an arbitrary fashion. One can inductively see that $P \subset \mathrm{\textbf{P}}(\mathcal{T}, \mathrm{O})$.

\end{proof}

\subsection{Modifying a pre-triangulation and its effect on the associated pants decomposition}

In this subsection, we consider certain operations on pre-triangulations and their effect on the associated pants decompositions.

\begin{lem}[Effect of subdividing a pre-triangulation on the associated pants decomposition]
	Let $\mathcal{T}$ be a pre-triangulation of a compact orientable surface $S$, and $\mathrm{O}$ be an ordering on the edges of $\mathcal{T}$. Consider any of the following two cases:
	
	\begin{enumerate}
		\item[a)] Let $\mathcal{T}'$ be a pre-triangulation obtained by subdividing an edge $e$ of $\mathcal{T}$. Define an ordering on the edges of $\mathcal{T}'$ as follows: if $e$ appears as the $i$th element in the ordering $\mathrm{O}=(e_1, e_2, \cdots, e_n)$ and $e_i^1, \cdots, e_i^k$ are the edges obtained by subdividing $e_i$, define $\mathrm{O}'= (e_1, \cdots, e_{i-1}, e_i^1, \cdots, e_i^k,e_{i+1}, \cdots, e_n)$. 
		
		\item[b)]  Let $R$ be a complementary region to $\mathcal{T}$, and $\mathcal{T}'$ be the pre-triangulation obtained by adding a vertex $v$ in $R$ and connecting some of the vertices of $R$ to $v$ by new edges. Define an ordering on the edges of $\mathcal{T}'$ where the edges of $\mathcal{T}$ appear first and the newly added edges appear at the end in any order. 
	\end{enumerate}
	
	Then the associated pants decompositions $\mathrm{\textbf{P}}(\mathcal{T}, \mathrm{O})$ and $\mathrm{\textbf{P}}(\mathcal{T}', \mathrm{O}')$ are equal.
	
	\label{subdividing}
\end{lem}

\begin{proof}
	Part a) is clear. To see Part b), let $O' = (e_1, \cdots, e_\ell, e'_1, \cdots, e'_m)$ where $e'_j$ are the newly added edges. Let $N_i$ and $F_i$ for $1 \leq i \leq \ell +m$ be the subsurfaces associated with $(\mathcal{T}', \mathrm{O}')$ as in Construction \ref{const:triangulation-to-pants}. Part b) follows from the following three facts: 
	
	\begin{enumerate}
	\item[i)] $\mathrm{\textbf{P}}(\mathcal{T}, \mathrm{O}) $ is obtained from $\cup_{i=1}^{\ell} \partial F_i$ by discarding repeated or boundary parallel curves. In particular the inclusion $\mathrm{\textbf{P}}(\mathcal{T}, \mathrm{O})  \subset \mathrm{\textbf{P}}(\mathcal{T}', \mathrm{O}')$ holds.
	\item[ii)]	$\partial R \subset \partial N_\ell$.
	\item[iii)]	Every essential simple closed curve in $R$ is parallel to a component of $\partial R$. In particular, for $i>\ell$, the subsurface $\partial F_i$ contains no essential curve that was not already present in $\partial F_\ell$. 
	\end{enumerate}	  
\end{proof}

\subsection{Changing the ordering  and its effect on the associated pants decomposition} 

\begin{definition}[Consecutive transposition]
	Let $\mathcal{T}$ be a pre-triangulation of a compact orientable surface, and $\mathrm{O}$ and $\mathrm{O}'$ be two orderings on the edges of $\mathcal{T}$. We say that $\mathrm{O}'$ is obtained from $\mathrm{O}$ by a \emph{consecutive transposition} if there are edges $e_1$ and $e_2$ of $\mathcal{T}$ that are consecutive in the ordering $\mathrm{O}$ and such that $\mathrm{O}'$ is obtained from $\mathrm{O}$ by swapping the order of $e_1$ and $e_2$. 
\end{definition}

\begin{lem}
	Let $\mathcal{T}$ be a pre-triangulation of a compact orientable surface $S$. Let $\mathrm{O}$ be an ordering on the edges of $\mathcal{T}$, and assume that $\mathrm{O}'$ is obtained from $\mathrm{O}$ by a consecutive transposition. The distance between the pants decompositions $\mathrm{\textbf{P}}(\mathcal{T}, \mathrm{O})$ and $\mathrm{\textbf{P}}(\mathcal{T}, \mathrm{O}')$ in the pants graph is $O(1)$.
	\label{transposition}
\end{lem}

\begin{proof}
	 Let $N_1, \cdots, N_n$ (respectively $N'_1, \cdots, N'_n$) be the sequence of surfaces associated with $(\mathcal{T}, \mathrm{O})$ (respectively $(\mathcal{T}, \mathrm{O}')$) as in Construction \ref{const:triangulation-to-pants}. Define $F_1, \cdots, F_n$ and $F'_1, \cdots, F'_n$ similarly. Assume that $\mathrm{O}'$ is obtained from $\mathrm{O}$ by swapping the $j$-th and $(j+1)$-th element. Then 
	\begin{eqnarray*}
		F_k = F'_k \hspace{3mm} \text{if} \hspace{3mm} |k-j|\geq 1. 
	\end{eqnarray*}
	Hence the pants decompositions $\mathrm{\textbf{P}}(\mathcal{T}, \mathrm{O}')$ and $\mathrm{\textbf{P}}(\mathcal{T}, \mathrm{O})$ can differ only in curves that are supported on the subsurface $F_{j+1} \setminus F_{j-1}$. By Lemma \ref{lem:attaching a band}, after discarding any disc or annulus components, $F_{j+1} \setminus F_{j-1}$ is an essential subsurface that has Euler characteristic at least $-2$. Moreover, at most one component of $F_{j+1} \setminus F_{j-1}$ is not a pair of pants or annulus. Hence, after discarding all disc and annulus components, $F_{j+1} \setminus F_{j-1}$ is either an essential $4$-times punctured sphere, an essential twice punctured torus, an essential once-punctured torus, or a disjoint union of at most two essential pairs of pants. The restrictions of $\mathrm{\textbf{P}}(\mathcal{T}, \mathrm{O}')$ and $\mathrm{\textbf{P}}(\mathcal{T}, \mathrm{O})$ to the subsurface $F_{j+1} \setminus F_{j-1}$, after discarding boundary-parallel curves, form two pants decompositions $P$ and $P'$ of $F_{j+1} \setminus F_{j-1}$. Since $F_{j+1} \setminus F_{j-1}$ is an essential subsurface of $S$, and the restrictions of $\mathrm{\textbf{P}}(\mathcal{T}, \mathrm{O}')$ and $\mathrm{\textbf{P}}(\mathcal{T}, \mathrm{O})$ to the exterior of $F_{j+1} \setminus F_{j-1}$ are equal, the distance between $\mathrm{\textbf{P}}(\mathcal{T}, \mathrm{O}')$ and $\mathrm{\textbf{P}}(\mathcal{T}, \mathrm{O})$ in the pants graph of $S$ is bounded from above by the distance between $P$ and $P'$ in the pants graph of $F_{j+1} \setminus F_{j-1}$. We would like to give an upper bound for the distance between $P$ and $P'$ in the pants graph of $F_{j+1} \setminus F_{j-1}$. We claim that
	\begin{eqnarray*}
		i(\partial N_{j}, \partial N'_{j}) \leq 4.
	\end{eqnarray*}
	If at least one of $e_j$ and $e_{j+1}$ is a simple closed curve, or at least one end of $e_j$ or $e_{j+1}$ is disjoint from $\partial S \cup e_1 \cup \cdots \cup e_{j-1}$, then $\partial N_j$ and $\partial N'_j$ can  be isotoped to be disjoint. So assume that $e_j$ and $e_{j+1}$ have both ends on $\partial S \cup e_1\cup \cdots \cup e_{j-1}$. Choose regular neighbourhoods $N$ and $N'$ of $\partial S \cup e_1 \cup \cdots \cup e_{j-1}$  in $S$ such that $N \subset \text{int} N'$ and $N' \setminus N \cong \partial N \times [0,1]$. Take a band $B_j \cong [0,1] \times [0,1]$ around $e_j$ and connecting $\partial N$ to itself such that 
	\begin{enumerate}
		\item [-] $B_j \cap \partial N = [0,1] \times \{ 0 ,1 \}$; and
		\item[-] $B_j \cap (\overline{N' \setminus N})$ is identified with $[0,1] \times ([0,\frac{1}{4}]\cup[\frac{3}{4},1])$ in $B_j$ and with $(B_j \cap \partial N) \times [0,1]$ in $N' \setminus N$. 
	\end{enumerate}
	Choose a band $B_{j+1}$ around $e_{j+1}$, and disjoint from $B_j$, that connects $\partial N'$ to itself. Then $\partial N_j$ is isotopic to $\partial (N \cup B_j)$, and $\partial N_{j+1}$ is isotopic to $\partial (N' \cup B_{j+1})$. The only intersections of $\partial N_j$ and $\partial N_{j+1}$ come from the $4$ intersection points between $\partial N'$ and $\partial B_{j}$, proving the desired inequality. See Figure \ref{fig:bands}.
	\begin{figure}
		\labellist
		\pinlabel $B_j$ at 200 101 
		\pinlabel $B_{j+1}$ at 120 106
		\endlabellist
		
		\includegraphics[width = 2.5 in]{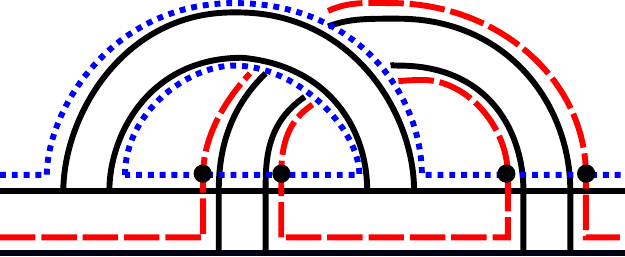}
		\caption{A schematic picture showing that the multicurves $\partial N_j$ and $\partial N_{j+1}$ (in red and blue respectively, and in different dashed lines) intersect in at most $4$ points (solid dots). Here the horizontal rectangle is $N' \setminus N$, and the area below this rectangle is $N$. There are other possibilities: the bands are allowed to have twists, and the endpoints of the bands might not interlace or they might lie on different components of $\partial N$.}
		\label{fig:bands}
	\end{figure}
	Hence 
	\begin{eqnarray*}
		i(\partial F_{j}, \partial F'_{j}) \leq i(\partial N_j , \partial N'_j) \leq 4.
	\end{eqnarray*} 
	Therefore, the pants decompositions $P$ and $P'$ have distance $O(1)$ in the pants graph of $F_{j+1} \setminus F_{j-1}$, proving the lemma. 
\end{proof}

\begin{cor}
	Let $\mathcal{T}$ be a pre-triangulation of a compact orientable surface $S$. Let $\mathrm{O}$ and $\mathrm{O}'$ be two orderings on the edges of $\mathcal{T}$. The distance between the pants decompositions $\mathrm{\textbf{P}}(\mathcal{T}, \mathrm{O})$ and $\mathrm{\textbf{P}}(\mathcal{T}, \mathrm{O}')$ in the pants graph is $O(n^2)$ where $n$ is the number of edges of $\mathcal{T}$. 
	\label{changing the ordering} 
\end{cor}

\begin{proof}
	If $S_n$ is the symmetric group on $n$ elements, then every permutation $\sigma \in S_n$ can be written as a product of at most 
	\begin{eqnarray*}
		1 + 2 + \cdots + (n-1) = O(n^2)
	\end{eqnarray*}
	\emph{consecutive transpositions} of the form $\sigma_j = (j,j+1)$, where $1 \leq j \leq n-1$. Here $\sigma_j$ swaps the $j$-th and $(j+1)$-th elements. This can be seen by induction on $n$: assume $\sigma(n)= i$ and note that $\sigma_{n-1} \circ \cdots \circ \sigma_{i+1} \circ \sigma_i \circ \sigma$ fixes $n$ and hence can be considered as an element of $S_{n-1}$.  
	
	We can identify the symmetric group $S_n$ with the group of permutations of the edges of $\mathcal{T}$ such that $\mathrm{O}$ corresponds to the identity. Hence there are orderings $\mathrm{O}_1, \cdots, \mathrm{O}_k$ on the edges of $\mathcal{T}$ such that
	
	\begin{enumerate}
		\item $\mathrm{O} = \mathrm{O}_1$, $\mathrm{O}'=\mathrm{O}_k$;
		\item each $\mathrm{O}_{i+1}$ is obtained from $\mathrm{O}_i$ by a single consecutive transposition; and
		\item $k = O(n^2)$. 
	\end{enumerate}
	
	By Lemma \ref{transposition}, for each $1 \leq i < k$ the pants decompositions $\mathrm{\textbf{P}}(\mathcal{T}, \mathrm{O}_{i+1})$ and $\mathrm{\textbf{P}}(\mathcal{T}, \mathrm{O}_i)$ have distance $O(1)$ in the pants graph. The result now follows from the triangle inequality.
\end{proof}

\subsection{Upper bound for distance between pants decompositions coming from pre-triangulations}

Given two pre-triangulations, and two orderings on their edges, the next lemma gives an upper bound for the distance between their associated pants decompositions.

\begin{lem}
	Let $\mathcal{T}$ and $\mathcal{T}'$ be pre-triangulations of a compact orientable surface $S$. Assume that for any edge $e$ of $\mathcal{T}$ and any edge $e'$ of $\mathcal{T}'$, the intersection $e \cap e'$ is a disjoint union of subintervals of $e$ and also of $e'$. Here we allow a subinterval to be a single point as well. Let $\mathrm{O}$ and $\mathrm{O}'$ be orderings on the edges of $\mathcal{T}$ and $\mathcal{T}'$. Then the pants decompositions $\mathrm{\textbf{P}}(\mathcal{T},\mathrm{O})$ and $\mathrm{\textbf{P}}(\mathcal{T}',\mathrm{O}')$ have distance at most $O(n^2),$ where $n$ is the number of edges of the pre-triangulation obtained by superimposing $\mathcal{T}$ and $\mathcal{T}'$. In particular, if $\mathcal{T}$ and $\mathcal{T}'$ are transverse to each other, and $|\mathcal{T}|$ and $|\mathcal{T}'|$ are the number of edges of $\mathcal{T}$ and $\mathcal{T}'$, then $n \leq 2 i(\mathcal{T},\mathcal{T}')+|\mathcal{T}|+|\mathcal{T}'|$.
	\label{bound in terms of intersection} 
\end{lem}

\begin{proof}
	Let $\mathcal{U}$ be the pre-triangulation obtained by superimposing $\mathcal{T}$ and $\mathcal{T}'$. Construct an ordering $\mathrm{O}_2$ on the edges of $\mathcal{U}$ as follows.
	
	\begin{enumerate}
		\item First subdivide the edges of $\mathcal{T}$ by introducing the new vertices that are in the intersection of $\mathcal{T}$ and $\mathcal{U}$. Let $\mathcal{T}_1$ be the new pre-triangulation, and $\mathrm{O}_1$ a new ordering on the edges of $\mathcal{T}_1$ constructed by repeated application of Lemma \ref{subdividing} a). 
		
		\item Secondly, starting with $(\mathcal{T}_1, \mathrm{O}_1)$, add the vertices of $\mathcal{U}$ that are disjoint from $\mathcal{T}_1$ together with their edges in $\mathcal{U}$ to obtain $\mathcal{U}$. Let $\mathrm{O}_2$ be an ordering on the edges of $\mathcal{U}$ constructed by repeated application of Lemma \ref{subdividing} b). 
	\end{enumerate}	
	
	By Lemma \ref{subdividing}, the pants decompositions $\mathrm{\textbf{P}}(\mathcal{T}, \mathrm{O})$ and $\mathrm{\textbf{P}}(\mathcal{U},\mathrm{O}_2)$ are equal. Similarly, there is an ordering $\mathrm{O}'_2$ on the edges of $\mathcal{U}$ such that the pants decompositions $\mathrm{\textbf{P}}(\mathcal{T}', \mathrm{O}')$ and $\mathrm{\textbf{P}}(\mathcal{U},\mathrm{O}'_2)$ are equal. Now by Lemma \ref{changing the ordering} the distance between the pants decompositions $\mathrm{\textbf{P}}(\mathcal{U}, \mathrm{O}_2)$ and $\mathrm{\textbf{P}}(\mathcal{U},\mathrm{O}'_2)$ is $O(n^2)$, proving the first part of the lemma. For the second part we have 
	\begin{align*} 
		2n &= \text{sum of the degrees of vertices in } \mathcal{U} \\
		&\leq 4 i(\mathcal{T}, \mathcal{T}') + \text{sum of degrees in } \mathcal{T} + \text{sum of degrees in } \mathcal{T}'  \\
		& = 4 i(\mathcal{T}, \mathcal{T}') + 2|\mathcal{T}|+2|\mathcal{T}'|.
	\end{align*}	
	
\end{proof}

\begin{cor}
	Let $P$ and $P'$ be pants decompositions of a compact orientable surface $S$. The distance between $P$ and $P'$ in the pants graph is $O(i(P, P')^2)$. 	
	\label{quadratic bound}
\end{cor}

\begin{proof}
	Let $P_{\cap}$ be the union of simple closed curves that appear in both $P$ and $P'$. Let $S_1$ be the surface obtained by cutting $S$ along $P_\cap$, and let $P_1$ and $P'_1$ be the corresponding pants decompositions of $S_1$. Hence $P_1$ is the image of $P \setminus P_\cap$ in $S_1$, and $P'_1$ is defined similarly.  The pants decompositions $P_1$ and $P'_1$ are (disconnected) pre-triangulations of $S_1$. Isotope $P'_1$ to intersect $P_1$ minimally. Note that for any ordering $\mathrm{O}$ on $P_1$ we have $\mathrm{\textbf{P}}(P_1, \mathrm{O}) = P_1$; and similarly $\mathrm{\textbf{P}}(P'_1, \mathrm{O}') = P'_1$. Hence the distance between the pants decompositions $P_1 = \mathrm{\textbf{P}}(P_1, \mathrm{O})$ and $P'_1 = \mathrm{\textbf{P}}(P'_1, \mathrm{O}')$ is $O(n^2)$, where $n$ is the number of edges of the pre-triangulation $P_1\cup P'_1$ obtained by superimposing $P_1$ and $P'_1$. In particular, using the degree sum formula for the 4-regular graph $P_1\cup P'_1$ we have
	\[ n = \text{number of edges of } P_1 \cup P'_1 = \frac{1}{2} \text{ degree sum} = 2 i(P_1, P'_1) = 2 i(P, P'). \] 
	
\end{proof}

\section{Train tracks, and Agol--Hass--Thurston algorithm}

Agol, Hass, and Thurston used their algorithm to count the number of connected components of a simple closed multicurve $\gamma$ carried by a train track $\tau$, in time that is bounded above by a polynomial function of $E \log(N)$, where $E$ is the number of branches of $\tau$ and $N$ is the total weight of $\gamma$ with respect to $\tau$. We are interested in the way that the algorithm does this count; this will be via a sequence of \emph{splitting} and \emph{twirling} (Definition \ref{def:twirling}, and Figure \ref{fig:twirling}) the weighted train track until it completely unwinds to the original curve $\gamma$. We will apply the AHT algorithm to another similar setting and show in Proposition \ref{running AHT} that if one starts with say a one-vertex triangulation $\mathcal{T}$ carried by a \emph{based train track} $\tau$ (Definition \ref{def:based train track}), and set up an appropriate orbit counting problem (Definition \ref{orbit counting problem}), then the AHT algorithm does the count via splitting and twirling the based weighted train track until it completely unwinds it to $\mathcal{T}$. Moreover the number of splitting and twirling is bounded above by a polynomial function of $E \log(N)$, with $E$ equal to the number of branches of the based train track and $N$ equal to the total weight of $\tau$.

\subsection{Based integrally weighted train tracks, and orbit counting problems}

Train tracks were introduced by Thurston to study simple closed multicurves on a surface. Here we will work with 1-complexes, and a variant of a train track, which we call a based train track, will be useful. 

\begin{definition}[Based integrally weighted train track]
	A \emph{based train track} $(\tau, V)$ in a surface $S$ is an embedded finite $1$-complex with a distinguished subset  $V$ of its vertices called \emph{base vertices} such that 
	\begin{enumerate}
		\item[-] the embedding is $C^1$ in the interior of each edge of $\tau$;
		\item[-] at every vertex $w \notin V$ there is a well-defined tangent line, and there are edges entering $w$ from both directions. See Figure \ref{fig:based train track}, top-left. 
	\end{enumerate}
	An edge of $\tau$ is called a \emph{branch}. Fix a point $p_e$ in the interior of each branch $e$. Each component of $e \setminus p_e$ is called a \emph{half-branch}. Therefore the half-branches at a vertex $w \notin V$ can be partitioned into two non-empty subsets, which we can locally think of as incoming vs outgoing; the choice of which subset is incoming and which is outgoing is arbitrary. An \emph{integral weight} $\mu$ on $\tau$ is an assignment of \emph{positive integers} to branches of $\tau$ such that at each vertex $w \notin V$ the \emph{switch condition} is satisfied; i.e. sum of the weights of incoming half-branches is equal to that of outgoing half-branches. We denote this common number by $\mu_w$. A \emph{based integrally weighted train track} $(\tau, V, \mu)$ is a based train track $(\tau, V)$ together with an integral weight $\mu$ on its branches.  When $V$ is empty, we obtain the usual notions of a \emph{train track} $\tau$ and an \emph{integrally weighted train track} $(\tau, \mu)$. 
	\label{def:based train track}
\end{definition}

\begin{figure}
	\labellist
	\pinlabel $2$ at 50 270
	\pinlabel $2$ at 40 200
	\pinlabel $2$ at 10 240
	\pinlabel $3$ at 200 270 
	\pinlabel $3$ at 200 185
	\pinlabel $3$ at 60 30
	\pinlabel $2$ at 60 95
	\pinlabel $2$ at 165 95
	\pinlabel $1$ at 165 30 
	\endlabellist
	
	\includegraphics[width = 3 in]{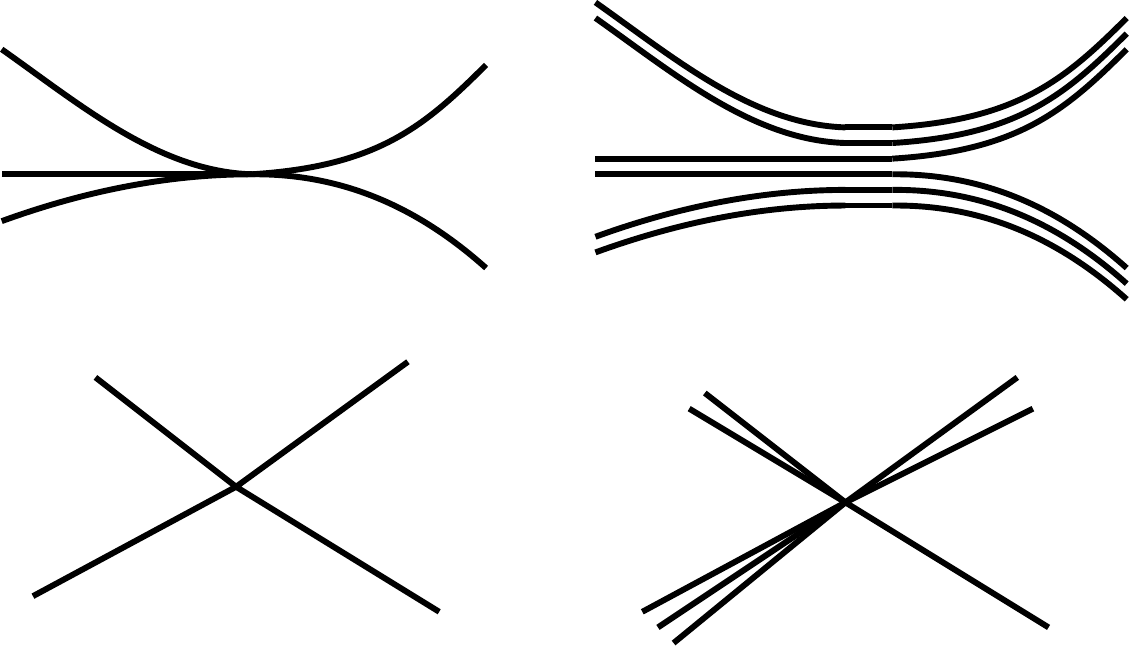}
	\caption{Left: on top, a non-base vertex of the based integrally weighted train track $(\tau, V, \mu)$ is shown; a base vertex is depicted in the bottom. The local picture for the associated $1$-complex $\mathcal{CC}(\tau, V, \mu)$ is shown on the right.}
	\label{fig:based train track}
\end{figure}

\begin{remark}
	Note that unlike some texts, we do not require the complementary regions to a train track to have negative (or non-positive) \emph{index} (a variant of the Euler characteristic). 
\end{remark}

The following is analogous to the multicurve carried by an integrally weighted train track.

\begin{definition}[The $1$-complex carried by a based integrally weighted train track]
	Given a based integrally weighted train track $(\tau, V, \mu)$ on a surface $S$, define the $1$-complex $\mathcal{CC}(\tau, V, \mu)$ embedded in $S$ as follows: given a branch $e$ of $\tau$, not adjacent to any vertex in $V$, replace $e$ by $\mu_e$ parallel segments where $\mu_e$ is the weight of $e$. Given a branch $e$ that is adjacent to a vertex $v \in V$, replace $e$ with $\mu_e$ parallel segments and identify one ends of these segments with $v$. At each vertex $w \notin V$ of $\tau$, glue the endpoints of the incoming segments adjacent to $w$ to those of the outgoing segments in an order-preserving way. See Figure \ref{fig:based train track}. We call $\mathcal{CC}(\tau, V, \mu)$ the \emph{$1$-complex carried by $(\tau, V,  \mu)$}.
	\label{def:carried 1-complex}
\end{definition}	

Given a based integrally weighted train track $(\tau, V , \mu)$ on a surface $S$, define a \emph{branched neighbourhood $N(\tau, V, \mu)$} of $(\tau, V,  \mu)$ in $S$ as follows: For each branch $e$, take a rectangle $R_e \cong e \times [0, \mu_e]$ around $e$ equipped with the horizontal foliation whose leaves consist of $e \times \text{point}$. At each vertex $w \notin V$ glue the rectangles $R_e$ adjacent to $w$ along their vertical boundary $\partial e \times [0, \mu_e]$ according to their adjacency and the length of their vertical boundary; this is possible by the switch condition. See Figure \ref{branched neighbourhood}. For any base vertex $v \in V$, let $\deg(v)$ be the number of half-branches adjacent to $v$. Take a $2 \deg(v)$-gon $P_v$ around $v$ and identify half of its edges with the vertical boundaries of the adjacent rectangles in an alternating fashion; see Figure \ref{fig: branched-neighbourhood-base-vertex}. The \emph{free sides} of $P_v$ are those that are not identified with the vertical boundary components of adjacent rectangles. Note that $N(\tau,V , \mu) \setminus (\cup_{v\in V} P_v )$ comes equipped with a \emph{horizontal foliation} obtained by gluing together the horizontal foliations of all rectangles $R_e$. 

Each rectangle $R_e$ has an $I$-bundle structure given by fibres $\text{point} \times [0, \mu_e]$. The \emph{tie interval above a point} in $\tau$ is defined as follows:
\begin{itemize}
	\item For each point in the interior of an edge $e$, the \emph{tie interval} above the point is the $I$-fibre above that, see the (blue) vertical solid lines in Figure \ref{branched neighbourhood}. 
	
	\item For a vertex $w \notin V$, the \emph{tie interval above $w$} in $N(\tau, \mu)$ is the union of the $I$-fibres over $w$ coming from the half-branches adjacent to $w$; see Figure \ref{branched neighbourhood}. Moreover, a point $c$ on the tie interval above $w \notin V$ is called a \emph{cusp} if a neighbourhood of $c$ in the leaf of the horizontal foliation through it is not homeomorphic to $\mathbb{R}$; see Figure \ref{branched neighbourhood}. 
	
	\item For a base vertex $v \in V$, the \emph{tie interval} above $v$ is the disjoint union of $I$-fibres over $v$ coming from the half-branches adjacent to $v$, see Figure \ref{fig: branched-neighbourhood-base-vertex}. 
\end{itemize}
Note that $(\tau, V, \mu)$ is determined uniquely, up to isotopy, by the branched neighbourhood $N(\tau, V, \mu)$.  

\begin{figure}
	\labellist
	\pinlabel $w$ at 155 50
	\pinlabel $6$ at 30 155
	\pinlabel $3$ at 30 115
	\pinlabel $3$ at 30 77
	\pinlabel $3$ at 295 155
	\pinlabel $6$ at 295 85
	\pinlabel $3$ at 295 25
	
	\pinlabel $c_1$ at 510 100
	\pinlabel $c_2$ at 550 70
	\pinlabel $c_3$ at 510 32
	\endlabellist
	
	\includegraphics[width = 3.5 in]{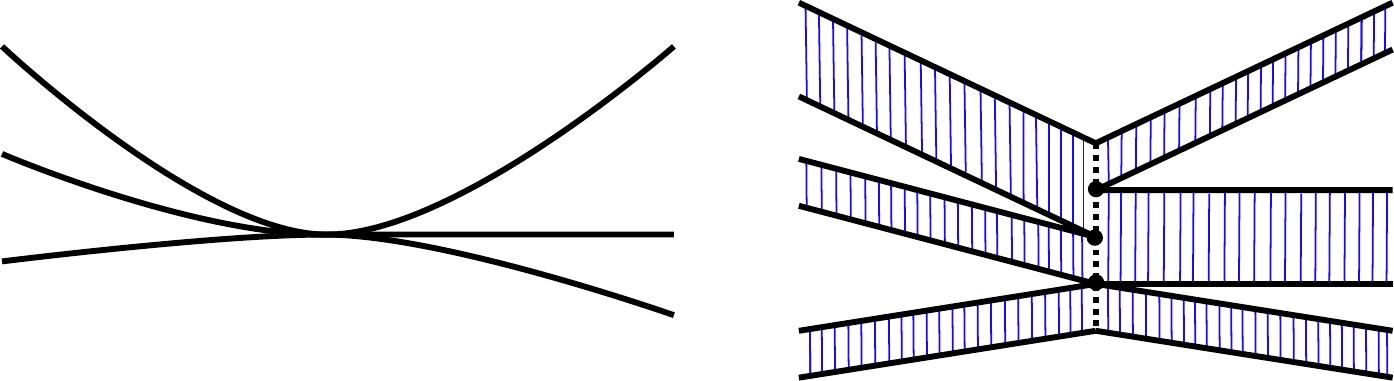}
	\caption{A non-base vertex $w$ of a based integrally weighted train track $(\tau, V , \mu)$ on the left, and the local picture for the branched neighbourhood $N(\tau, V , \mu)$ on the right. The tie interval above $w$ is shown by dashed lines, and the cusp points $c_i$ are depicted as dots on it. Other tie intervals are shown with (blue) vertical solid lines. }
	\label{branched neighbourhood}
\end{figure}

\begin{figure}
	\labellist
	\pinlabel $v$ at 50 80 
	\pinlabel $P_v$ at 295 80
	\pinlabel $3$ at 30 30 
	\pinlabel $3$ at 30 130
	\pinlabel $2$ at 110 95  
	\endlabellist
	
	\includegraphics[width=3 in]{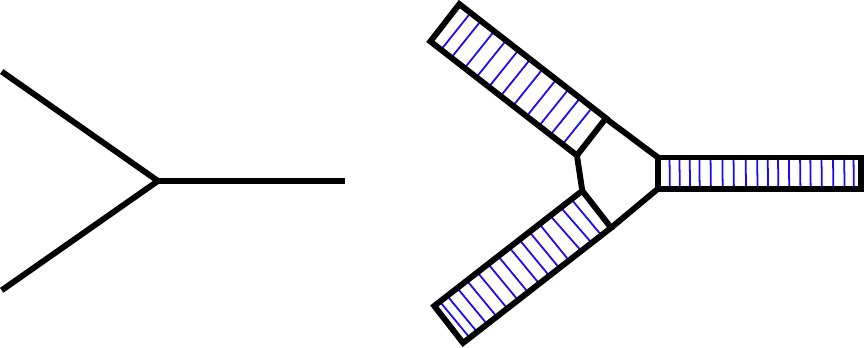}
	\caption{A base vertex  of a based integrally weighted train track $(\tau, V , \mu)$ on the left, and the local picture for the branched neighbourhood $N(\tau, V , \mu)$ on the right. The tie intervals are shown with solid arcs.}
	\label{fig: branched-neighbourhood-base-vertex}
\end{figure}

Just as in splittings of (unweighted) train tracks, it makes sense to talk about a splitting of an (unweighted) based train track $(\tau, V)$ along a branch $e$. Let $e$ be an oriented edge with the initial vertex $w \notin V$ and the terminal vertex $w'$. A splitting of $(\tau, V)$ is a combinatorial local move as in Figure \ref{fig: splitting unweighted train track}.

A splitting of $(\tau, V, \mu)$ is defined just like a splitting of a weighted train track except we require the support of the splitting to be disjoint from $V$. Namely let $N(\tau, V, \mu)$ be a branched neighbourhood of $(\tau, V, \mu)$. Consider a vertex $w \notin V$, and the tie interval $t$ above $w$. Let $p$ be a cusp point of $t$, and $l$ be the leaf of the horizontal foliation on $N(\tau, V, \mu) \setminus (\cup_{v\in V} P_v )$ passing through $p$ and oriented starting at $p$. Denote by $e$ the first branch of $\tau$ travelled by $l$, and give $e$ the orientation from $l$. If $l \neq p$, let $q$ be the next point of intersection of $l$ with the union of the tie intervals above the vertices of $\tau$, it could be that $q =p$ if $l$ is a simple closed curve. Denote the segment of $l$ between $p$ and $q$ by $l_0= [p, q)$ where $l_0$ contains $p$ but not $q$. If $l = p$, set $l_0 = p$. Then the \emph{splitting of $N(\tau, V, \mu)$ along the cusp point $p$ (or the leaf segment $l_0$)} is defined as the closed complement $N(\tau, V, \mu) \setminus \setminus l_0$ with one exception: if $q$ lies on a tie interval above a base vertex $v \in V$ then we split $N(\tau, V, \mu)$ along $l_0$ and then modify the polygon $P_v$ by adding a free side around $q$; see Figures \ref{fig: splitting unweighted train track} and \ref{fig:splitting}. Sometimes we refer to this operation as a splitting of $N(\tau, V, \mu)$ along the branch $e$, if we do not need to stress the choice of the cusp point $p$. We say that $(\tau', V, \mu')$ is obtained by \emph{splitting} $(\tau, V, \mu)$ if $N(\tau', V, \mu')$ is obtained by splitting $N(\tau, V, \mu)$; note that the set of base vertices $V$ is not changed during the splitting.

\begin{figure}
	
	\includegraphics*[width = 3 in]{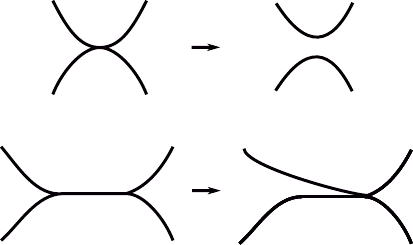}
	\centering
	\caption{Splitting an unweighted train track along an edge or a vertex. The picture on the left is replaced with the one on the right, after splitting along the middle vertex or edge.}
	\label{fig: splitting unweighted train track}
\end{figure}

\begin{figure}
	\labellist
	\pinlabel $P_v$ at 75 55 
	\pinlabel $P'_v$ at 75 10
	\endlabellist
	\includegraphics[width=2 in]{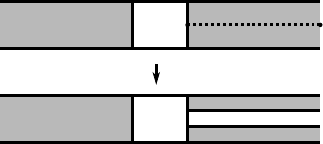}
	\caption{Splitting of a branched neighbourhood. The shaded regions are the rectangles over various edges. The dashed line is the leaf segment $l = [p,q)$, where $q$ lies on $P_v$ here. }
	\label{fig:splitting}
\end{figure}

\begin{definition}[Orbit counting problem associated with a based integrally weighted train track]
	
	Let $(\tau, V, \mu)$ be a based integrally weighted train track. For each vertex $w \notin V$, define $\mu_w$ as sum of the weights of the incoming (or outgoing) edges at $w$, and pick an interval $t_w \subset \mathbb{N}$ of length $\mu_w$. Visually we think of $t_w$ as the intersection of the tie interval above $w$ with the carried 1-complex $\mathcal{CC}(\tau, V, \mu)$. Similarly, let $e_1, \cdots, e_r$ be the half-branches adjacent to base vertices, and for each $e_j$ consider an interval $t_j \subset  \mathbb{N}$ of length $\mu_{e_j}$. Again we think of $t_j$ as the intersection of a component of the tie interval above a base vertex $v$ with the 1-complex $\mathcal{CC}(\tau, V, \mu)$. By abuse of notation, we call the intervals $t_w$ and $t_j$ the tie intervals. Pick intervals $I_{j} \subset \mathbb{N}$ where the index $j$ varies in $\mathcal{J}: = \{ 1, \cdots, r\} \cup \{w | w \notin V  \text{ is a vertex} \}$ such that
	\begin{enumerate}
		\item the interval $I_j$ has the same length as $t_j$ for $j \in \mathcal{J}$; 
		\item the intervals $I_j$ for $j \in \mathcal{J}$ are pairwise disjoint and their union is $[1,N]$ for some positive integer $N$; and 
		\item the union of $I_j$ where $j \in \{ 1 , \cdots, r\}$ is $[1, M]$ for some positive integer $M $. 
	\end{enumerate}
	
	Pick isometric identifications $i_j \colon t_j \rightarrow I_j$; there are two such choices for each index $j$. The data $\mathcal{I}: = (t_j , I_j , i_j)_{j \in \mathcal{J}}$ is called an \emph{initial interval identification}. Define an orbit counting problem $\mathrm{OCP}(\tau, V, \mu, \mathcal{I})$ as follows. Let $e$ be an edge of $\tau$ connecting vertices $w$ and $w'$, and assume that $I_{w'}$ lies to the right of $I_w$ as subsets of $\mathbb{N}$. Define an isometric pairing $g_e$ by following the segments of the $1$-complex $\mathcal{CC}(\tau, V, \mu)$ along the rectangle $R_e$ where $\text{domain}(g_e) \subset I_w$ and $ \text{range}(g_e) \subset I_{w'}$. Then the isometric pairings $\{ g_e| e \text{ is a branch} \}$ form the orbit counting problem $\mathrm{OCP}(\tau, V, \mu, \mathcal{I})$ with total interval $[1,N]$. 
	
	\label{orbit counting problem}
\end{definition}

\begin{remark}
	By construction, every point in $[1, N]$ appears in the domain and range of at most two pairings of $\mathrm{OCP}(\tau, V, \mu, \mathcal{I})$. 
\end{remark}

%
%

The following geometric operation on a based integrally weighted train track appears naturally when applying the AHT algorithm to the orbit counting problem $\mathrm{OCP}(\tau, V, \mu, \mathcal{I})$. 

\begin{definition}[Twirling]
	Let $(\tau, V, \mu)$ be a based integrally weighted train track on an orientable surface $S$, and let $N = N(\tau, V, \mu)$ be a branched neighbourhood of it. Let $w \notin V$ be a vertex, and $e$ be a branch connecting $w$ to itself. Denote the tie interval above $w$ by $t$, and fix an identification of $t$ with the interval $[0, \mu_w]$ (there are two choices for such identification). Call $0 \in [0, \mu_w] \cong t$ the \emph{top} point of $t$. Let $e \times [0, \mu_e] \cong R_e \subset N$ be the rectangle around $e$. Write $\partial e = \partial_{-} e \cup \partial_{+} e$, and orient $e$ from $\partial_{-} e$ to $\partial_{+} e$. Call $\partial_{-} e \times [0, \mu_e]$ the negative vertical boundary of $R_e$, similarly define the positive vertical boundary. Assume that $R_e$ starts and ends on different sides of $t$, and that the images of the negative and positive vertical boundary components of $R_e$ in $N$ intersect each other. (The operation of twirling is only defined under this assumption). Let the positive side of $t$ be the side of $\partial_+ e$; similarly for the negative side. Assume that $e$ is the top branch coming into $w$ on the negative side of $t$. See Figure \ref{fig:twirling}.
	
	Let $p_1, \cdots, p_k$ be the cusp points on $t$ that lie above $\partial_{+} e \times (0, \mu_e)$. Let $l_i$ be the leaf of the horizontal foliation through $p_i$, and $s_i$ be connected component of $l_i \cap R_e$ that contains $p_i$. Split $N$ along all $s_i$. More precisely, set $p_i^0 = p_i$, and if $l_i \neq \{ p_i^0\}$ then define $p_i^1$ to be the first intersection point of $l_i$ with the positive vertical boundary of $R_e$; and $p_i^1$ is not defined otherwise. Then by the orientability of $S$, the relative order of $p_i^1$ on $t$ is the same as that of $p_i$; i.e. $p_i^1$ is on top of $p_j^1$ if $i<j$ as long as both are defined. Now if all $p_i^1$ are defined and lie in the negative vertical boundary of $R_e$, we can repeat the same process starting with $p_i^1$ to obtain points $p_i^2$ and so on. Hence there is an $m$ such that at least one of the points $p_1^m, \cdots, p_k^m$ is either not defined or does not lie in the negative vertical boundary of $R_e$; say only $p_1^m, \cdots , p_n^m$ are defined and lie in the negative vertical boundary of $R_e$ for some $0 \leq n <k$. Then we split one further time along $p_1^m, \cdots, p_n^m$. The resulting based integrally weighted train track is called a \emph{twirling of $(\tau, V, \mu)$ along the oriented branch $e$}. 
	\label{def:twirling}
\end{definition}

\begin{remark}
	If $(\tau', V, \mu')$ is obtained from $(\tau, V, \mu)$ via a splitting, then the number of branches of $\tau'$ is no more than that of $\tau$. Since every twirling is a concatenation of a number of splits, the same holds if $(\tau', V, \mu') $ is obtained from $(\tau, V, \mu)$ via a twirling. 
	\label{number of branches after splitting and twirling}
\end{remark}

\begin{figure}
	\labellist
	\pinlabel $w$ at 410 90
	\pinlabel $e$ at 385 110
	
	\pinlabel $A$ at 30 145 
	\pinlabel $A$ at 290 20
	
	\pinlabel $\textbf{a}$ at 455 50
	\pinlabel $\textbf{b}$ at 460 70
	\pinlabel $\textbf{c}$ at 450 105

	\endlabellist
	
	\includegraphics[width=5 in]{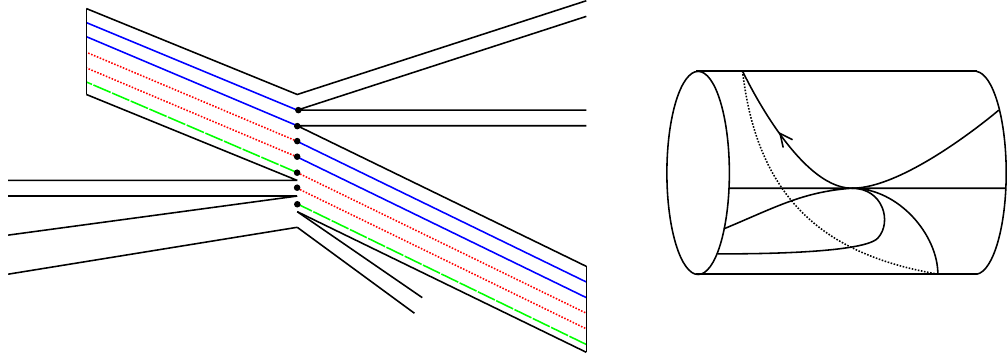}
	\caption{A twirling. A neighbourhood of the branch $e$ and the underlying train track is shown on the right. Here $k=2$, $m=2$, and $n=1$. In the left side, the two $A$ should be identified with each other to construct the rectangle $R_e$. The solid points from top to bottom are $p_1^0$, $p_2^0$, $p_1^1$, $p_2^1$, $p_1^2$, $p_2^2$ and $p_1^3$. The numbers $\textbf{a}$, $\textbf{b}$, and $\textbf{c}$ are the weights of $e$ and the two edges coming into it from right. }
	\label{fig:twirling}
\end{figure}

\begin{lem}[Effect of twirling on the underlying train track]
	Let $(\tau, V, \mu)$ be a based integrally weighted train track on an orientable surface $S$. Assume that $(\tau', V, \mu')$ is obtained from $(\tau, V, \mu)$ by twirling along the branch $e$. Let $\alpha$ be the simple closed curve that is the closure of $e$, and $m$ be as in Definition \ref{def:twirling}. Denote the Dehn twist along $\alpha$ by $T_\alpha$. Then $\tau'$ is obtained from $\tau$ via twisting by $(T_\alpha)^{m-1}$ followed by splitting the (unweighted) based train track $(T_\alpha)^{m-1}(\tau, V)$ a number of times along $e$. 
\label{effect of twirling on the underlying 1-complex}
\end{lem}

\begin{proof}
	Define the points $p_i^r$ as in Definition \ref{def:twirling}. Let $l_i^r$ be the segment of $l_i$ between $p_i^r$ and $p_i^{r+1}$, if both $p_i^r$ and $p_i^{r+1}$ are defined. Then repeatedly splitting along $\cup_{i=1}^{k} l_i^r$ for $r=0, 1, \cdots, m-2$ replaces $\tau$ with $(T_\alpha)^{m-1}(\tau)$. Then splitting along $\cup_{i=1}^{k} p_i^{m-1}$ followed by splitting further along $\cup_{i=1}^{n} p_i^{m}$ is a sequence of splits along the oriented branch $e$. See Figure \ref{fig:twirling}.
\end{proof}

\begin{lem}
	Let $(\tau, V)$ be a based train track on a compact orientable surface $S$, and that $(\tau' , V)$ is obtained from $(\tau, V)$ by a splitting. 
	Assume that the underlying 1-complex of $\tau'$ is a polygonal decomposition (respectively a pre-triangulation) of $S$. 
	Then the underlying 1-complex of $\tau$ is also a polygonal decomposition (respectively a pre-triangulation). The corresponding statement holds if based train tracks are replaced with based integrally weighted train tracks. 
	\label{underlying 1-complex}
\end{lem}

\begin{proof}
	
	The Euler characteristics of the complementary components of a train track will not increase as a result of splitting, but they might decrease if some complementary components merge together. Now each complementary component of 
	$\tau'$
	has Euler characteristic at least $1$ (respectively $-1$). It follows that the complementary components of $\tau'$ have Euler characteristic at least $1$ (respectively $-1$) as well. It remains to show that no complementary component of $\tau$ is topologically a torus minus a disc. Observe that if $\tau'$ is obtained by splitting $\tau$, and $\alpha$ is an essential (i.e. homotopically non-trivial and non-boundary-parallel) curve in a complementary component $R$ of $\tau$, then some complementary component of $\tau'$ also contains an essential curve. It follows that no complementary component $R$ of $\tau$ is a torus minus a disc; otherwise $R$ contains an essential curve which implies that some complementary component of 
	$\tau'$
	also contains an essential curve, a contradiction. 
	
	A splitting of a based integrally weighted train track is in particular a splitting of the underlying unweighted based train track. Therefore, the same proof applies to the weighted case as well. 
	
	%
\end{proof}

\begin{lem}
	Let $(\tau, V, \mu)$ be a based integrally weighted train track with $E$ edges on a compact orientable surface $S$. Assume that $(\tau', V, \mu')$ is obtained from $(\tau, V, \mu)$ by either a split or a twirl. Assume that the underlying 1-complex of $\tau$ is a pre-triangulation. Set $k = \chi(\tau') - \chi(\tau) \geq 0$, where $\chi(\tau)$ is the Euler characteristic of $\tau$ as a 1-complex. For every ordering $\mathrm{O}$ on the edges of $\tau$, there is an ordering $\mathrm{O}'$ on the edges of $\tau'$ such that the pants decompositions $\mathrm{\textbf{P}}(\tau, \mathrm{O})$ and $\mathrm{\textbf{P}}(\tau', \mathrm{O}')$ have distance $O((k+1)E)$ in the pants graph. 
	
	Moreover, given the train track $\tau$, the ordering $\mathrm{O}$, and the split or twirl move from $\tau$ to $\tau'$, there is an algorithm that constructs such an ordering $\mathrm{O}'$ together with a sequence of $O((k+1) E)$ consecutive transpositions taking $\mathrm{O}$ to $\mathrm{O}'$. The algorithm runs in time that is a polynomial function of $E$. Here the train track is given as follows: the complementary components of $\tau$ are cusped polygons, annuli, or pairs of pants; these are given together with their side identifications. Furthermore, the twirl or split move is given by specifying an oriented edge.
	
	\label{lem:local distance}
\end{lem}

\begin{proof}
	We first show how to reduce to the case where $\tau'$ is obtained from $\tau$ by splitting a number of times along the same oriented edge. Assume that $(\tau', V, \mu')$ is obtained from $(\tau, V, \mu)$ by a twirl along a branch $e$, and let $\alpha$ be the simple closed curve that is the closure of $e$. By Lemma \ref{effect of twirling on the underlying 1-complex}, $\tau'$ is obtained from $\tau$ by some power $(T_{\alpha})^\ell$ of the Dehn twist along $\alpha$ followed by splitting $(T_\alpha)^\ell(\tau, V)$ a number of times along $e$. Given the ordering $\mathrm{O}$ on $\tau$, we can apply at most $E$ consecutive transpositions to obtain $\mathrm{O}_1$ such that $e$ appears as the first edge in $\mathrm{O}_1$. The distance between the pants decompositions $\mathrm{\textbf{P}}(\tau, \mathrm{O})$ and $\mathrm{\textbf{P}}(\tau, \mathrm{O}_1)$ is $O(E)$ by Lemma \ref{transposition}. Since $(T_\alpha)^\ell$ is a mapping class, it induces an ordering $((T_\alpha)^\ell)_*(\mathrm{O}_1)$ on $(T_\alpha)^\ell(\tau)$.  We claim that the pants decompositions $\mathrm{\textbf{P}}(\tau, \mathrm{O}_1)$ and $\mathrm{\textbf{P}}((T_\alpha)^\ell(\tau), ((T_\alpha)^\ell)_*(\mathrm{O}_1))$ are the same. To see this, note that $e$ appears as the first edge in $\mathrm{O}_1$, and so Construction \ref{const:triangulation-to-pants} produces $\alpha$ as the first curve in $\mathrm{\textbf{P}}(\tau, \mathrm{O}_1)$. Hence 
	\[ \mathrm{\textbf{P}}((T_\alpha)^\ell(\tau), ((T_\alpha)^\ell)_*(\mathrm{O}_1)) = (T_\alpha)^\ell(\mathrm{\textbf{P}}(\tau, \mathrm{O}_1)) = \mathrm{\textbf{P}}(\tau, \mathrm{O}_1). \]
	The last equality holds because Dehn twisting a pants decomposition along one of its curves does not change the pants decomposition. This shows that at the cost of distance $O(E)$ in the pants graph, we may replace $(\tau, \mathrm{O})$ with $((T_\alpha)^\ell(\tau), ((T_\alpha)^\ell)_*(\mathrm{O}_1)) $, completing the reduction. 
	
	Now assume that $\tau'$ is obtained from $\tau$ by splitting a number of times along the same oriented edge $e$. Again at the cost of distance $E$ in the pants graph, we may assume that $e$ appears as the first edge in $\mathrm{O}$. Assume that $k \leq 1$; the general case follows from this. Since $\tau$ splits to $\tau'$, the train track $\tau'$ is carried by $\tau$. This means that there is a neighbourhood $N(\tau)$ of $\tau$ foliated by interval fibers, a projection map $\pi \colon N(\tau) \rightarrow \tau$ collapsing each fiber to a point, and a map $f \colon \tau' \rightarrow N(\tau)$ transverse to the fibers. In particular, the composition $\pi \circ f$ gives a map $ \pi \circ f \colon \tau' \rightarrow \tau$, which we abbreviate to $\pi \colon \tau' \rightarrow \tau$. Consider the following cases:
	
	\textit{Case 1)}: If $k=0$ and $e$ has distinct endpoints. Here $T: = \pi^{-1}(e)$ is a tree. Contract $e$ in $\tau$ to a point to obtain a pre-triangulation $\tau_1$. Let $\mathrm{O}_1$ be the ordering on $\tau_1$ induced from the ordering $\mathrm{O}$ on $\tau$ (by removing $e$). By Construction \ref{const:triangulation-to-pants}, since $e$ appears as the first edge in $\mathrm{O}$ and the endpoints of $e$ are distinct, we have 
	\[ \mathrm{\textbf{P}}(\tau, \mathrm{O}) = \mathrm{\textbf{P}}(\tau_1, \mathrm{O}_1). \]
	For each point in $\pi^{-1}(\partial e)$, add it as a vertex (possibly of degree two) in $\tau'$. Define an ordering $\mathrm{O}'$ on $\tau'$, where the restriction of $\mathrm{O}'$ to $\tau' - T$ agrees with the restriction of $\mathrm{O}$ to $\tau - e$ and such that the edges of $\tau'$ lying in $T$ appear first (in any order) in $\mathrm{O}'$. Contract the tree $T$ in $\tau'$ to obtain $\tau'_1$, let $\mathrm{O}'_1$ be the ordering on $\tau'_1$ obtained from the ordering $\mathrm{O}'$ on $\tau'$ by deleting the edges that lie in $T$. We have 
	\[ \mathrm{\textbf{P}}(\tau'_1, \mathrm{O}'_1) = \mathrm{\textbf{P}}(\tau', \mathrm{O}'). \]
	Note that $(\tau_1, \mathrm{O}_1)$ is equal to $(\tau'_1, \mathrm{O}'_1)$ up to isotopy. Hence 
	\[ \mathrm{\textbf{P}}(\tau, \mathrm{O}) = \mathrm{\textbf{P}}(\tau_1, \mathrm{O}_1) = \mathrm{\textbf{P}}(\tau'_1, \mathrm{O}'_1) = \mathrm{\textbf{P}}(\tau', \mathrm{O}'). \] 
	
	\textit{Case 2)}: If $k=1$ and $e$ has distinct endpoints. Here $\pi^{-1}(e)$ has two connected components, each of which is a tree. Contract $e$ in $\tau$ to obtain $(\tau_1, \mathrm{O}_1)$.  For each point in $\pi^{-1}(\partial e)$ that is not a vertex of $\tau'$, add it as a vertex of degree two to obtain  $\hat{\tau'}$. Attach an extra edge $e'$ (joining the two components of $T$) to $\hat{\tau'}$ to obtain a 1-complex $\tau'_1$ such that if we contract $e' \cup \pi^{-1}(e)$ in $\tau'_1$ the result is isotopic to $\tau_1$. See Figure \ref{fig: new ordering lemma}.
	
	\begin{figure}
		\includegraphics*[width = 2.5 in]{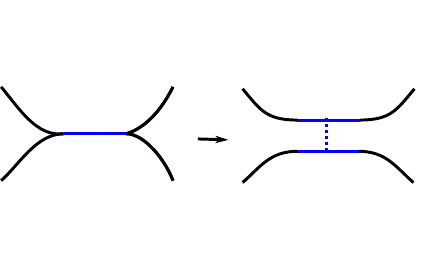}
		\caption{The train track $\tau'$ on the right is obtained by splitting $\tau$ (left). The added edge $e'$ is shown with a dotted segment. The edge $e$ and the preimage $\pi^{-1}(e)$ are shown in solid blue segments.}
		\label{fig: new ordering lemma}
	\end{figure}

	Let $\mathrm{O}'_1$ be the ordering on $\tau'_1$ such that the edges in $e' \cup \pi^{-1}(e)$ appear first, and otherwise it agrees with $\mathrm{O}_1$. We have 
	\[ \mathrm{\textbf{P}}(\tau, \mathrm{O}) = \mathrm{\textbf{P}}(\tau_1, \mathrm{O}_1) = \mathrm{\textbf{P}}(\tau'_1, \mathrm{O}'_1). \]
	Let $\mathrm{O}'_2$ be the ordering on $\tau'_1$ such that $e'$ appears last, and otherwise the ordering agrees with $\mathrm{O}'_1$. Then $\mathrm{O}'_2$ can be obtained from $\mathrm{O}'_1$ by $O(E)$ consecutive transpositions. Therefore, the distance between $\mathrm{\textbf{P}}(\tau'_1, \mathrm{O}'_2)$ and $\mathrm{\textbf{P}}(\tau'_1, \mathrm{O}'_1)$ is $O(E)$. Let $\hat{\mathrm{O}'}$ be the ordering on $\hat{\tau'}$ obtained by removing $e'$ (the last edge) from $\tau'_1$. We claim that $\mathrm{\textbf{P}}(\hat{\tau'}, \hat{\mathrm{O}'})$ is equal to $\mathrm{\textbf{P}}(\tau'_1, \mathrm{O}'_2)$. This is because, $e'$ appears last in $\mathrm{O}'_2$, and given that $\tau'$ and hence $\hat{\tau'}$ is a pre-triangulation, adding $e$ to $\hat{\tau'}$ (to obtain $\tau'_1$) does not create any new curve in Construction \ref{const:triangulation-to-pants}. Hence the distance between $\mathrm{\textbf{P}}(\tau, \mathrm{O}) = \mathrm{\textbf{P}}(\tau'_1, \mathrm{O}'_1)$ and $\mathrm{\textbf{P}}(\hat{\tau'}, \hat{\mathrm{O}'}) = \mathrm{\textbf{P}}(\tau'_1, \mathrm{O}'_2)$ is $O(E)$. On the other hand, the ordering $\hat{\mathrm{O}'}$ on $\hat{\tau'}$ induces an ordering $\mathrm{O}'$ on $\tau'$, simply by deleting the vertices $\pi^{-1}(e)$ and we have $\mathrm{\textbf{P}}(\hat{\tau'}, \hat{\mathrm{O}'}) = \mathrm{\textbf{P}}(\tau', \mathrm{O}')$. 
	
	\textit{Case 3}: If $k \leq 1$ and the endpoints of $e$ coincide. In this case we can add a vertex of degree $2$ in the middle of $e$ to obtain edges $e_1$ and $e_2$ such that $e_1 \cup e_2 = e$, each $e_i$ has distinct endpoints, and the splittings along $e$ can be obtained by splittings along $e_1$ followed by splittings along $e_2$. This reduces the problem to the previous two cases. 
	
	It is straightforward to see that the above construction of $\mathrm{O}'$ can be done algorithmically in time that is a polynomial function of $E$, $|\chi(S)|$, and $k$. Since $\tau$ is a pre-triangulation, $E$ is at least linear in $|\chi(S)|$. Similarly, $k$ is at most $E$. Therefore the algorithm runs in polynomial time in $E$.
		
\end{proof}

\subsection{Simplifying a based integrally weighted train track using the AHT algorithm}

Given a based integrally weighted train track $(\tau, V, \mu)$, one can always simplify it to its carried 1-complex using successive splittings. The number of splittings could be large, and one can instead simplify $(\tau, V, \mu)$ much faster by using both splitting and twirling moves. This section gives an upper bound for the numbers of splitting and twirling moves needed to simplify a based integrally weighted train track to its carried 1-complex, using the algorithm of Agol, Hass, and Thurston. The arguments are technical and could be skipped in a first reading, but see Remark \ref{complete unwinding}. Before giving the technical details we show the underlying principle using a simple example related to Euclid's algorithm. 

\begin{example}[Euclid's algorithm revisited]
	Let $(\tau, \mu)$ be the weighted train track on the torus shown in Figure \ref{fig:train track for Euclid's algorithm} top-left, with weights $a$, $b$, and $a+b$. This train track has two vertices $v$ and $w$ and no base vertices. Denote the standard meridian and longitude of the torus by $m$ and $\ell$ respectively. The carried 1-complex by $ (\tau , \mu) $ is a multicurve $\gamma$ whose homology class is $a m + b \ell$. Therefore the number of connected components of $\gamma$ is $\mathrm{GCD}(a, b)$. In Figure \ref{fig:train track for Euclid's algorithm} top-right, a branched neighbourhood $N$ of $(\tau , \mu)$ is shown, obtained by gluing together 3 rectangles together one for each branch of $\tau$. Identify the tie interval above the vertex $v$ (shown with a red bracket interval) with $[1, a+b]$ such that the top part of the bracket in the figure is identified with $1 \in [1, a+b]$.  Similarly identify the tie interval above $w$ with $[a+b+1, 2a+ 2b]$ such that the top part of the bracket in the figure is identified with $a+b+1 \in [a+b+1, 2a+2b]$. This is our initial interval identification $\mathcal{I}$. The orbit counting problem associated to $(\tau, \mu , \mathcal{I})$ has the following pairings: 
		\begin{align*}
		&f(x)= x+a + 2b , \hspace{6mm} & f \colon [1, a] \rightarrow [a+2b+1, 2a+2b ], \\
		&g(x) = x+b, \hspace{6mm} & g \colon [a+1, a+b] \rightarrow [a+b+1, a+2b], \\
		&h(x) = x+ a+b, \hspace{6mm} & h \colon [1, a+b] \rightarrow [a+b+1, 2a+2b].		
	\end{align*}
	Here $f, g$ and $h$ are the pairings corresponding to the edges with weights respectively $a, b$ and $a+b$. 
	\begin{figure}
		\labellist 
		\pinlabel $a$ at 65 95
		\pinlabel $b$ at 50 113
		\pinlabel $a+b$ at 90 108
		
		\pinlabel $v$ at 30 110
		\pinlabel $w$ at 69 113
		
		\pinlabel $b$ at 80 32
		\pinlabel $a$ at 65 7
		
		\endlabellist
		\includegraphics[width= 5 in]{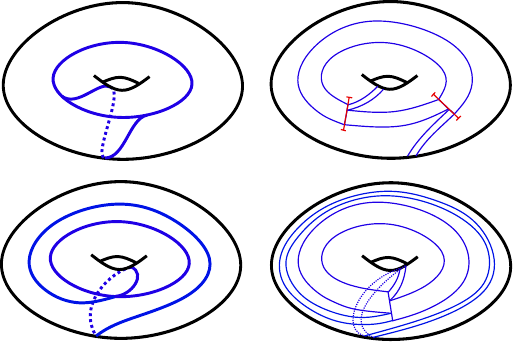}
		\caption{Top: a weighted train track (left) where the weights are $a, b, a+b$ and the vertices are $v, w$, and a branched neighbourhood of it (right). In Example \ref{ex:Euclid's algorithm revisited} we apply the AHT algorithm to unwound this train track, and see how the effect of each shifted cycle of the algorithm can be seen as a splitting or twirling on the train track. Bottom: the weighted train track obtained by splitting $\tau$ (left), and a branched neighbourhood of it (right). }
		\label{fig:train track for Euclid's algorithm}
	\end{figure}
	
	We saw in Example \ref{ex:Euclid's algorithm} that the AHT algorithm counts the number of orbits, which we know is equal to $\mathrm{GCD}(a,b)$. Let us see the effect of each shifted cycle of AHT algorithm on the weighted train track. The first shifted cycle consists of Steps (1)--(4), which are not applicable in our case, so we start with the second shifted cycle. Recall that a shifted cycle applies Steps (5) and (6), followed by Steps (1)--(4). 
	
	The pairing $h$ is maximal. Step (4) transmits $g$ by $h$ and replaces $g$ with $g' = h^{-1} \circ g$ where 
	\[ g'(x) = x-a, \hspace{6mm} g' \colon [a+1, a+b] \rightarrow [1, b].  \]
	Equivalently we can use the inverse $(g')^{-1}$ instead of $g'$ where 
	\[ (g')^{-1}(x) = x+a, \hspace{6mm} (g')^{-1} \colon [1, b] \rightarrow [a+1, a+b]. \]
	Similarly transmitting $f$ by $h$ replaces $f$ with $f' = h^{-1} \circ f$ where  
	\[ f'(x)= x+b, \hspace{6mm} f' \colon [1,a] \rightarrow [1+ b, a+b]. \]
	Step (5) then truncates $h$ to $h'$, and $N = 2a+2b$ to $N'= a+b$, where
	\[ h '(x) = x, \hspace{6mm} h' \colon [1, a+b] \rightarrow [1, a+b]. \]
	
	Step (1) then deletes the pairing $h'$, and Steps (2)--(4) are not applicable. Therefore the second shifted cycle of the AHT algorithm replaces $\{ f, g , h \}$ with $\{f' , (g')^{-1} \}$. This is the orbit counting problem associated to the train track $\tau'$ shown in Figure \ref{fig:train track for Euclid's algorithm} bottom-left. Note that $\tau'$ is obtained by splitting $\tau$. 
	
	Now we apply the shifted cycles of AHT algorithm to the weighted train track $\tau'$. Assume that $a<b$, the other case is similar. Let $b = na +r$, where $n$ is an integer and $ 0 \leq r <a$. The pairing $(g')^{-1}$ is maximal now. Step (4) transmits $f'$ by $(g')^{-1}$ and replaces $f'$ by $f'' = (g')^n \circ f'$ where 
	\[ f''(x) = x+b - na = x+r, \hspace{6mm} f'' \colon [1, a] \rightarrow [1+r, r+a]. \]
	Then Step (5) truncates $(g')^{-1}$ to $g''$, and replaces $N' = a+b$ with $N'' = a+r$, where 
	\[ g'' (x) = x+a, \hspace{6mm} g'' \colon [1, r] \rightarrow [1+a, r+a]. \]
	Assume first that $r \neq 0$. Then Steps (1)--(4) are not applicable. Hence the shifted cycle replaces $\{ f', (g')^{-1}\}$ with $\{ f'', g'' \}$. This is the orbit counting problem associated with the train track $\tau''$ obtained by twirling the edge of $\tau'$ that has weight $a$, see Figure \ref{fig:train track for Euclid's algorithm} bottom-left. If $r= 0$, then Step (1) deletes the identity pairing $f''$ and the pairing $g''$ has empty domain, so at this point the algorithm terminates and outputs $a$. This corresponds to the weighted train track getting completely unwound to the multiccurve $\gamma$ by the corresponding twirling.
\label{ex:Euclid's algorithm revisited}
\end{example}

The following is a restatement of a definition due to Dynnikov and Wiest \cite[Definition 2.8 and Remark 2.2]{dynnikov2007complexity}. It will allow us, as in the case of Dynnikov and Wiest, to obtain an improved estimate for the running time of the AHT algorithm applied to the orbit counting problem associated with an integrally weighted train track. 

\begin{definition}[AHT-complexity]
	Let $S$ be a compact orientable surface, and $(\tau, V, \mu)$ be a based integrally weighted train track on $S$. Let $\mathcal{P} = \mathrm{OCP}(\tau, V, \mu, \mathcal{I})$ be the orbit counting problem for some initial interval identification $\mathcal{I}$. Denote the pairings of $\mathcal{P}$ by $g_e$ where $e$ ranges over the branches of $\tau$. Denote by $z_1, \cdots, z_r$ the length of the intervals $\text{domain}(g_e)$; i.e. $z_i$ are the measures associated to the branches of $\tau$. Let $N$ be the maximum natural number appearing in the intervals $\text{domain}(g_e)\cup \text{range}(g_e)$. Denote by $\tilde{z}$ the length of the narrower interval attached to $N$; i.e. if $N$ is an endpoint of a tie interval $t_w$ then $\tilde{z}$ is the smaller measure for the two branches that come into $t_w$ on the side of $N$. Define the \emph{AHT-complexity} as 
	\[ c_{\text{AHT}}(\mathcal{P}):= r + \sum_{i=1}^{r} \log(z_i) - \frac{1}{2} \log(\tilde{z}). \]
	\label{def:AHT-complexity}
\end{definition}

The next proposition studies the effect of running one shifted cycle of the AHT algorithm on a based integrally weighted track track $(\tau, V, \mu)$, or more precisely on the associated orbit counting problem $\mathrm{OCP}(\tau, V, \mu, \mathcal{I})$. Roughly speaking, it states that running one shifted cycle of the AHT algorithm on $\mathrm{OCP}(\tau, V, \mu, \mathcal{I})$ gives $\mathrm{OCP}(\tau' , V', \mu', \mathcal{I}')$ for a new based integrally weighted train track $(\tau' , V', \mu')$ that is obtained from $(\tau, V, \mu)$ via splitting or twirling. In what follows, we write $(\tau, V, \mu)$ as a disjoint union of the \emph{active part} $(\tau_{a}, V_{a}, \mu_{a})$ and the \emph{stationary part} $(\tau_{s}, V_{s}, \mu_{s})$. The stationary part is completely unwound, so that $\tau_s$ has no vertices beside its base vertices $V_s$. We apply the AHT algorithm to the active part, and as a result the active and stationary parts are updated. While this notation is cumbersome, it will be useful to keep track of the stationary part. See Remark \ref{stationary part} in the proof of the proposition.

\begin{prop}[Applying the AHT algorithm to a weighted train track]
	Let $S$ be a compact orientable surface, and $(\tau, V, \mu)$ be a based integrally weighted track track on $S$ having $E$ branches. Assume that $(\tau, V, \mu)$ is a disjoint union of $(\tau_{a}, V_{a}, \mu_{a})$ and $(\tau_{s}, V_{s}, \mu_{s})$, where $\tau_{a}$ (the active part of $\tau$) has at least one vertex besides its base vertices $V_{a}$, and $\tau_{s}$ (the stationary part of $\tau$) has no vertices besides $V_{s}$. Let $\mathcal{I}_{a}$ be an initial interval identification for $(\tau_{a}, V_{a}, \mu_{a})$. There is a based integrally weighted train track $(\tau', V', \mu')$ on $S$, which is a disjoint union of active $(\tau'_a, V'_a, \mu'_a)$ and stationary $(\tau'_s, V'_s, \mu'_s)$ parts, and an initial interval identification $\mathcal{I}'_a$ on the active part of $\tau'$ with the following properties:

	\begin{enumerate}
		
		\item Applying one shifted cycle of the AHT algorithm to $\mathrm{OCP}(\tau_{a}, V_{a}, \mu_{a}, \mathcal{I}_{a})$ gives \\
		$\mathrm{OCP}(\tau'_{a}, V'_{a}, \mu'_{a}, \mathcal{I}'_{a})$.
		
		\item $(\tau', V', \mu')$ is obtained from $(\tau, V, \mu)$ by either
		
		\begin{itemize}
			 \item at most $E$ splits along the same oriented branch of $\tau$; or
			 \item a twirling.
		\end{itemize}

		\item The AHT-complexity of $\mathrm{OCP}(\tau'_{a}, V'_{a}, \mu'_{a}, \mathcal{I}'_{a})$ is at most that of $\mathrm{OCP}(\tau_{a}, V_{a}, \mu_{a}, \mathcal{I}_{a})$ minus one.
	
	\end{enumerate} 

Moreover, there is an algorithm that constructs $(\tau', V', \mu')$ together with its partition into active and stationary parts and the splits or twirl move from $(\tau, V, \mu)$ to $(\tau', V', \mu')$, in time that is a polynomial function of $E$ and $\log |\mu|$. 
	
	\label{running AHT}
\end{prop}

\begin{proof}
	
	We first prove items (1)--(2). Recall the various steps in the orbit counting algorithm:
	
	\begin{enumerate}
		
		\item Delete any pairings that are restrictions of the identity.
		
		\item Make any possible contractions.
		
		\item Trim all orientation-reversing pairings whose domain and range overlap. 
		
		\item Perform periodic mergers.
		
		\item Find a maximal $g_i$. For each $g_j \neq g_i$ whose range is contained in $[c_i,N']$, transmit $g_j$ by $g_i$.
		
		\item Find the smallest value of $c$ such that the interval $[c,N']$ intersects the range of exactly one pairing. Truncate the pairing whose range contains the interval $[c,N']$.
		
	\end{enumerate}	

\begin{remark}
	Before starting the proof, we note that the reason for updating the stationary part of the weighted train track is that step (1) above may delete a pairing. This results in loss of data regarding some part of $\tau$. To remedy this, whenever step (1) deletes a paring, we preserve the data by adding the corresponding part of the train track to the stationary part. For example, if the carried 1-complex of $(\tau, \mu)$ is a pants decomposition $P$, then by successively applying the AHT algorithm all components of $P$ will eventually be removed, and so we keep a record of them by defining $\tau'_{\mathrm{s}}$. 
	\label{stationary part}
\end{remark}
	
Consider the first shifted cycle of the AHT algorithm applied to $\mathrm{OCP}(\tau_{a}, V_{a}, \mu_{a}, \mathcal{I}_{a})$. This consists of applying steps (1)--(4). First consider the combined effect of steps (1)--(2). In step (2), there could be no contractions unless we had deleted some pairing in step (1) that was the restriction of the identity. This happens only when $\tau$ has a component that is a simple closed curve. If this happens, then we add the resulting weighted simple closed curve to $(\tau_s, V_s, \mu_s)$. Note that steps (3)--(4) are not applicable: there is no trimming since $S$ is orientable, and periodic mergers are not possible since every point in $[1,N']$ appears in the domain and range of at most $2$ pairings. Hence the result holds for the first shifted cycle of the AHT algorithm.
	
Now consider a shifted cycle of the AHT algorithm consisting of steps (5)--(6) followed by steps (1)--(4). As argued before, steps (3)--(4) are not applicable.

First we show that step (5) includes a transmission. The union of the tie intervals above the base vertices $V_a$ is $[1,M]$ for some integer $M$ by the definition of the initial interval identification; see Definition \ref{orbit counting problem}. Furthermore $M \neq N'$ by the assumption that $\tau_a$ has at least one vertex besides the base vertices $V_a$. Hence, there is a vertex $w' \notin V_a$ such that the tie interval $t'$ above $w'$ includes $N'$. Let the top of $t'$ be the side containing $N'$. Then at least one of the two half-branches entering $t' \subset N(\tau_a, V_a, \mu_a)$ on the top side corresponds to a maximal pairing, allowing a transmission. 

Let $g_i$ be a maximal pairing, and $g_{j_1}, \cdots, g_{j_k}$ be the pairings whose range is contained in $[c_i, N']$. After re-indexing, we may assume that $\{j_1, \cdots, j_k\} = \{ 1, \cdots, k\}$. Let $g_i = g_e$ for a branch $e$ with initial vertex $w$ and terminal vertex $w' \notin V_a$. Let $e_-$ and $e_+$ be the half-branches of $e$ adjacent to $w$ and $w'$. Let $p_1, \cdots, p_\ell$ be the cusp points on $\text{range}(g_e) \subset t'$. See e.g. Figure \ref{fig:twirling} left, where there are two such cusp points. Note that $\ell \geq k$ with equality if and only if for each $1 \leq j \leq k$ we have $\text{domain}(g_j) \nsubseteq \text{range}(g_e)$.

For $1 \leq j \leq \ell$, denote by $l_j$ the leaf of the horizontal foliation on $N(\tau_a, V_a, \mu_a)$ containing $p_j$. Consider the following possibilities: 

\textit{Case a: If the domain and range of $g_e$ are disjoint.}

Denote by $R_e$ the rectangle corresponding to $e$ in the branched neighbourhood $N(\tau_a, V_a, \mu_a)$. Let $s_j$ be the connected component of $l_j|R_e$ containing $p_j$. Let $N(\tau'_a, V'_a, \mu'_a)$ be obtained by splitting $N(\tau_a, V_a, \mu_a)$ along $\cup_{j=1}^{\ell} s_j$. Then applying one shifted cycle of the AHT algorithm to $\mathrm{OCP}(\tau_a, V_a, \mu_a, \mathcal{I}_a)$ will result in $\mathrm{OCP}(\tau'_a, V'_a, \mu'_a, \mathcal{I}'_a)$ for some $\mathcal{I}'_a$. More precisely, for each $1 \leq j \leq k$, step (5) replaces $g_j$ with $g'_j := g_e^{-1} \circ g_j \circ g_e^{s}$, where $s \in \{ 0 ,1\}$ is equal to $1$ if and only if $\text{domain}(g_j) \subset \text{range}(g_e)$. Let $c := \min_{j=1}^{\ell} p_j \in [c_0, N'] = t$. Hence, after step (5) the points in the interval $[c,N']$ lie in the range of only $g_e$, and $c$ is minimal with this property. Subsequently, step (6) will truncate the pairing $g_e$. After step (6) there are no pairings that are restrictions of the identity, and so step (1) will have no effect. Finally step (2) will contract the static interval $[c, N']$. 

In this case $(\tau'_a, V'_a, \mu'_a)$ is obtained from $(\tau_a, V_a, \mu_a)$ via $\ell$ splits along the same branch $e$. Note that $\ell$ is at most twice the number of branches of $\tau$. 

\begin{remark}
	Note that item (2) in the definition of an initial interval identification is used here. This was to ensure that $N'$ lies on a tie interval above a vertex $w \notin V_a$; otherwise step (2) could have deleted a branch of the train track adjacent to a base vertex.
\end{remark}

\textit{Case b: If the domain and range of $g_e$ have non-empty intersection.}

In this case $w= w'$. Let $s_j$ be the connected component of $l_j|R_e$ containing $p_j$. If $l_j \neq \{ p_j\}$, then $s_j$ starts at $p_j$, runs horizontally across $R_e$ until it reaches the other vertical boundary of $R_e$, then it possibly enters $R_e$ again, and continues this for a finite number of times until it hits a tie interval in a point not included in $\text{range}(g_e)$. Let $N(\tau'_a, V'_a, \mu'_a)$ be obtained by splitting $N(\tau_a, V_a, \mu_a)$ along $\cup_{j=1}^{\ell} s_j$. Then applying Steps (5)--(6) of the AHT algorithm to $\mathrm{OCP}(\tau_a, V_a, \mu_a, \mathcal{I}_a)$ will result in $\mathrm{OCP}(\tau'_a, V'_a, \mu'_a, \mathcal{I}'_a)$ for some $\mathcal{I}'_a$. Next if we apply Steps (1)--(2), the entire collection of pairings corresponding to some components of $\tau$ might be deleted; any such component is necessarily a simple closed curve, because of item (2) in the definition of an initial interval identification. We add any such components to the stationary part. 

Let $\alpha$ be the simple closed curve that is the image of $e$ in $\tau$. Then $(\tau'_a, V'_a, \mu'_a)$ is obtained from $(\tau_a, V_a, \mu_a)$ via twirling along $\alpha$, and then possibly transferring some simple closed curve components to the stationary part. See Definition \ref{def:twirling} and Figure \ref{fig:twirling}. This finishes the proof of items (1)--(2). 

Now we prove item (3) following the proof of \cite[Lemma 2.9, part a]{dynnikov2007complexity}. Let $N$ be the maximum natural number appearing in the orbit counting problem $\mathrm{OCP}(\tau_a, V_a, \mu_a, \mathcal{I}_a)$; hence $N$ corresponds to an endpoint of a tie interval above a vertex $w \notin V_a$. There are exactly two pairings whose domain or range include $N$; let $z \geq \tilde{z}$ be the length of the intervals associated with these pairings; in particular $z$ is the length of $\text{range}(g_e)$. Define $z' \geq \tilde{z}'$ similarly for $\mathrm{OCP}(\tau'_a, V'_a, \mu'_a, \mathcal{I}'_a)$. Let $z_i$ be the measure of branches of $\tau$. Define $z'_j$ similarly. Recall that in both Case a and Case b, $N(\tau'_a, V'_a, \mu'_a)$ is obtained by splitting $N(\tau_a, V_a, \mu_a)$ along $\cup_{j=1}^{\ell} s_j$, and then possibly transferring some components to the stationary part. Assume that $p_\ell$ is the lowest point between all $p_j$ on the tie interval containing it; i.e. $p_\ell$ is the furthest from the end point $N$. There are two cases to consider:

\begin{enumerate}
	\item[i)] If $p_\ell$ is the minimum of $\text{range}(g_e)$, then the number of pairings will be reduced by one as a result of splitting. Hence 
	\begin{align*}
	 c_{\text{AHT}}^{\text{old}} - c_{\text{AHT}}^{\text{new}} &= \big( r+ \sum_{i=1}^{r} \log(z_i) - \frac{1}{2} \log(\tilde{z}) \big)  - \big( r-1 + \sum_{i=1}^{r-1} \log(z'_i) - \frac{1}{2} \log(\tilde{z}') \big) \\ &= 1 + \log(z) - \frac{1}{2} \log(\tilde{z}) + \frac{1}{2} \log(\tilde{z}') \geq 1.  
	 \end{align*}
	
	\item[ii)] If $p_{\ell}$ is not the minimum of $\text{range}(g_e)$, then the number of pairings will not increase after splitting, and it decreases only if some simple closed curve component of $\tau$ is transferred to the stationary part. Moreover, $\tilde{z}'$ will be equal to the length of $\text{range}(g_e)$ minus a whole positive multiple of the length of $ [p_\ell , N]$, where $[p_\ell, N]$ is the closed segment between $p_\ell$ and $N$ along the tie interval $I_w$. Hence $z \geq \tilde{z} + \tilde{z}'$. We have
	\begin{align*}  
		c_{\text{AHT}}^{\text{old}} - c_{\text{AHT}}^{\text{new}} &\geq   \big( r+ \sum_{i=1}^{r} \log(z_i) - \frac{1}{2} \log(\tilde{z}) \big)  - \big( r + \sum_{i=1}^{r} \log(z'_i) - \frac{1}{2} \log(\tilde{z}') \big) \\  & = \big( r+ \log(z) + \sum_{z_i \neq z} \log(z_i) - \frac{1}{2} \log(\tilde{z}) \big)  - \big( r + \sum_{z'_i \neq \tilde{z}'} \log(z'_i) + \frac{1}{2} \log(\tilde{z}') \big) \\  & \geq  \log(z) - \frac{1}{2} \log(\tilde{z}) -  \frac{1}{2} \log(\tilde{z}')  =  \frac{1}{2} \log(\frac{z^2}{\tilde{z} \tilde{z}'}) \geq 1, 
	\end{align*}
where we used the inequality $(a+b)^2 \geq 4ab$.
\end{enumerate}

This completes the proof of (3). Since the AHT-complexity is reduced by at least 1 at each step, the above algorithm terminates in time that is a polynomial function of $E$ and $\log|\mu|$.  

\end{proof}

\begin{remark}
	The operations of splitting and twirling do not change the 1-complex carried by a based integrally weighted train track. Hence, if $(\tau', V', \mu')$ is obtained from $(\tau, V, \mu)$ via applying a shifted cycle of the AHT algorithm as in Proposition \ref{running AHT}, then the carried 1-complexes $\mathcal{CC}(\tau', V', \mu')$ and $\mathcal{CC}(\tau, V, \mu)$ are equal. Therefore, if $\tau'$ has no vertices besides the base vertices $V$ and no complementary component to $\mathcal{CC}(\tau, V, \mu)$ is a bigon, then $\tau' = \mathcal{CC}(\tau, V, \mu)$.
	\label{complete unwinding}
\end{remark}

The reduction of AHT-complexity in part (3) of Proposition \ref{running AHT} allows us to control the number of shifted cycles needed to completely simplify a based integrally weighted train track.

\begin{lem}[Bound on the number of moves for simplifying a weighted train track using the AHT algorithm]
	Let $(\tau, V, \mu)$ be as in Proposition \ref{running AHT}. Assume that no complementary component to the carried 1-complex $\mathcal{CC}(\tau, V, \mu)$ is a bigon. Apply shifted cycles of the AHT algorithm repeatedly to obtain a sequence of based integrally weighted train tracks $(\tau_i, V, \mu_i)$ for $i = 1, \cdots, n$ with $(\tau_1, V, \mu_1) = (\tau, V, \mu)$ and $\tau_n= \mathcal{CC}(\tau, V , \mu)$. Let $|\mu|$ be the sum of the weights of all branches of $\tau$, and $E$ be the number of branches of $\tau$. Then the number $n$ of the train tracks is at most
	\[  1+ E + E \log|\mu|. \] 
	\label{number of train tracks}
\end{lem}

\begin{proof}
	The number $n-1$ is bounded above by the number of shifted cycles of the AHT algorithm applied to the orbit counting problem $\mathrm{OCP}(\tau, V, \mu, \mathcal{I})$. Let $c_i$ be the AHT-complexity of $\mathrm{OCP(\tau_i, V, \mu_i, \mathcal{I}_i)}$; see Definition \ref{def:AHT-complexity}. By part 3) of Proposition \ref{running AHT} we have $0 \leq c_{i+1} \leq c_i -1$ for each $i$, hence the number of shifted cycles is at most the value of $c_1$. We have 
	\[ c_1 = r + \sum_{i=1}^{r} \log(z_i) - \frac{1}{2} \log(\tilde{z}) \leq E + E \log|\mu|. \] 
\end{proof}

\section{Upper bound for distance in the pants graph}

\theoremstyle{theorem}
\newtheorem*{main}{Theorem \ref{thm:upper bound for distance in pants graph}}
\begin{main}
	Let $S$ be a compact connected orientable surface, and $P$ and $P'$ be pants decompositions of $S$ with $i(P, P') \geq 2$. The distance between $P$ and $P'$ in the pants graph is
	\[ O(|\chi(S)|^2) \hspace{1mm}  \log (i(P,P')). \] 
	
	Moreover, there is an algorithm that constructs a (not necessarily geodesic) path $P = P_0, P_1, \cdots, P_n = P'$ of length $ n = O(|\chi(S)|^2) \hspace{1mm} \log (i(P,P'))$ connecting $P$ and $P'$ in the pants graph, in time that is a polynomial function of $|\chi(S)|$ and $\log(i(P,P'))$. Here, for the input, $P$ is given as a union of pairs of pants with gluing instructions, and $P'$ is given in normal form with respect to $P$. For the output, each $P_i$ is given as a union of pairs of pants with gluing instructions together with the pants move from $P_{i-1}$ to $P_{i}$. 
	
\end{main}

\begin{proof}

\textbf{Step 1:} Assume that $P'$ is in normal form with respect to $P$. By squeezing parallel (in $S - P$) arcs of $P'$ together, we obtain an integrally weighted train track $(\tau, \mu)$ with vertices lying on $P$ and with the carried $1$-complex $\mathcal{CC}(\tau,\mu)$ isotopic to $P'$. Each component of $S - P$ contains at most $3$ branches of $\tau$, and so the number of branches of $\tau$ is at most $3|\chi(S)|$. Moreover, the total weight $|\mu|$ of all branches of $\tau$ is equal to $i(P, P').$


\textbf{Step 2}: Pick any initial interval identification $\mathcal{I}$ for $(\tau, \mu)$, and repeatedly apply shifted cycles of the AHT algorithm to the orbit counting problem $\mathrm{OCP}(\tau, \mu, \mathcal{I})$. Proposition \ref{running AHT} describes the geometric effect of each shifted cycle of the AHT algorithm on an integrally weighted train track: this is either at most $O(|\chi(S)|)$ splittings along the same oriented branch or a twirling. We apply shifted cycles of the AHT algorithm until we obtain the carried complex $\mathcal{CC}(\tau, \mu)$, see Remark \ref{complete unwinding} to see why the output of the algorithm is the carried 1-complex. This is possible since the AHT-complexity is reduced by at least 1 after applying each shifted cycle; see Proposition \ref{running AHT}, item (3). The algorithm returns the carried 1-complex $\mathcal{CC}(\tau, \mu)$ which we know is equal to $P'$, but we are interested in the way that the algorithm computes the output rather than the output itself. By Proposition \ref{running AHT}, there is a sequence of integrally weighted train tracks $(\tau_1, \mu_1), \cdots, (\tau_n, \mu_n)$ with the following properties.
\begin{enumerate}
	\item $(\tau_1, \mu_1)=(\tau, \mu)$, and $\tau_n  = P'$. 
	
	\item The underlying 1-complex of each $\tau_i$ is a pre-triangulation. This can be seen as follows. By Proposition \ref{running AHT}, each $(\tau_{i+1}, \mu_{i+1})$ is obtained from $(\tau_i, \mu_i)$ by either a number of splittings and/or a twirling. By Lemma \ref{underlying 1-complex}, the property of being a pre-triangulation pulls back under splitting. By Lemma \ref{effect of twirling on the underlying 1-complex}, the effect of twirling on the underlying train track is applying a power of a Dehn twist, which keeps the property of being a pre-triangulation invariant, and a number of splittings. Therefore, we have shown that if $\tau_{i+1}$ is a pre-triangulation, then so is $\tau_i$. Since $\tau_n = P'$ is a pre-triangulation, it follows inductively that each $\tau_i$ is a pre-triangulation. 	
	
	Set $k_i = \chi(\tau_{i+1}) - \chi(\tau_i) \geq 0$. By Proposition \ref{lem:local distance}, there are orderings $\mathrm{O}_i$ on the edges of $\tau_i$ such that for each $i$ the pants decompositions $\mathrm{\textbf{P}}(\tau_i, \mathrm{O}_i)$ and $\mathrm{\textbf{P}}(\tau_{i+1}, \mathrm{O}_{i+1})$ have distance $O((k_i+1) |\chi(S)|)$ in the pants graph. Note that 
	\[ \sum_{i=1}^{n-1} k_i = \chi(\tau_n) - \chi(\tau_1) = O(|\chi(S)|), \]
 	and so if $d(\cdot , \cdot)$ denotes the distance in the pants graph, then
 	\[  \sum_{i=1}^{n-1} d(\mathrm{\textbf{P}}(\tau_i, \mathrm{O}_i), \mathrm{\textbf{P}}(\tau_{i+1}, \mathrm{O}_{i+1})) = (n-1) O(|\chi(S)|) + O(|\chi(S)|^2). \]
	
 	\item By Lemma \ref{number of train tracks}, the number $n$ of the train tracks is 
 	\[ 1+ O(|\chi(S)|) \log (i(P, P')). \]
 	
\end{enumerate}

\textbf{Step 3}: 
Recall that $\tau_n$ has no vertices and so $\mathrm{\textbf{P}}(\tau_n, \mathrm{O}_n) = \mathcal{CC}(\tau_n, \mu_n) = P'$. Applying the triangle inequality we have the following
\begin{align*}
	& d(P, P') \leq d(P, \mathrm{\textbf{P}}(\tau_1, \mathrm{O}_1)) + \sum_{i=1}^{n-1} d(\mathrm{\textbf{P}}(\tau_i, \mathrm{O}_i), \mathrm{\textbf{P}}(\tau_{i+1}, \mathrm{O}_{i+1})) \\
	& \leq d(P, \mathrm{\textbf{P}}(\tau_1, \mathrm{O}_1)) + (n-1)O(|\chi(S)|)+O(|\chi(S)|^2).
\end{align*}
Note that by Lemma \ref{bound in terms of intersection}, the pants decompositions $\mathrm{\textbf{P}}(\tau_1, \mathrm{O}_1)$ and $P$ have distance $O(|\chi(S)|^2)$, since $i(P, \tau_1) = O(|\chi(S)|)$. This gives an upper bound for $d(P,P')$ of the form 
\[ O(|\chi(S)|^2) + O(|\chi(S)|^2) \hspace{1mm} \log (i(P,P')).\] 
Since we assumed that $i(P,P') \geq 2$, we have 
\[ |\chi(S)|^2 \hspace{1mm} \log (i(P,P')) \geq \log2 \cdot |\chi(S)|^2. \]
Therefore, the upper bound for $d(P, P')$ can equivalently be written as 
\[ O(|\chi(S)|^2) \hspace{1mm}  \log (i(P,P')). \]

We now discuss the algorithmic and computational part of the statement. By Proposition \ref{running AHT}, the sequence $(\tau_i, \mu_i)$ for $1 \leq i \leq n$ can be constructed in time that is a polynomial function of $|\chi(S)|$ and $\log(i(P, P'))$. By Proposition \ref{lem:local distance}, given $\mathrm{O}_{i}$ one can construct $\mathrm{O}_{i+1}$ together with a sequence of at most $O((k_{i}+1)|\chi(S)|)$ consecutive transpositions taking $\mathrm{O}_{i}$ to $\mathrm{O}_{i+1}$, in time that is a polynomial function of $|\chi(S)|$. For each pair $(\tau_i, \mathrm{O}_{i})$, the pants decomposition $\mathrm{\textbf{P}}(\tau_i, \mathrm{O}_i)$ can be constructed in time that is polynomial in $|\chi(S)|$. Moreover, given two pairs $(\tau_{i}, \mathrm{O}_i)$ and $(\tau_{i+1}, \mathrm{O}_{i+1})$, one can construct a sequence of pants moves of length $O((k_i+1)|\chi(S)|)$ taking $\mathrm{\textbf{P}}(\tau_i, \mathrm{O}_i)$ to $\mathrm{\textbf{P}}(\tau_{i+1}, \mathrm{O}_{i+1})$, in time that is polynomial in $|\chi(S)|$. This completes the proof of the algorithmic part of the statement.

\end{proof}

\section{Polygonal decompositions}

In this section, we prove Theorem \ref{thm: one-vertex triangulations, variable vertices}, which will be used to relate two polygonal decompositions of a surface.

\begin{definition}
\label{Def:DiagonalInsertion}
Let $S$ be a compact surface, and let $\mathcal{T}$ be a polygonal decomposition that is disjoint from $\partial S$. Let $e$ be an arc embedded in the interior of $S$ with both endpoints being vertices of $\mathcal{T}$ and $e \cap \mathcal{T} = \partial e$. Then the polygonal decomposition $\mathcal{T} \cup e$ is obtained from $\mathcal{T}$ by \emph{adding the diagonal} $e$. We say that $\mathcal{T}$ is obtained from $\mathcal{T} \cup e$ by \emph{deleting the diagonal $e$}.
\end{definition}

\begin{thm}
	Let $S$ be a compact orientable surface, and $\mathcal{T}$ and $\mathcal{T}'$ be polygonal decompositions of $S$ disjoint from $\partial S$ with at most $E$ edges. There is a sequence $\mathcal{T}= \mathcal{T}_0, \mathcal{T}_1, \cdots, \mathcal{T}_n = \mathcal{T}'$ of polygonal decompositions of $S$ such that:
	
	\begin{enumerate}
		\item Each $\mathcal{T}_{i+1}$ is obtained from $\mathcal{T}_i$ by contracting or expanding an embedded edge, or adding or deleting a diagonal, or a power of a Dehn twist. When $\mathcal{T}_{i+1}$ is obtained from $\mathcal{T}_i$ by Dehn twisting $k$ times about a curve $\alpha$, then $\alpha$ is the closure of an edge (necessarily with the same endpoints) of $\mathcal{T}_i$, and the absolute value of $k$ is bounded above by $i(\mathcal{T}, \mathcal{T}')$.  
		
		\item Each $\mathcal{T}_i$ is disjoint from $\partial S$.
		
		\item The number of vertices of each $\mathcal{T}_i$ is $O(E)$.
		
		\item The number $n$ of steps in this sequence is $O(E^2\log(i(\mathcal{T}, \mathcal{T}')) + E^2)$. 
	\end{enumerate}

	Moreover, there is an algorithm that constructs the sequence $\mathcal{T}_1, \cdots, \mathcal{T}_n$ in time that is a polynomial function of $\log(i(\mathcal{T}, \mathcal{T}'))$ and $E$. For the input, $\mathcal{T}$ is given as a union of polygonal discs and annuli with gluing instructions, and $\mathcal{T}'$ is given by its normal coordinates with respect to $\mathcal{T}$. For the output, each $\mathcal{T}_i$ is given as a union of polygonal discs and annuli with gluing instructions together with the move from $\mathcal{T}_i$ to $\mathcal{T}_{i+1}$.
	 
	 \label{thm: one-vertex triangulations, variable vertices}
\end{thm}

\begin{remark}
	Note that contracting or expanding edges and Dehn twists are not enough to go between two arbitrary polygonal decompositions of $S$, since they preserve the number of complementary regions. Similarly, adding or deleting diagonals and Dehn twists preserve the number of vertices. 
\end{remark}

\begin{lem}
	Let $\mathcal{T}$ and $\mathcal{T}'$ be polygonal decompositions of a compact surface $S$ that are disjoint from $\partial S$. Let $|\mathcal{T}\cup \mathcal{T}'|$ be the number of edges of the polygonal decomposition $\mathcal{T} \cup \mathcal{T}'$ obtained by superimposing $\mathcal{T}$ and $\mathcal{T}'$. Then $\mathcal{T}$ can be transformed into $\mathcal{T}'$ using $O(|\mathcal{T}\cup \mathcal{T}'|)$ contracting or expanding embedded edges, and $O(|\mathcal{T}\cup \mathcal{T}'|)$ adding or deleting diagonals. 
	
	Moreover, there is an algorithm that produces such sequence of moves in time that is polynomial in $|\mathcal{T} \cup \mathcal{T}'|$. Here $\mathcal{T}$ is given as a union of polygonal discs and annuli with gluing instructions, and $\mathcal{T}'$ is given in normal form with respect to $\mathcal{T}$. The output gives a sequence of polygonal decompositions (described as a union of polygonal discs and annuli with gluing instructions) interpolating between $\mathcal{T}$ and $\mathcal{T}'$ together with the moves between consecutive polygonal decompositions. 

	\label{distance between polygonal decompositions in terms of intersection number}
\end{lem}

\begin{proof}
	It is enough to show that $\mathcal{T}\cup \mathcal{T}'$ can be transformed into each of $\mathcal{T}$ and $\mathcal{T}'$ using $O(|\mathcal{T}\cup \mathcal{T}'|)$ such moves. For any complementary region $R$ of $\mathcal{T}$, remove all edges of $\mathcal{T} \cup \mathcal{T}'$ that lie in the interior of $R$ by contracting embedded edges and deleting diagonals. Then remove any remaining extra vertices of $\mathcal{T}\cup \mathcal{T}'$ that lie on $\mathcal{T}$ by contracting embedded edges. It is easy to see that this can be done in polynomial time as a function of $|\mathcal{T}\cup \mathcal{T}'|$.
\end{proof}

\begin{proof}[Proof of Theorem \ref{thm: one-vertex triangulations, variable vertices}] 
	Let $V$ be the set of vertices of $\mathcal{T}'$. First we construct a based integrally weighted train track $(\tau, V, \mu)$ such that the carried 1-complex $\mathcal{CC}(\tau, V, \mu)$ is equal to $\mathcal{T}'$. To do this, for each edge of $\mathcal{T}$ squeeze together all points of intersection with $\mathcal{T}'$ to a single switch of $\tau$, and then for each set of parallel normal arcs of $\mathcal{T}'$ squeeze them together to a single edge of $\tau$; this gives the based integrally weighted train track $(\tau, V, \mu)$. See Figure \ref{fig:squeezing parallel arcs} for an example, where hollow dots show the vertices of $\mathcal{T}'$ which in turn will correspond to base vertices of the associated train track. 
	
	\begin{figure}
		\labellist
		\pinlabel $2$ at 265 85
		\pinlabel $2$ at 313 50
		\pinlabel $2$ at 330 25
		\endlabellist
		\includegraphics[width = 3 in]{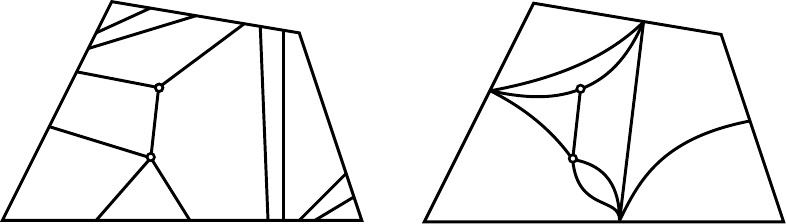}
		\caption{Constructing a based integrally weighted train track by squeezing parallel normal arcs together. On the left the intersection of the polygonal decomposition $\mathcal{T}'$ with a face of the polygonal decomposition $\mathcal{T}$ is shown. On the right, the associated based integrally weighted track track is shown, any edge with no weight written on it has weight 1. The midpoints of the polygon will then correspond to non-base vertices of the train track. }
		\label{fig:squeezing parallel arcs}
	\end{figure}
	
	Apply shifted cycles of the AHT algorithm to unwind $(\tau, V, \mu)$ to $\mathcal{T}'$. By Proposition \ref{running AHT} we obtain a sequence of based integrally weighted train tracks $(\tau_i, V, \mu_i)$ for $1 \leq i \leq n$ such that 
	
	\begin{enumerate}
		\item $(\tau_1, V, \mu_1) = (\tau, V, \mu)$, and $(\tau_n, V, \mu_n)= (\mathcal{T}', V, \textbf{1})$ where $\textbf{1}$ denotes the weight assigning $1$ to every branch.
		
		\item Each $(\tau_{i+1}, V, \mu_{i+1})$ is obtained from $(\tau_i, V, \mu_i)$ by either splits along the same oriented branch, or by a twirling. 
		
		\item The number $n$ is $O(E \log(i(\mathcal{T}, \mathcal{T}')) + E)$. This follows from Lemma \ref{number of train tracks} and the following two observations:
		\begin{enumerate}
			\item[-] $|\mu|$ is at most $i(\mathcal{T}, \mathcal{T}') + |E|$. 
			\item[-] The number of edges of $\mathcal{T}$ is at most $E$, and so the number of branches of the based train track $(\tau, V)$ is $O(E)$ as well.  
		\end{enumerate}
		Lemma \ref{number of train tracks} then implies that $n$ is at most $O(E) + O(E) \log (i(\mathcal{T}, \mathcal{T}') + |E|)$. When $i(\mathcal{T}, \mathcal{T}') \geq |E|$, then $i(\mathcal{T}, \mathcal{T}') + |E| \leq 2 i(\mathcal{T}, \mathcal{T}')$, and so the upper bound on $n$ becomes $O(E \log(i(\mathcal{T}, \mathcal{T}')) + E)$, as required. On the other hand, when $i(\mathcal{T}, \mathcal{T}') < |E|$, Lemma \ref{distance between polygonal decompositions in terms of intersection number} implies that we can find a sequence of polygonal decompositions with length $O(E)$, which is at most the required bound.
	
		\item The sequence $(\tau_i, V, \mu_i)$ for $1 \leq i \leq n$ can be constructed in time that is a polynomial function of $\log (i(\mathcal{T}, \mathcal{T}'))$ and $E$.


\end{enumerate}

By Lemma \ref{distance between polygonal decompositions in terms of intersection number}, there is a sequence of $O(E)$ addition or deletion of diagonals that takes $\tau_1$ to $\mathcal{T}$ and this sequence can be constructed in time that is polynomial in $E$. This is because $i(\mathcal{T}, \tau_1)$ is $O(E)$ and so superimposing $\tau_1$ and $\mathcal{T}$ will produce a 1-complex with $O(E)$ edges.

Additionally, each $\tau_{i+1}$ is obtained from $\tau_i$ by $O(E)$ contractions or expansions of embedded edges, $O(E)$ addition or deletion of diagonals, and at most one twist map. This again follows from Lemma \ref{distance between polygonal decompositions in terms of intersection number} since if $\tau_{i+1}$ is obtained from $\tau_i$ by splits along the same oriented branch, then after a suitable isotopy the 1-complex $\tau_i \cup \tau_{i+1}$ obtained by superimposing them has $O(E)$ edges. In this case, by Lemma \ref{distance between polygonal decompositions in terms of intersection number} such a sequence of moves from $\tau_i$ to $\tau_{i+1}$ can be constructed in time that is polynomial in $E$. When $\tau_{i+1}$ is obtained from $\tau_i$ by a twist map, then the curve that we twist about and the power of the Dehn twist are already given to us by Proposition \ref{running AHT}.

Combining the sequence of moves 

\[  \mathcal{T} \rightarrow \tau_1 \rightarrow \tau_2 \rightarrow \cdots \rightarrow \tau_n= \mathcal{T}'  \]
gives the desired sequence of polygonal decompositions from $\mathcal{T}$ to $\mathcal{T}'$ of length 
\[n = O(E^2 \log(i(\mathcal{T}, \mathcal{T}')) + E^2).\]	 
\end{proof}

\begin{notation}
	Let $S$ be a surface and let $\Gamma$ be an embedded finite 1-complex in $S$. Then $S \setminus \setminus \Gamma$ denotes the surface obtained by cutting $S$ along $\Gamma$. In other words, equipping $S$ with a Riemannian metric, then $S \setminus \setminus \Gamma$ is the metric completion of $S - \Gamma$ with respect to the induced metric on $S - \Gamma$. 
\end{notation}

We now consider an alternative version of Theorem \ref{thm: one-vertex triangulations, variable vertices} for the following important type of polygonal decompositions.

\begin{definition}
A \emph{spine} for a closed surface $S$ is an embedded 1-complex $\Gamma \subset S$ such that the surface $S \setminus \setminus \Gamma$ obtained by cutting $S$ along $\Gamma$ is a disc.  
A \emph{spine} for a compact surface $S$ with non-empty boundary is a 1-complex $\Gamma$ embedded in the interior of $S$ such that the surface $S \setminus \setminus \Gamma$ is a regular neighbourhood of $\partial S$. See Figure \ref{fig:spine}.
\end{definition}

\begin{figure}
	\includegraphics[width = 4 in]{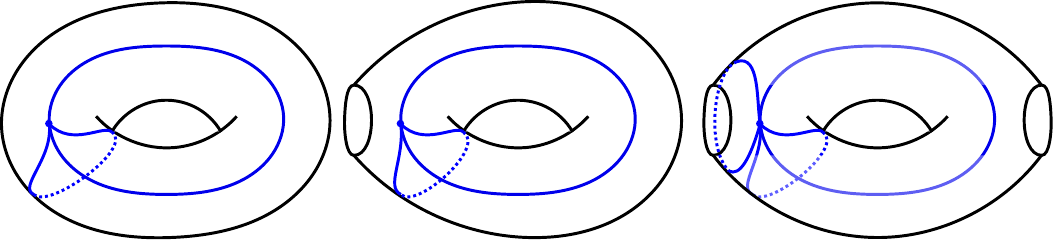}
	\caption{An example of a spine for a torus (left), a surface of genus one with one boundary component (middle), and a surface of genus one with two boundary components.}
	\label{fig:spine}
\end{figure}

The following is one natural way of modifying spines.

\begin{definition}[Edge swap \cite{lackenby2019triangulation}]
	Let $\Gamma$ be a spine for a compact surface $S$. Let $e'$ be an arc properly embedded in $S \setminus \setminus \Gamma$ with endpoints on $\Gamma$. Let $e$ be an edge of the graph $\Gamma \cup e'$ that has distinct components of $S \setminus \setminus (\Gamma \cup e')$ on either side of it, at least one of which is a disc. Then the result of removing $e$ from $\Gamma$ and adding $e'$ is a new spine $\Gamma'$ for $S$. We say that $\Gamma$ and $\Gamma'$ are related by an \emph{edge swap}. See Figure \ref{fig: edge swap}.
\end{definition}

\begin{figure}
	\labellist
	\pinlabel $e'$ at 58 50 
	\pinlabel $e$ at 3 38 
	\pinlabel $e$ at 108 50 
	\endlabellist

	\includegraphics[width = 1.2 in]{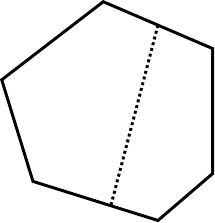}
	\caption{An edge swap on a spine $\Gamma$ of $S$. Here the polygon represents $S \setminus \setminus \Gamma$.}
	\label{fig: edge swap}
\end{figure}

\begin{example}
	Let $\Gamma$ be a spine for the torus shown in Figure \ref{fig:edge swap example} left. Here $\Gamma$ (shown in blue) is the union of all edges except for $e'$ (shown in red), and the vertices of $\Gamma$ are shown with solid dots. The complement of $\Gamma$ is a 6-gon, and so $\Gamma$ is a spine. Swapping the edge $e$ with $e'$ results in the spine shown on the right hand side of the same Figure, whose complement is an 8-gon in this case. Note that in an edge swap, the endpoints of $e'$ need not be vertices of $\Gamma$, as seen in this example. 
	
	\begin{figure}
		\labellist 
		\pinlabel $e$ at 80 10
		\pinlabel $e'$ at 53 70
		\endlabellist
		\includegraphics[width = 3 in]{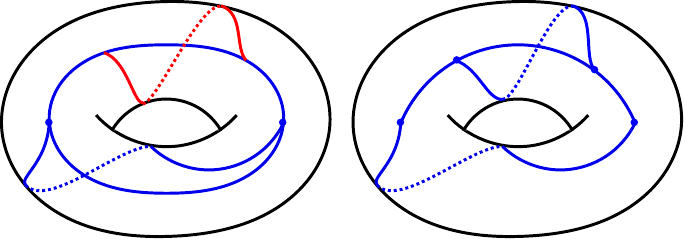}
		\caption{Edge swap: the edge $e$ of the spine (in blue) is replaced with the edge $e'$ (in red). The new spine is shown on the right. }
		\label{fig:edge swap example}
	\end{figure}
\end{example}

\begin{thm}
	Let $S$ be a compact orientable surface, and let $\Gamma$ and $\Gamma'$ be spines of $S$. There is a sequence $\Gamma = \Gamma_0, \Gamma_1, \cdots, \Gamma_n = \Gamma'$ of spines for $S$ such that:
	
	\begin{enumerate}
		\item Each $\Gamma_{i+1}$ is obtained from $\Gamma_i$ by an edge swap, an expansion or contraction of an embedded edge, or a power of a Dehn twist. When $\Gamma_{i+1}$ is obtained from $\Gamma_i$ by Dehn twisting $k$ times about a curve $\alpha$, then $\alpha \cap \Gamma_i$ is a vertex of $\Gamma_i$, and the absolute value of $k$ is bounded above by $i(\Gamma, \Gamma')$.
				
		\item The number $n$ of steps in this sequence is $O(\chi(S)^2\log(i(\Gamma, \Gamma')) + \chi(S)^2)$. 
	\end{enumerate}

	Moreover, there is an algorithm that constructs the sequence $\Gamma_0, \cdots, \Gamma_n$ in time that is a polynomial function of $\log (i(\Gamma, \Gamma'))$ and $\chi(S)$. For the input, $\Gamma$ is given as a polygon or a collection of annuli with gluing instructions, and $\Gamma'$ is given by its normal coordinates with respect to $\Gamma$. For the output, each $\Gamma_i$ is given as a polygon or annuli with gluing instructions together with the move from $\Gamma_i$ to $\Gamma_{i+1}$.
	 
	 \label{thm: edge swaps on spines}
\end{thm}

\begin{proof}
By Theorem \ref{thm: one-vertex triangulations, variable vertices}, there is a sequence of polygonal decompositions 
	\[ \Gamma = \mathcal{P}_0, \mathcal{P}_1, \cdots, \mathcal{P}_m = \Gamma' \] 
	such that:
	
	\begin{enumerate}
		\item Each $\mathcal{P}_{i+1}$ is obtained from $\mathcal{P}_i$ by contracting or expanding an embedded edge, adding or deleting a diagonal, or a power of a Dehn twist. When $\mathcal{P}_{i+1}$ is obtained from $\mathcal{P}_i$ by Dehn twisting $k$ times about a curve $\alpha$, then $\alpha$ is the closure of an edge of $\mathcal{P}_i$, and the absolute value of $k$ is bounded above by $i(\Gamma, \Gamma')$.
		
		\item Each $\mathcal{P}_i$ is disjoint from $\partial S$.
		
		\item The number of vertices of each $\mathcal{P}_i$ is $O(|\chi(S)|)$.
		
		\item \[ m=O(\chi(S)^2 \log(i(\Gamma, \Gamma')) + \chi(S)^2). \]
		 
		\item There is an algorithm that constructs the sequence $\mathcal{P}_1, \cdots, \mathcal{P}_m$ together with the moves from $\mathcal{P}_i$ to $\mathcal{P}_{i+1}$, in time that is a polynomial function of $\log(i(\Gamma, \Gamma'))$ and $\chi(S)$.
	\end{enumerate}

For each polygonal decomposition $\mathcal{P}_i$, we will pick a collection of edges $X_i$ in $\mathcal{P}_i$ so that $\mathcal{P}_i \setminus X_i$ is a spine $\Gamma_i$. 	
Initially we will assume that $S$ is closed.

We will construct $X_i$ recursively. Initially, $X_0 = \emptyset$. To define $X_{i+1}$ from $X_i$, we consider various cases: 

	\begin{enumerate}
		\item[(i)] If $\mathcal{P}_i \rightarrow \mathcal{P}_{i+1}$ is the insertion of a diagonal $e$, we define $X_{i+1}$ to be $X_i \cup \{ e \}$. So $\Gamma_{i+1} = \Gamma_i$.
		
		\item[(ii)] Suppose that $\mathcal{P}_i \rightarrow \mathcal{P}_{i+1}$ is the removal of a diagonal $e$. If $e$ is in $X_i$, then set $X_{i+1}$ to be $X_i \setminus \{ e \}$. Then $\Gamma_{i} = \mathcal{P}_i \setminus X_i = \mathcal{P}_{i+1} \setminus X_{i+1} = \Gamma_{i+1}$. So we consider when $e$ does not lie $X_i$. Since we are assuming that $S$ is closed, $S \setminus \Gamma_i$ is a single disc, and so there is an arc in this disc region joining the two sides of $e$. However, in $S \setminus \mathcal{P}_i$, the two sides of $e$ lie in distinct disc regions. So there is an arc $\alpha$ in $X_i$, consisting of a union of edges, so that $\alpha$ lies in the boundaries of both these disc regions. See Figure \ref{fig:from-polygonal-decomposition-to-spine}. Let $X_{i+1}$ equal $X_i$ minus the edges in $\alpha$. Then $\Gamma_{i+1}$ is obtained from $\Gamma_i$ by removing $e$ and adding this arc $\alpha$. In other words, $\Gamma_{i+1}$ is obtained from $\Gamma_i$ by performing an edge swap.
		
		\begin{figure}
			\labellist 
			\pinlabel $e$ at 100 100
			\pinlabel $\alpha$ at 55 10 
			\pinlabel $\alpha$ at 180 215
			
			\endlabellist
			\includegraphics[width= 1.5 in]{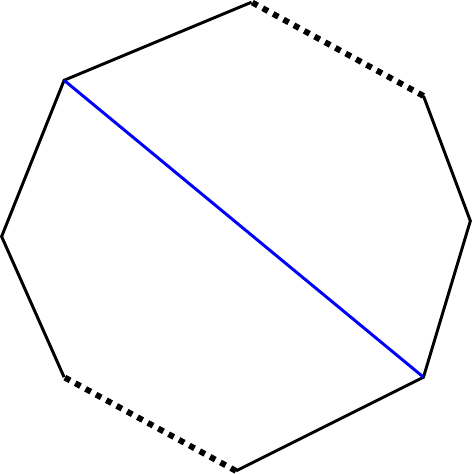}
			\caption{Proof of Theorem \ref{thm: edge swaps on spines}, case (ii): The edge $e$ is a diagonal of the disc $S \setminus \Gamma_i$, and $\alpha \subset X_i$ is a union of edges.}
			\label{fig:from-polygonal-decomposition-to-spine}
		\end{figure}
		
		\item[(iii)] If $\mathcal{P}_i \rightarrow \mathcal{P}_{i+1}$ is the expansion of an edge $e$, we set $X_{i+1}$ to be $X_i$, unless this creates a 1-valent vertex of $\Gamma_{i+1}$, in which case we set $X_{i+1} = X_i \cup \{ e \}$. So, $\Gamma_{i+1}$ is obtained from $\Gamma_i$ by expanding the edge $e$, or is a copy of $\Gamma_i$.
		
		\item[(iv)] Suppose that $\mathcal{P}_i \rightarrow \mathcal{P}_{i+1}$ is the contraction of an embedded edge $e$. If $e$ is not in $X_i$, then we set $X_{i+1} = X_i$ and we note $\Gamma_{i+1}$ is obtained from $\Gamma_i$ by contracting $e$. So suppose that $e$ is in $X_i$. If we were to add $e$ to $\Gamma_i$, the complementary disc $S \setminus \Gamma_i$ would be divided into two discs, $A$ and $B$ say. Moreover, the complementary regions incident to one endpoint of $e$ are not all in $A$ or $B$. Hence, there is some edge $e'$ of $\Gamma_i$ emanating from this endpoint of $e$ that has $A$ on one side and $B$ on the other. Let $X'_i = X_i \cup \{e' \} \setminus \{ e \}$. Let $\Gamma'_i = \mathcal{P}_i \setminus X'_i$. Then $\Gamma_i'$ is obtained from $\Gamma_i$ by an edge swap. Moreover, we may set $X_{i+1}$ to be $X'_i$. Then $\Gamma_{i+1}$ is obtained from $\Gamma_i'$ by contracting $e$.
		
		\item[(v)] Suppose $\mathcal{P}_i \rightarrow \mathcal{P}_{i+1}$ is obtained by Dehn twisting about a curve $\alpha$ that is the closure of an embedded edge. Then we may perturb $\alpha$ so that it intersects $\mathcal{P}_i$ in the vertex at the endpoints of $\alpha$. Hence, $\alpha$ intersects $\Gamma_i$ in exactly this vertex. We let $X_{i+1}$ be the image of $X_i$ under these Dehn twists.
	\end{enumerate}

Suppose now that $\partial S$ is non-empty. Cases (i), (iii) and (v) are identical to the situation where $S$ is closed.

We now explain how (ii) is modified in the case where $\partial S$ is non-empty. Again suppose that $\mathcal{P}_i \rightarrow \mathcal{P}_{i+1}$ is the removal of a diagonal $e$. Again, the difficult case is where $e$ does not lie in $X_i$, and hence lies in $\Gamma_i$. If the two sides of $e$ lie in the same annular region of $S \setminus \setminus \Gamma_i$, then there must be an arc $\alpha$ consisting of a union of edges of $X_i$ in this annulus separating these two copies of $e$ in the boundary of the annulus. In that case, we set $X_{i+1}$ to be equal to $X_i$ minus the edges of $\alpha$, and then $\Gamma_{i+1}$ is obtained from $\Gamma_i$ by an edge swap. So suppose that the two sides of $e$ lie in distinct annular regions of $S \setminus \setminus \Gamma_i$. One side of $e$ lies in a disc component $D$ of $S \setminus \setminus \mathcal{P}_i$. Say that this lies in a component $A$ of $S \setminus \setminus \Gamma_i$. Then the edges in $X_i$ separate $D$ from $\partial S \cap A$. Hence, there is an arc $\alpha$ properly embedded in $A$ consisting of a union of edges of $X_i$ that separates $D$ from $\partial S \cap A$. Again set $X_{i+1}$ to be equal to $X_i$ minus the edges of $\alpha$, and again $\Gamma_{i+1}$ is obtained from $\Gamma_i$ by an edge swap.

The argument in case (iv) is very similar to the case when $S$ is closed. Suppose that $\mathcal{P}_i \rightarrow \mathcal{P}_{i+1}$ is the contraction of an embedded edge $e$. The difficult situation is where $e$ is in $X_i$. Then $e$ is an arc properly embedded in an annulus component of $S \setminus \setminus \Gamma_i$. It is disjoint from $\partial S$, and hence it separates the annulus into an annulus and a disc $D$. Emanating from the vertex at one endpoint of $e$, there is an edge $e'$ that has $D$ on one side and an annulus of $S \setminus \setminus (\Gamma_i \cup e)$ on the other. Let $X'_{i} = X_i \cup \{ e' \} \setminus \{ e \}$ and let $\Gamma'_i = \mathcal{P}_i \setminus X'_i$. Then $\Gamma'_{i}$ is obtained from $\Gamma_i$ by an edge swap. Setting $X_{i+1}$ to be $X_i'$, $\Gamma_{i+1}$ is obtained from $\Gamma_i$ by the contraction of $e$.

\end{proof}

\section{One-vertex and ideal triangulations}

In this section, we improve Theorem \ref{thm: one-vertex triangulations, variable vertices} by showing that one can stay within the class of one-vertex triangulations or ideal triangulations.

\theoremstyle{theorem}
\newtheorem*{triangulations}{Theorem \ref{thm: one-vertex triangulations, fixed vertices}}
\begin{triangulations} 
	Let $S$ be a compact orientable surface. When $S$ is closed (respectively, has non-empty boundary), let $\mathcal{T}$ and $\mathcal{T}'$ be one-vertex (respectively, ideal) triangulations of $S$. Assume that $i(\mathcal{T}, \mathcal{T}') \geq 2$. There is a sequence $\mathcal{T}= \mathcal{T}_0, \mathcal{T}_1, \cdots, \mathcal{T}_n = \mathcal{T}'$ of one-vertex (respectively, ideal) triangulations of $S$ such that: 
	
	\begin{enumerate}
		\item Each $\mathcal{T}_{i+1}$ is obtained from $\mathcal{T}_i$ by either a flip or a power of a Dehn twist. When $\mathcal{T}_{i+1}$ is obtained from $\mathcal{T}_i$ by Dehn twisting $k$ times about a curve $\alpha$, then $\alpha$ is a normal curve intersecting each edge of $\mathcal{T}_i$ at most three times, and the absolute value of $k$ is bounded above by $i(\mathcal{T}, \mathcal{T}')$.
		
		\item $n = O(|\chi(S)|^3) \log(i(\mathcal{T}, \mathcal{T}'))$. 
		
	\end{enumerate}
	
	Moreover, there is an algorithm that constructs the sequence $\mathcal{T}_1, \cdots, \mathcal{T}_n$ in time that is a polynomial function of $\log(i(\mathcal{T}, \mathcal{T}'))$ and $|\chi(S)|$. Here we assume that $\mathcal{T}$ and $\mathcal{T}'$ have the same vertex (when $S$ is closed), $\mathcal{T}$ is given as a union of (possibly ideal) triangles with gluing instructions, and $\mathcal{T}'$ is given in terms of its normal coordinates with respect to $\mathcal{T}$. For the output, each $\mathcal{T}_i$ is given as a union of (possibly ideal) triangles with gluing instructions together with the flip or twist move from $\mathcal{T}_i$ to $\mathcal{T}_{i+1}$. 
\end{triangulations}

We will prove this by dualising the spines of Theorem \ref{thm: edge swaps on spines}. Recall that the \emph{dual} of a spine $\Gamma$ is a polygonal decomposition that has a complementary disc for each vertex of $\Gamma$ and an edge dual to each edge of $\Gamma$. When the surface is closed, the dual of $\Gamma$ has a single vertex in the disc region $S \setminus \setminus \Gamma$. When the surface has non-empty boundary, the vertices of the dual of $\Gamma$ are all 1-valent and lie on $\partial S$.

	The dual 1-complex to a one-vertex triangulation of a closed surface $S$ or an ideal triangulation of a surface with boundary is a \emph{trivalent spine}; i.e. a spine in which every vertex has degree 3. A flip on a triangulation or ideal triangulation corresponds to a \emph{Whitehead move} on its dual spine. Therefore, Theorem \ref{thm: one-vertex triangulations, fixed vertices} can be rephrased in terms of two given trivalent spines and the existence of a `short' sequence of Whitehead moves and twist maps taking one to the other. 

Each edge swap between trivalent spines can be written as a composition of a controlled number of Whitehead moves. This is the content of the next lemma which is an analogue of Lemma 8.3 of \cite{lackenby2019triangulation} for \emph{trivalent} spines.

\begin{lem}
	Let $\Gamma$ and $\Gamma'$ be trivalent spines for a compact orientable surface $S$ such that $\Gamma'$ is obtained from $\Gamma$ by an edge swap. Then $\Gamma'$ can be obtained from $\Gamma$ by $O(|\chi(S)|)$ Whitehead moves. 
	
	Moreover, there is an algorithm that constructs such a sequence of Whitehead moves in time that is a polynomial function of $|\chi(S)|$. 
	\label{edge swap as whitehead moves}
\end{lem}

\begin{proof}
	We follow the proof of Lemma 8.3 in \cite{lackenby2019triangulation}. Assume that $\Gamma'$ is obtained from $\Gamma$ by deleting the edge $e$ and adding $e' \not= e$. Note that by definition of an edge swap, $e$ and $e'$ can only intersect at endpoints if at all. Let $A$ be the surface obtained by cutting $S$ along $\Gamma \setminus \{e\} $. This is an annulus (when $S$ is closed) or a three-times punctures sphere or once-punctured annulus (when $\partial S \not= \emptyset$). Then $e$ and $e'$ are two essential properly embedded arcs in $A$, and hence they are isotopic in $A$. Indeed, there is a disc component of $A \backslash\backslash (e \cup e')$ with boundary equal to the concatenation of an arc in $\partial A$, the edge $e$, another arc in $\partial A$ and the other edge $e'$. First we move one endpoint of $e$ across this disc to one endpoint of $e'$. Then we move the other endpoint of $e$ across this disc to the other endpoint of $e'$. Finally we isotope $e$ to $e'$ keeping its endpoints fixed; this third step only requires an isotopy and no Whitehead moves. It is enough to show that the first step can be done with $O(|\chi(S)|)$ Whitehead moves, as the second step is similar. There are $O(|\chi(S)|)$ vertices of $\Gamma$ between the first endpoints of $e$ and $e'$ along $\partial A$. Passing the endpoint of $e$ across any one of these vertices can be seen as a Whitehead move. Therefore, the total number of Whitehead moves needed is $O(|\chi(S)|)$.
\end{proof}

\begin{lem}
Let $D$ be a polygon with $n$ sides. Let $\mathcal{T}$ and $\mathcal{T}'$ be triangulations of $D$ where each side of $D$ is an edge and with no vertices in the interior of $D$. Then $\mathcal{T}$ and $\mathcal{T}'$ differ by a sequence of at most $2n-6$ flips. Moreover, there is an algorithm that constructs such a sequence of flips in time that is polynomial in $n$. 
\label{Lem:FlipsInDisc}
\end{lem}

We term triangulations $\mathcal{T}$ and $\mathcal{T}'$ as above \emph{diagonal subdivisions} of $D$.

\begin{proof}
The proof is as in \cite[Lemma 2]{sleator1988rotation} and we repeat it to make the algorithmic part of the statement clear.  Given a diagonal subdivision of $D$ and a vertex $x$, if the degree $\deg(x)$ is not equal to $n-3$, then we can increase $\deg(x)$ by performing a flip. Hence after $n-3 - \deg(x)$ flips, we can covert the subdivision into a new subdivision where all diagonals have one endpoint at $x$. Hence, $\mathcal{T}$ can be converted to $\mathcal{T}'$ using at most $2n-6$ flips.	
\end{proof}

The dual of a triangulation $\mathcal{T}$ as in the above lemma is a tree embedded within the disc $D$ that has 1-valent vertices on $\partial D$ and trivalent vertices in the interior of $D$. We can view the lemma as providing a sequence of Whitehead moves between any two such trees.

\begin{proof}[Proof of Theorem \ref{thm: one-vertex triangulations, fixed vertices}]
We are given 1-vertex or ideal triangulations $\mathcal{T}$ and $\mathcal{T}'$. Let $\Gamma$ and $\Gamma'$ be the spines dual to $\mathcal{T}$. By Theorem \ref{thm: edge swaps on spines},
there is a sequence $\Gamma = \Gamma_0, \Gamma_1, \cdots, \Gamma_n = \Gamma'$ of spines for $S$ such that: 
	
	\begin{enumerate}
		\item Each $\Gamma_{i+1}$ is obtained from $\Gamma_i$ by an edge swap, an expansion or contraction of an embedded edge, or a power of a Dehn twist. When $\Gamma_{i+1}$ is obtained from $\Gamma_i$ by Dehn twisting $k$ times about a curve $\alpha$, then $\alpha \cap \Gamma_i$ is a vertex of $\Gamma_i$, and the absolute value of $k$ is bounded above by $i(\Gamma, \Gamma')$.
				
		\item The number $n$ of steps in this sequence is $O(\chi(S)^2\log(i(\mathcal{T}, \mathcal{T}')) + \chi(S)^2)$. 
	\end{enumerate}

We will remove a small regular neighbourhood of each vertex of $\Gamma_i$ and replace it by a tree. Each vertex of this tree has either degree 1 or 3. The 1-valent vertices lie where the tree meets the remnants of the edges of $ \Gamma_i$, and the remaining vertices of the tree are trivalent, see Figure \ref{fig:inserting trees}. Let $F_i$ be the union of these trees. Set $\mathcal{Q}_i$ to be the resulting trivalent spine.

\begin{figure}
	\includegraphics[width = 3 in]{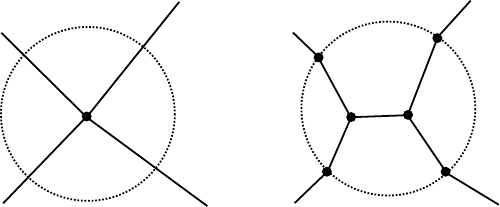}
	\caption{Proof of Theorem \ref{thm: one-vertex triangulations, fixed vertices}: a small regular neighbourhood of each vertex of $\Gamma_i$ is replaced with a tree.}
	\label{fig:inserting trees}
\end{figure}

The spine $\Gamma_0$ is dual to a one-vertex or ideal triangulation, so any small neighbourhood of one of its vertices is of the right form. The procedure can be initiated by taking $F_0$ to be the union of a small neighbourhood of each vertex and then $\mathcal{Q}_0 = \Gamma_0$. Initially, each component of $F_0$ has a single trivalent vertex and three 1-valent vertices. To define $F_{i+1}$, we consider various cases:

\begin{enumerate}
	\item Suppose that $\Gamma_i \rightarrow \Gamma_{i+1}$ is an edge swap, removing an edge $e$ and inserting an edge $e'$. Then $e$ corresponds to an edge (also called $e$) of $\mathcal{Q}_i$. The edge $e'$ may have one or both of its endpoints on a vertex of $\Gamma_i$, in which case when we remove a regular neighbourhood of these vertices, we also remove the end segments of $e'$, but we can then extend the remnant of $e'$ to an edge $e''$ with one or both endpoints on the interior of an edge of $F_i$. See Figure \ref{fig:edge swap on dual spines}. If we remove $e$ from $\mathcal{Q}_i$ and attach on $e''$, the result is a trivalent spine $\mathcal{Q}_{i+1}$. The forest $F_{i+1}$ is defined to be the intersection between $\mathcal{Q}_{i+1}$ and the regular neighbourhood of the vertices of $\Gamma_{i+1}$. By construction, $\mathcal{Q}_{i+1}$ is obtained from $\mathcal{Q}_i$ by an edge swap. Hence by Lemma \ref{edge swap as whitehead moves}, $\mathcal{Q}_i$ and $\mathcal{Q}_{i+1}$ differ by a sequence of $O(|\chi(S)|)$ Whitehead moves.

	\begin{figure}
		\labellist
		\pinlabel $e'$ at 120 55
		\pinlabel $e''$ at 320 57
		\endlabellist
		\includegraphics[width = 4 in]{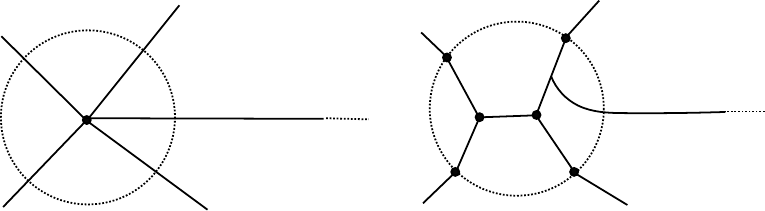}
		\caption{Proof of Theorem \ref{thm: one-vertex triangulations, fixed vertices}, case (1): the edge $e'$ has an endpoint on a vertex of $\Gamma_i$, see the left side of the figure. We then construct $e''$ as in right such that its corresponding endpoint lies in the interior of an edge of $\mathcal{Q}_i$. }
		\label{fig:edge swap on dual spines}
	\end{figure}

	\item Suppose $\Gamma_i \rightarrow \Gamma_{i+1}$ is the contraction of an edge $e$. In $\mathcal{Q}_i$, there is a copy of $e$ and at its endpoints there are two components of $F_i$. We amalgamate them into a single tree by attaching the edge $e$, and we declare that this is a component of $F_{i+1}$. The remaining components of $F_i$ become components of $F_{i+1}$. In this way, $\mathcal{Q}_{i+1}$ is isotopic to $\mathcal{Q}_i$. See Figire \ref{fig: joining trees}.
	
	\begin{figure}
		\labellist
		\pinlabel $e$ at 125 170
		\pinlabel $\tiny{contract}$ at 310 160
		\endlabellist
		\includegraphics[width = 4 in]{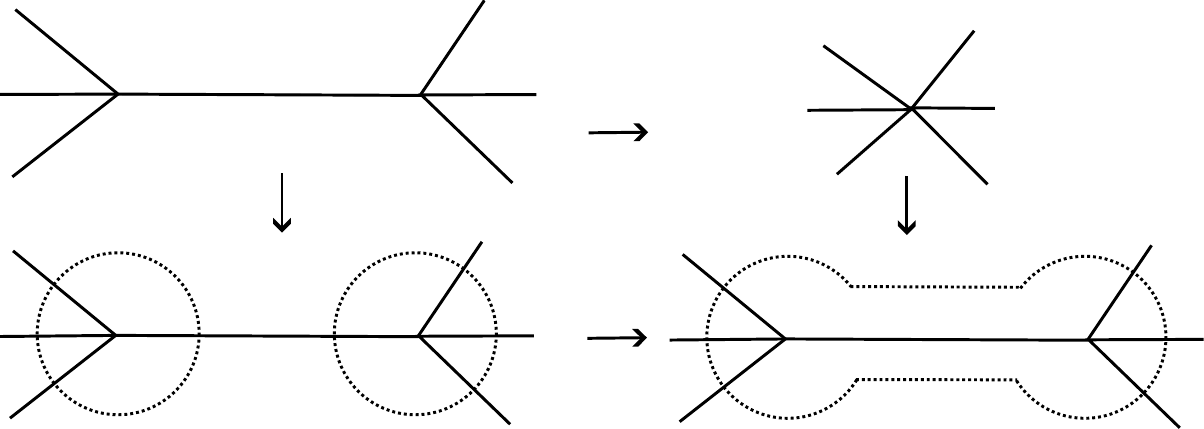}
		\caption{Proof of Theorem \ref{thm: one-vertex triangulations, fixed vertices}, case (2): The new spine is obtained by contracting an edge $e$, see the top. The corresponding trivalent spine is isotopic to the previous one, see the bottom. }
		\label{fig: joining trees}
	\end{figure}
	
	\item Now consider the case where $\Gamma_i \rightarrow \Gamma_{i+1}$ is the expansion of an edge $e$ from a vertex $v$. Let $T$ be the component of $F_i$ in a regular neighbourhood of $v$. Let $v_1$ and $v_2$ be the vertices at the endpoint of $e$. Pick trees of the required form $T_1$ and $T_2$ for these regular neighbourhoods to be components of $F_{i+1}$. We can view $T_1 \cup e \cup T_2$ to be a tree lying in a regular neighbourhood of $v$. Using Lemma \ref{Lem:FlipsInDisc}, $T$ can be transformed into $T_1 \cup e \cup T_2$ using $O(|\chi(S)|)$ Whitehead moves. Hence, $\mathcal{Q}_i$ and $\mathcal{Q}_{i+1}$ differ by a sequence of $O(|\chi(S)|)$ Whitehead moves.
	
	\item Finally suppose that $\Gamma_i \rightarrow \Gamma_{i+1}$ is a power of a Dehn twist along a curve $\alpha$ that intersects $\Gamma_i$ in a vertex. In $\mathcal{Q}_i$, this vertex is replaced by a tree, and $\alpha$ can be arranged to intersect this tree in a connected union of edges or a single vertex. We set $F_{i+1}$ to be the image of $F_i$ under this power of a Dehn twist.
	
\end{enumerate}

We now dualise this sequence of trivalent spines to form a sequence of 1-vertex or ideal triangulations of $S$. The dual of each Whitehead move is a flip. Hence, we need at most $O(|\chi(S)|^3\log(i(\mathcal{T}, \mathcal{T}')) + |\chi(S)|^3)$ flips and powers of Dehn twists. Since we assumed that $i(\mathcal{T}, \mathcal{T'}) \geq 2$, this can be equivalently written as $O(|\chi(S)|^3)\log(i(\mathcal{T}, \mathcal{T}'))$. When we Dehn twist, the curves that we twist along intersects the spine in a connected union of edges or a single vertex. When we dualise, this curve is a concatenation of normal arcs, together with a part that runs into a (possibly ideal) vertex of the triangulation. Push the curve off the (possibly ideal) vertex, and we obtain a normal curve that intersects each edge of the triangulation at most three times. To see this note that when we push the curve off the vertex, it skirts around the vertex, and in doing so it picks up at most one new normal arc of each type in each triangle. Since such a triangle may already have had a normal arc in it, we get at most four normal arcs, and these may intersect each edge at most three times. See Figure \ref{fig:pushing the curve off the vertex} for an example. 

\begin{figure}
	\includegraphics[width= 3.5 in]{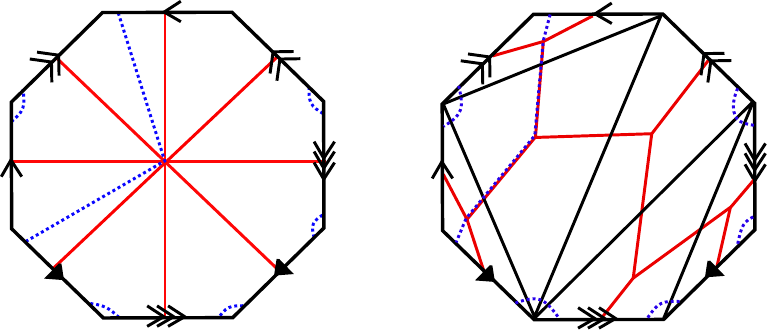}
	\caption{Left: A surface of genus two is shown as a polygon with side identifications. The spine $\Gamma$ (in red solid lines) is the union of solid curves passing through the center of the polygon. The normal curve $\alpha$ (in dashed blue lines) is the union of the two dashed lines passing through the center of the polygon and five other normal arcs going close to the vertex of the polygon. Right: The spine $\Gamma$ is perturbed to a trivalent spine $\mathcal{Q}$ (in red solid lines), which is dual to the shown triangulation $\mathcal{T}$ of the surface. The curve $\alpha$ is isotoped to be a concatenation of normal arcs (with respect to $\mathcal{T}$) together with a connected union of two edges of  $\mathcal{Q}$. It can be seen in this example that the isotoped $\alpha$ could have 3 normal arcs in some triangles, including 2 normal arcs of the same type.}
	\label{fig:pushing the curve off the vertex}
\end{figure}

It is clear that this sequence of 1-vertex and ideal triangulations is constructible in polynomial time as a function of $|\chi(S)|$ and $\log(i(\mathcal{T}, \mathcal{T}'))$.

\end{proof}

\begin{lem}
\label{Lem:TriangulationWithFewIntersections}
Let $S$ be a compact connected orientable surface with non-empty boundary, and let $A$ be a collection of disjoint arcs properly embedded in $S$. Let $V$ be a finite collection of points on $\partial S$ disjoint from $A$, with at least one point of $V$ on each component of $\partial S$. In the case where $S$ is a disc, suppose also $|V| \geq 3$. Then $V$ is the vertex set of a triangulation $\mathcal{T}$ of $S$ with the property that each edge of $\mathcal{T}$ intersects each component of $A$ at most twice. Moreover, given a triangulated surface $S$, a normal multi-arc $A$, and a set of points $V$ as above, there is an algorithm that constructs the triangulation $\mathcal{T}$. The algorithm runs in time that is a polynomial function of the number of triangles in some input triangulation of $S$, the cardinality of $V$, and the $\ell^1$-norm of the normal coordinates of $A$. 
\end{lem}

\begin{proof}
Consider first the case where $S$ is a disc. Note that each component of $\partial S \setminus V$ intersects each component of $A$ at most twice, in a subset of the endpoints of that component. These components of $\partial S \setminus V$ will be edges of $\mathcal{T}$, and so we have verified the required condition for these edges. If $|V| = 3$, then we set $\mathcal{T}$ to be a single triangle. So suppose $|V| > 3$. Pick three vertices in $V$ that are consecutive around $\partial S$. Join the outermost two by an edge that runs parallel to $\partial S$. This will be an edge of $\mathcal{T}$. It intersects each component of $A$ at most twice. These three vertices now span a triangle. Removing this triangle from $S$ gives a disc with one fewer vertices in its boundary. Hence, by induction, $S$ has the required triangulation.

Now suppose that $S$ is not a disc. Suppose also that some component of $\partial S$ contains more than one vertex in $V$. Pick three consecutive vertices on this component of $\partial S$ (where the outermost two may be equal) and join the outermost two by an edge and then remove a triangle, as above. In this way, we may suppose that each component of $\partial S$ contains a single vertex. We now modify $A$ to a new set of arcs $A'$ as follows:
\begin{enumerate}
\item remove any inessential arcs;
\item replace parallel essential arcs by a single arc;
\item if any complementary region is not a disc or is a disc with more than three arcs in its boundary, then add in an essential arc not parallel to a previous one, and avoiding $A$. Repeat this as much as possible. This step turns the arc system into a hexagonal decomposition of the surface where for each hexagon the edges alternately lie in $\partial S$ and in $A'$.
\end{enumerate}

The resulting arcs form the 1-skeleton of an ideal triangulation of $S$. By construction, each either is equal to a component of $A$ or is disjoint from $A$. We now add further arcs to $A'$, one for each component of $\partial S$. Consider any component $C$ of $\partial S$ and the vertex $v$ in $V$ that it contains. Pick some orientation on $C$. Let $p_1$ and $p_2$ be the endpoints of $A'$ that are adjacent to $v$, where the orientation on $C$ runs from $p_1$ to $v$. Say that $p_i$ lies in the arc $a_i$ in $A'$. We add the following arc to $A'$: it starts at the end of $a_1$ that is not $p_1$, it runs along $a_1$ and then along the sub-arc of $\partial S$ containing $v$ up to $p_2$. This sub-arc of $\partial S$ may contain endpoints of inessential arcs of $A$, in which case we modify this new arc so that it avoids these inessential arcs of $A$. We repeat this for each component of $\partial S$. Let $A''$ be the union of $A'$ and these new arcs, perturbed a little so that they are disjoint from each other. By construction, they are disjoint from $A$. 

We now slide the endpoints of $A''$ along $\partial S$, using the chosen orientations on the components of $\partial S$. We stop when all the endpoints of $A''$ lie in $V$. The result is the 1-skeleton of the required triangulation $\mathcal{T}$. Note that this sliding operation may introduce points of intersection between the edges of $\mathcal{T}$ and $A$, but each edge of $\mathcal{T}$ intersects each component of $A$ at most twice, near the endpoints of that component of $A$.
\end{proof}

\begin{thm}
	Let $S$ be a closed orientable surface, and $\mathcal{T}$ be a one-vertex triangulation of $S$. Let $\gamma$ be an essential simple closed normal curve given by its normal vector $(\gamma)$ with respect to $\mathcal{T}$, and denote the bit-sized complexity of $(\gamma)$ by $|(\gamma)|_\mathrm{bit}$ and its $\ell^1$-norm by $|(\gamma)|_1$. There is an algorithm that constructs a sequence of one-vertex triangulations $\mathcal{T} = \mathcal{T}_0, \mathcal{T}_1, \cdots, \mathcal{T}_n$ of $S$ and a sequence of curves $\gamma = \gamma_0, \gamma_1, \cdots, \gamma_n$ such that 
	
	\begin{enumerate}
		\item $\gamma_i$ is isotopic to $\gamma$ for every $i$.
		\item $\gamma_i$ is in normal form with respect to $\mathcal{T}_i$ for every $i<n$.
		\item $\gamma_n$ lies in the 1-skeleton of $\mathcal{T}_n$. 
		\item Each $\mathcal{T}_{i+1}$ is obtained from $\mathcal{T}_i$ by either a flip or a power of a Dehn twist. When $\mathcal{T}_{i+1}$ is obtained from $\mathcal{T}_i$ by Dehn twisting $k$ times about a curve $\alpha$, then $\alpha$ intersects each edge of $\mathcal{T}_i$ at most three times, and the absolute value of $k$ is bounded above by a polynomial function of $|(\gamma)|_1$ and $|\chi(S)|$.
		\item The algorithm runs in time that is a polynomial function of $|(\gamma)|_\mathrm{bit}$ and $|\chi(S)|$.
	\end{enumerate}

For the output, each $\mathcal{T}_i$ is given as a union of triangles with gluing instructions together with the flip or twist move from $\mathcal{T}_i$ to $\mathcal{T}_{i+1}$, and $\gamma_i$ is given by its normal coordinates with respect to $\mathcal{T}_i$.
\label{thm: normal curve}
\end{thm}

\begin{proof}
	The idea is to extend $\gamma$ to a one-vertex triangulation $\mathcal{T}'$ of $S$, and then repeat the proof of Theorem \ref{thm: one-vertex triangulations, fixed vertices}. More precisely, we show that there is an algorithm that extends $\gamma$ to a one-vertex triangulation $\mathcal{T}'$ of $S$ such that:
	
	\begin{enumerate}
		\item $\mathcal{T}'$ is in normal form with respect to $\mathcal{T}$.
		\item The algorithm runs in time that is a polynomial function of $|(\gamma)|_\mathrm{bit}$ and $\chi(S)$. In particular the bit-sized complexity of the normal coordinates of $\mathcal{T}'$ with respect to $\mathcal{T}$ are bounded from above by such a polynomial function. 
	\end{enumerate}

\textit{Step 1:} Construction of $\mathcal{T}'$. 

The normal arcs of $\gamma$ decompose triangles of $\mathcal{T}$ into several $0$-handles (or $2$-cells). A $0$-handle of $S \setminus \setminus (\mathcal{T} \cup \gamma)$ is called a \emph{parallelity handle} if it is a 4-gon with two of its opposite sides being parallel normal arcs of $\gamma$ and the other two sides lying in the edges of $\mathcal{T}$. Each parallelity 0-handle comes with the structure of an $I$-bundle over an interval, where the interval base is parallel to a normal arc of $\gamma$. See Figure \ref{fig: parallelity handle}. The $I$-bundle structures on parallelity $0$-handles glue together to form an $I$-bundle $\mathcal{B}$, called the \emph{parallelity bundle}, whose base $B$ is a possibly disconnected compact 1-manifold. Moreover, the base $B$ can be considered as a normal (possibly not closed) multicurve. Since we assumed that $\gamma$ is connected, the base $B$ is a finite union of closed intervals. The \emph{vertical boundary $\partial_v \mathcal{B} $ of $\mathcal{B}$} is defined as the restriction of the $I$-bundle $\mathcal{B}$ to $\partial B$. The \emph{horizontal boundary $\partial_h \mathcal{B}$ of $\mathcal{B}$} is the $(\partial I)$-bundle over $B$, obtained by the restriction of the $I$-bundle. Therefore, the boundary of $\mathcal{B}$ is the union of its vertical boundary and horizontal boundary. Similarly, for each component of $\mathcal{B}$, we can speak of its vertical and horizontal boundary.

\begin{figure}
	\includegraphics[width = 3 in]{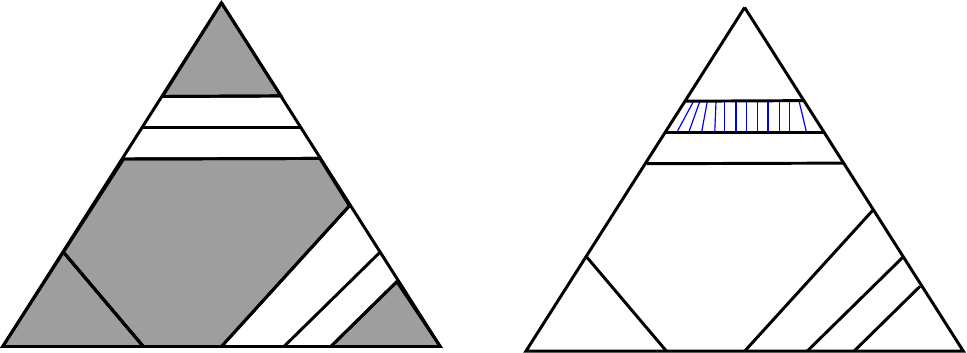}
	\caption{Normal arcs decompose a triangle into a union of $0$-handles (or $2$-cells). On the left the $0$-handles that are not parallelity handles are shaded. Note that in general there are at most four $0$-handles that are not parellelity handles. The fibers of the $I$-bundle structure of a parallelity handle (well-defined up to isotopy) are shown on the right hand side.}
	\label{fig: parallelity handle}
\end{figure}

In each triangle of $\mathcal{T}$, there are at most four 0-handles that are not parallelity handles, see Figure \ref{fig: parallelity handle}. The union of the $0$-handles that are not parallelity handles forms a 2-complex called the \emph{gut region}. Therefore, the number of 0-handles of the gut region is at most $4t$, where $t$ is the number of triangles in $\mathcal{T}$. 

Denote the vertex of $\mathcal{T}$ by $v$. Let $\Delta$ be a triangle $0$-handle of the gut region, and $x$ be the side of $\Delta$ opposite the vertex $v$. Place a vertex $w$ on $x$, and isotope $\gamma$ by dragging the vertex $w$ to $v$ along a straight line in $\Delta$. After this isotopy and normalisation, $\gamma$ is a simple closed normal curve with one vertex on it that coincides with the vertex of $\mathcal{T}$. See Figure \ref{fig: isotoping vertex}. We will apply the Agol--Hass--Thurston algorithm to compute the following data about the parallelity bundle: Denote the components of $\mathcal{B}$ by $\mathcal{B}_1, \cdots, \mathcal{B}_k$. For each $\mathcal{B}_i$, we compute the normal coordinates for its base, together with the attachment of the vertical boundary of $\mathcal{B}_i$ to the gut region, and the relative $I$-direction for the two components of the vertical boundary of $\mathcal{B}$.

\begin{figure}
	\labellist
	\pinlabel $v$ at 87 94
	\pinlabel $w$ at 143 84
	\pinlabel $\Delta$ at 175 80
	\endlabellist
	\includegraphics[width = 2 in]{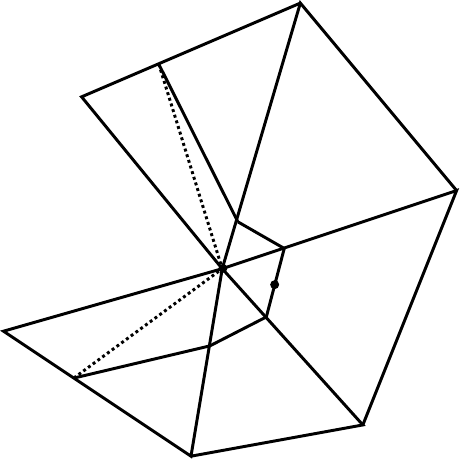}
	\caption{Proof of Theorem \ref{thm: normal curve}, isotoping $\gamma$ such that it becomes a normal curve passing through the vertex $v$ of the triangulation: Some of the triangles adjacent to $v$ are shown. We first drag the point $w$ of $\gamma$ to the vertex $v$ inside $\Delta$, and then normalise the resulting curve. The normalised curve passing through the vertex $v$ is shown in dashed lines. }
	\label{fig: isotoping vertex}
\end{figure}

Therefore, we have a handle decomposition $\mathcal{H}$ of $X = S \setminus \setminus \gamma$ into 0-handles, where each 0-handle is either a 0-handle of the gut region, or it is a 4-gon that is equal to a component of $\mathcal{B}$. This handle decomposition $\mathcal{H}$ of $X$ has $O(t)$ 0-handles. There is a natural immersion $i \colon X \rightarrow S$ whose restriction to the interior of $X$ is an embedding, and such that $\gamma$ lies in the image of the boundary of $X$  under the map $i$. Note that $X$ is a compact orientable surface with two or one connected components according to whether $\gamma$ is separating in $S$ or not. Let $V$ be the copies of the vertex $v$ in $X$. By Lemma \ref{Lem:TriangulationWithFewIntersections}, $X$ admits a triangulation $\mathcal{T}'_X$ with vertex set $V$ and where each edge intersects each component of $\partial_v \mathcal{B}$ at most twice.


Define the one-vertex triangulation $\mathcal{T}'$ of $S$ as the image of the triangulation $\mathcal{T}'_X$ of $S  \setminus \setminus \gamma$ under the map $i$. We can now read off the normal coordinate with respect to $\mathcal{T}$ of each edge $e$ of $\mathcal{T'}$. To see this, consider two cases:

\begin{enumerate}
	\item[i)] For each part of $e$ in the gut region, we can read its normal coordinate with respect to $\mathcal{T}$.
	\item[ii)] Let $H$ be a 0-handle of  $\mathcal{H}$ that forms a component  of the parallelity bundle $\mathcal{B}$. We previously computed the base of $H$ as a normal arc with respect to $\mathcal{T}$, using the Agol--Hass--Thurston algorithm. Each time $e$ runs through $H$, it enters and exists $H$ via $\partial_vH$ and so each component of $e \cap H$ is normally parallel to the base of $H$. So, we can read off the normal coordinate of $e \cap H$ with respect to $\mathcal{T}$.
	\end{enumerate}

Summing these coordinates over each 0-handle of $\mathcal{H}$ gives the normal vector of $e$ with respect to $\mathcal{T}$.
Finally, the normal coordinate of $\mathcal{T}'$ with respect to $\mathcal{T}$  can be obtained by summing up the normal coordinates of its edges. Note that by construction we have 
\begin{eqnarray}
	i(\mathcal{T}, \mathcal{T}') \leq O(|\chi(S)|) \cdot |(\gamma)|_1 + O(\chi(S)^2). 
	\label{eqn: size of triangulation}
\end{eqnarray}  
To see this note that the base of the parallelity bundle is a normal multi-arc of $\ell^1$-norm at most $|(\gamma)|_1$. Moreover, $\mathcal{T}' $ has $O(|\chi(S)|)$ edges and  each edge of $\mathcal{T}'$ passes through each component of $\mathcal{B}$ at most twice. Therefore the intersection of $\mathcal{T}'$ with $\mathcal{T} \cap \mathcal{B}$ contributes at most $O(|\chi(S)| ) \cdot |(\gamma)|_1$ intersection points. Additionally, each edge of $\mathcal{T}'$ intersects the gut region of $\mathcal{H}$ at most $O(|\chi(S)|)$ times, and so $\mathcal{T}'$ intersects the restriction of $\mathcal{T}$ to the gut region at most $O(\chi(S)^2)$ times.  \\

\textit{Step 2:} Construction of $\mathcal{T}_i$ and $\gamma_i$. \\

Let
\[ \mathcal{T} = \mathcal{T}_0, \mathcal{T}_1, \cdots, \mathcal{T}_n = \mathcal{T}' \]
be the sequence of one-vertex triangulations given by Theorem \ref{thm: one-vertex triangulations, fixed vertices}. Set $\gamma_0 := \gamma$. Each $\mathcal{T}_{i+1}$ is obtained from $\mathcal{T}_i$ by either a flip or a twist map. Given the normal curve $\gamma_i$ with respect to $\mathcal{T}_i$, we put $\gamma_{i}$ in normal form with respect to $\mathcal{T}_{i+1}$ and define $\gamma_{i+1}$ as this normal representative. Consider two cases: 

\begin{enumerate}
	\item[a)] If $\mathcal{T}_{i+1}$ is obtained from $\mathcal{T}_i$ by a flip, then it is easy to put $\gamma_i$ in normal form with respect to $\mathcal{T}_{i+1}$ and find its normal coordinates. 
	\item[b)] If $\mathcal{T}_{i+1}$ is obtained from $\mathcal{T}_i$ by a twist map $(T_\alpha)^k$, then the normal coordinates of $\gamma_i$ with respect to $\mathcal{T}_{i+1}$ are equal to the normal coordinates of $(T_\alpha)^{-k}(\gamma_i)$ with respect to $\mathcal{T}_i$. The normal coordinates of $(T_\alpha)^{-k}(\gamma_i)$ with respect to $\mathcal{T}_i$ can be read off from the normal coordinates of $\gamma_i$ with respect to $\mathcal{T}_i$, the normal coordinates of $\alpha$ with respect to $\mathcal{T}_i$, and the value of $k$, and all this information is given to us by Theorem \ref{thm: one-vertex triangulations, fixed vertices}.
\end{enumerate}

By construction, $\gamma_i$ satisfy conditions (1)--(3) of the statement of the theorem.  By Theorem \ref{thm: one-vertex triangulations, fixed vertices} and equation (\ref{eqn: size of triangulation}), conditions (4)--(5) of the statement are satisfied as well. 
	
\end{proof}

We can similarly prove a version of Theorem \ref{thm: normal curve} for ideal triangulations of surfaces with boundary. 

\begin{thm}
	Let $S$ be a compact orientable surface with non-empty boundary, and $\mathcal{T}$ be an ideal triangulation of $S$. Let $\gamma$ be an essential simple closed normal curve or an essential simple normal arc given by its normal vector $(\gamma)$ with respect to $\mathcal{T}$, and denote the bit-sized complexity of $(\gamma)$ by $|(\gamma)|_\mathrm{bit}$ and its $\ell^1$-norm by $|(\gamma)|_1$. There is an algorithm that constructs a sequence of ideal triangulations $\mathcal{T} = \mathcal{T}_0, \mathcal{T}_1, \cdots, \mathcal{T}_n$ of $S$ and a sequence of curves $\gamma = \gamma_0, \gamma_1, \cdots, \gamma_n$ such that 
	
	\begin{enumerate}
		\item $\gamma_i$ is isotopic to $\gamma$ for every $i$.
		\item $\gamma_i$ is in normal form with respect to $\mathcal{T}_i$ for every $i<n$.
		\item $\gamma_n$ intersects each edge of $\mathcal{T}_n$ at most twice. 
		\item Each $\mathcal{T}_{i+1}$ is obtained from $\mathcal{T}_i$ by either a flip or a power of a Dehn twist. When $\mathcal{T}_{i+1}$ is obtained from $\mathcal{T}_i$ by Dehn twisting $k$ times about a curve $\alpha$, then $\alpha$ intersects each edge of $\mathcal{T}_i$ at most three times, and the absolute value of $k$ is bounded above by a polynomial function of $|(\gamma)|_1$ and $|\chi(S)|$.
		\item The algorithm runs in time that is a polynomial function of $|(\gamma)|_\mathrm{bit}$ and $|\chi(S)|$.
	\end{enumerate}
	
	For the output, each $\mathcal{T}_i$ is given as a union of ideal triangles with gluing instructions together with the flip or twist move from $\mathcal{T}_i$ to $\mathcal{T}_{i+1}$, and $\gamma_i$ is given by its normal coordinates with respect to $\mathcal{T}_i$.
	
\end{thm}

\begin{proof}
	The proof is similar to the proof of Theorem \ref{thm: normal curve}. 
	When $\gamma$ is a closed curve, we first isotope it to create an arc that 
	passes through an ideal vertex $v$ of $\mathcal{T}$. We then repeat 
	the argument as in the proof of Theorem \ref{thm: normal curve} 
	to construct a sequence of ideal triangulations 
	$\mathcal{T}= \mathcal{T}_0, \cdots, \mathcal{T}_n$ and 
	arcs $\gamma_0, \cdots, \gamma_n$ such that 
	$\gamma_n$ is an edge of $\mathcal{T}_n$. 
	Finally when $\gamma$ is a closed curve, we perturb $\gamma_n$ off the ideal vertex 
	$v$ to create a curve that intersects each edge of $\mathcal{T}'$ at most twice. 
	When $\gamma$ is an arc, we perturb $\gamma_n$ so that it is normal
	and disjoint from the edges of $\mathcal{T}'$.
	
	Note that if $\gamma$ is a separating curve and $\mathcal{T}$ has exactly one ideal vertex, the geometric intersection number between $\gamma_n$ and an edge of $\mathcal{T}_n$ is even. Therefore, assuming further that $\gamma$ is essential, there is an edge $e$ of $\mathcal{T}_n$ such that $\gamma$ intersects $e$ at least twice. 
\end{proof}

\section{Application to volumes of hyperbolic 3-manifolds}	

Theorem \ref{thm:upper bound for distance in pants graph} together with Agol's explicit construction of hyperbolic structures in \cite{agol2003small} has the following corollary. 

\theoremstyle{theorem}
\newtheorem*{volume}{Corollary \ref{upper bound for volume}}
\begin{volume}
	Let $\Sigma$ a closed orientable surface of genus $g \geq 2$, and $P$ and $P'$ be pants decompositions of $\Sigma$ with no curve in common. Assume that $M$ is a maximal cusp obtained from a quasi-Fuchsian 3-manifold homeomorphic to $\Sigma \times \mathbb{R}$ by pinching the multicurves $P$ and $P'$ to annular cusps on the two conformal boundary components of $M$. The volume of the convex core of $M$ is  
	\[ O(g^2) \hspace{1mm} \log (i(P, P')) . \]
\end{volume}

\begin{proof}
Since $P$ and $P'$ have no curve in common, each curve in $P$ intersects $P'$ at least once. Therefore $i(P, P') \geq |P| \geq 2$, where $|P| = 3g -3$ is the number of curves in $P$. By Theorem \ref{thm:upper bound for distance in pants graph}, there is a path $C$ consisting of $P_0=P, P_1, \cdots, P_m=P'$ in the pants graph of $\Sigma$ with length $m = O(g^2) \hspace{1mm} \log(i(P, P'))$. 

We recall Agol's work from \cite{agol2003small}; Agol's proof is written for mapping tori but works with minor changes for product manifolds too as we will see below. Define the sequence of circles $\beta_1, \cdots, \beta_m$ in $\Sigma$ where $\beta_{i+1}$ is the circle in $P_{i+1}$ replacing a circle in $P_i$. Also set $\beta_0 = P_0$. Define the subset $\mathcal{A}$ of $\{ \beta_1 , \cdots, \beta_m \}$ as follows: $\beta_i$ belongs to $\mathcal{A}$ whenever $\beta_i$ is isotopic to a curve in $P'$ and for no $j>i$, $\beta_j$ is isotopic to $\beta_i$. For $0 \leq i \leq m$, define the multicurves $B_i$ in $\Sigma \times [ 0 , 1 ]$ as follows: $B_i = \beta_i \times \{ \frac{i}{m} \}$ if $\beta_i \notin \mathcal{A}$, and $B_i = \beta_i \times \{ 1\}$ if $\beta_i \in \mathcal{A}$.  Consider the complement $M_C := \Sigma \times [0,1] \setminus N(\cup B_i)$ where $N(\cup B_i)$ is a regular neighbourhood of $\cup B_i$. In other words, we drill out $\beta_i$ at successive heights from $\Sigma \times [0, 1]$, except that the curves in $P'$ are all drilled out from the level $\Sigma \times \{ 1\}$. Likewise, all the curves in $P_0$ are drilled from the level $\Sigma \times \{ 0\}$, since $\beta_0 \notin \mathcal{A}$. Let $M$ be the maximal cusp obtained from $\Sigma \times \mathbb{R}$ by pinching the multicurves $P$ and $P'$, and $N$ be the convex core of $M$. Then $N$ is obtained from $M_C$ by Dehn filling boundary components corresponding to those curves $B_i$ that lie in the interior of $\Sigma \times [0,1]$. Agol \cite[Lemma 2.3]{agol2003small} constructed a complete hyperbolic structure with totally geodesic boundary and with rank-1 and rank-2 cusps on $M_C$ by gluing together `model pieces'. The rank-1 cusps correspond to the curves drilled out from $\Sigma \times \{ 0 ,1\}$, which are $P \times \{ 0\}$ and $P' \times \{ 1\}$. The rank-2 cusps correspond to the curves drilled out from the interior of $\Sigma \times [0,1]$. Finally the totally geodesic boundary corresponds to the union of pair of pants in $\Sigma \times \{ 1\} \setminus P' \times \{ 1\}$ and $\Sigma \times \{ 0\} \setminus P \times \{ 0\}$.  The explicit construction of the hyperbolic structure shows that
\[ \textrm{Vol}(M_C) = (2 A +S) \mathrm{V}_\mathrm{oct} \leq 2 \mathrm{V}_\mathrm{oct} m, \]
where $\textrm{Vol}$ denotes the volume of the hyperbolic structure, $A$ and $S$ indicate the number of \emph{associativity/simple} moves in the path $C$, and $\mathrm{V}_\mathrm{oct}$ is the volume of the regular ideal octahedron; see the proof of \cite[Corollary 2.4]{agol2003small}. Let $N$ be the manifold obtained by filling in the boundary components of $M_C$ corresponding to those $B_i$ that lie in the interior of $\Sigma \times [0,1]$. Then $N$ is a hyperbolic 3-manifold with totally geodesic boundary and rank-1 cusps and 
\[ \textrm{Vol}(N) < \textrm{Vol}(M_C).  \]

To see this note that by Thurston's hyperbolisation theorem for Haken manifolds, $N$ is hyperbolic since it is Haken, atoroidal, and anannular. By Thurston \cite{thurston1979geometry}, the volume decreases under Dehn filling and so
\[ \textrm{Vol}(N) < \textrm{Vol}(M_C).  \]

Alternatively we can first double each of $M_C$ and $N$ along the subset $(\Sigma \times \{ 0\} \setminus P \times \{ 0\}) \cup (\Sigma \times \{ 1\} \setminus P' \times \{ 1\})$ of boundary, and work with the doubled manifold to avoid rank-1 cusps. Hence we have 

\[ \textrm{Vol}(N)<\textrm{Vol}(M_C) \leq 2 \mathrm{V}_\mathrm{oct} m = O(g^2) \hspace{1mm}  \log(i(P, P')). \]
	
\end{proof}

The next result shows that our bound in Corollary \ref{upper bound for volume} is sharp up to a multiplicative factor of $g \log(g)$. 

\begin{prop}
	There is a universal positive constant $C$ such that for any $g \geq 3$, there are maximal cusps $M$ obtained from $\Sigma_g \times \mathbb{R} $ by pinching the multicurves $P$ and $P'$ to annular cusps such that the volume of the convex core of $M$ is greater than
	\[ C \medspace \frac{g}{\log(g)} \log(i(P,P')). \] 
	
	\label{sharpness of volume bound}
\end{prop}

\begin{proof}
	Let $P_1$ be a pants decomposition of a surface $\Sigma_2$ of genus $2$, and denote the simple closed curves in $P_1$ by $\{ \alpha_1, \alpha_2 , \alpha_3 \}$. Pick a base point $b$ on $\Sigma_2$, and let $\pi_1(\Sigma_2, b)$ be the fundamental group. Let $f \colon \Sigma_2 \rightarrow \Sigma_2$ be a pseudo-Anosov mapping class that acts trivially on the first homology group $H_1(\Sigma_2; \mathbb{Z})$. After an isotopy, we can assume that $f$ fixes the base point $b$, and so it acts on $\pi_1(\Sigma_2, b)$. Define the pants decomposition $P'_1 := f(P_1)$. Since $f$ is pseudo-Anosov, it does not fix the isotopy class of any essential simple closed curve on $\Sigma_2$. So, after possibly replacing $f$ by a power of itself, we can assume that $f(\alpha_i)$ is not isotopic to $\alpha_j$ for any $1 \leq i , j \leq 3$.
		
	Picking a tree connecting $\alpha_i$ to the base point $b$ in $\Sigma_2$, we can identify $\alpha_i$ with elements of $\pi_1(\Sigma_2, b)$. Let $\phi \colon \pi_1(\Sigma_2,b) \rightarrow \mathbb{Z}$ be a surjective homomorphism such that $\phi(\alpha_i) = 0$ for $1 \leq i \leq 3$, for example $\phi$ can be taken to be the algebraic intersection pairing with some $\alpha_j$ that is non-separating. Denote by $\phi_k$ the composition of $\phi$ with the reduction map $ \mathbb{Z} \rightarrow \mathbb{Z}/k\mathbb{Z}$ modulo $k$ for $k \geq 2$, and define $G$ as the kernel of $\phi_k$. Therefore, $G$ is an index $k$ subgroup of $\pi_1(\Sigma,b)$ that contains all $\alpha_i$ for $1 \leq i \leq 3$. The image of $\phi_k$ is abelian and so it factors through the abelianisation map $\pi_1(\Sigma_2,b) \rightarrow H_1(\Sigma_2 ; \mathbb{Z})$. Since $f$ acts trivially on homology, we conclude that the curves in $P'_1 = f(P_1)$ also lie in $G = \ker (\phi_k)$. 
	
	Let $M_1$ be the maximal cusp obtained from $\Sigma_2 \times \mathbb{R}$ by pinching the pants decompositions $P_1$ and $P'_1$ to annular cusps, here we use the fact that $P_1$ and $P'_1$ have no curve in common.  Let $M$ be the $k$-sheeted cover of $M_1$ corresponding to the subgroup $G < \pi_1(\Sigma_2 , b) \cong \pi_1(M_1)$. The hyperbolic structure on $M_1$ lifts to a hyperbolic structure on $M$, and so there exists $g$ such that $M$ is obtained from $\Sigma_g \times \mathbb{R}$ by pinching the lifts $P := \tilde{P_1}$ and $P': = \tilde{P'_1}$ of respectively $P_1$ and $P'_1$ to annular cusps. Since the curves in $P_1$ and $P'_1$ lie in $G = \ker (\phi_k)$ by construction, we deduce that $P := \tilde{P_1}$ and $P': = \tilde{P'_1}$ are pants decompositions (that is, they cut $\Sigma$ into pairs of pants). By comparing Euler characteristics we have
	\[ \chi(\Sigma_g) = k \cdot \chi(\Sigma_2) \implies g-1 = k. \]
	Since $k \geq 2$, we have that $g \geq 3$. Denote the convex core of a hyperbolic manifold $N$ by $\mathrm{CC}(N)$, and its hyperbolic volume by $\mathrm{Vol}(N)$. 
	Since $M$ is a $k$-sheeted cover of $M_1$ 
	\[ \mathrm{Vol}(\mathrm{CC}(M)) = k \cdot \mathrm{Vol}(\mathrm{CC}(M_1)) = (g-1) \mathrm{Vol}(\mathrm{CC}(M_1)). \]
	To see this, let $\Gamma_1$ be a discrete group of isometries of hyperbolic three-space $\mathbb{H}^3$ such that $M_1 = \mathbb{H}^3/\Gamma_1$, and $\Gamma < \Gamma_1$ be a subgroup of index $k$ with $M = \mathbb{H}^3/\Gamma$. Denote the limit sets of $\Gamma_1$ and $\Gamma$ by respectively $\Lambda(\Gamma_1)$ and $\Lambda(\Gamma)$. It is easy to see that $\Lambda(\Gamma_1) = \Lambda(\Gamma)$, since $\Gamma < \Gamma_1$ is of finite index. Denote the convex hull of $\Lambda(\Gamma_1)$ by $\mathrm{CH}(\Lambda(\Gamma_1))$, and define $\mathrm{CH}(\Lambda(\Gamma))$ similarly. Then the convex core of the hyperbolic manifold $\mathbb{H}^3/\Gamma_1$ is the image of $\mathrm{CH}(\Lambda(\Gamma_1))$ under the projection $p_1 \colon \mathbb{H}^3 \rightarrow \mathbb{H}^3/\Gamma_1$, see \cite[Chapter 8.3]{thurston1979geometry}. Similarly, the convex core of $\mathbb{H}^3/\Gamma$ is the image of $\mathrm{CH}(\Lambda(\Gamma))$ under the projection $p \colon \mathbb{H}^3 \rightarrow \mathbb{H}^3/\Gamma$. It follows from $\Lambda(\Gamma_1) = \Lambda(\Gamma)$ that $\mathrm{Vol}(\mathrm{CH}(\Lambda(\Gamma_1))/\Gamma_1) = k \cdot \mathrm{Vol}(\mathrm{CH}(\Lambda(\Gamma))/\Gamma) $, proving the claim.
		
	We also have  
	\[ i(P, P') = k \cdot i(P_1 , P'_1) = (g-1) i(P_1 , P'_1). \]
	Fix the pants decompositions $P_1, P'_1$ on $\Sigma_2$ and the homomorphism $\phi$, and allow $k$ to vary. To simplify the notation set $A = \mathrm{Vol}(\mathrm{CC}(M_1))$ and $B = i(P_1, P_1')$, and note that $A, B>0$ are universal constants. Therefore, for $C>0$ sufficiently small we have 
	\begin{align*}
		\mathrm{Vol}(\mathrm{CC}(M)) &=   (g-1) \mathrm{Vol}(\mathrm{CC}(M_1))= A (g-1) \geq \\ 
		& C \frac{g}{\log(g)} \log(B (g-1)) = C \frac{g}{\log(g)} \log(i(P, P')).
	\end{align*}
	For example if $C = \frac{A}{3 \log(B)}$, then the above inequality is satisfied because
	\begin{align*}
		C \frac{g}{\log(g)} \log(B (g-1)) &= C  \frac{g}{\log(g)} (\log B + \log(g-1)) < \\
		&C  \frac{g}{\log(g)} (\log B + \log g) = C \log(B) \frac{g}{\log(g)} + C g < \\
		& \frac{A}{3}  \frac{g}{\log(g)} + \frac{A}{3 \log(B)} g  < \\
		& \frac{A}{3}  g + \frac{A}{3} g = \frac{2A}{3} g \leq  A(g-1). 		
	\end{align*}

In the penultimate line above we used the fact that $B = i(P_1, P_1') \geq 3$ since every curve in $P_1$ intersects $P_1'$ at least once, and so $\log(B)>1$. In the last line we used that $g \geq 3$ and hence $2g \leq 3(g-1)$. 

\end{proof}

\section{Application to Teichm\"uller geometry}

Our next application is to Teichm\"uller space, endowed with the Weil-Petersson metric. Denote by $\mathrm{Teich}(S)$ the Teichm\"uller space of an orientable surface $S$ of finite type (in other words, the space of marked hyperbolic metrics possibly with cusps but with no boundary). Denote its Weil--Petersson metric by $d_{\mathrm{WP}}$.

\theoremstyle{theorem}
\newtheorem*{Weil-Petersson distance}{Theorem \ref{thm: Weil-Petersson distance}}
\begin{Weil-Petersson distance}
Let $S$ be an orientable surface of finite type, and let $X, Y \in \mathrm{Teich}(S)$. Let $P_X$ and $P_Y$ be pants decompositions for $X$ and $Y$ respectively, in which each curve has length at most $2 \pi |\chi(S)|$, which exist by a theorem of Parlier \cite{parlier}. Then
\[ d_{\mathrm{WP}}(X,Y) \leq O(|\chi(S)|^2) \hspace{1mm} ( 1 + \log (i(P_X, P_Y)+1) ). \]
\end{Weil-Petersson distance}

\begin{proof}
Set $L = 2 \pi |\chi(S)|$. By Parlier's quantified version of Bers's theorem \cite{parlier}, which is building on and improving the work of Buser and Sepp\"al\"a \cite{buser1992symmetric}, every hyperbolic metric $X$ on $S$ has a pants decomposition $P$ in which every curve has length at most $L$. Let $\ell_X(P)$ denote the sum of the lengths of the curves in $P$. Hence, $\ell_X(P) \leq 3|\chi(S)| L$. Let $N(P)$ be the nodal surface where each curve in $P$ has length pinched to zero. We may view $N(P)$ as a point in the metric completion of $\mathrm{Teich}(S)$ with the Weil-Petersson distance. Wolpert showed (Corollary 4.10 in \cite{wolpert2008lengthfunctions}) that the Weil-Petersson distance from $X$ to $N(P)$ is at most 
\[ \sqrt{2 \pi \ell_X(P)} \leq \sqrt{ 12 } \pi |\chi(S)|.\] 
Therefore, $d_{\mathrm{WP}}(X,Y) \leq d_{\mathrm{WP}}(N(P_X),N(P_Y)) + O(|\chi(S)|)$. 

Cavendish and Parlier \cite{cavendishparlier} introduced a metric graph which they term the \emph{cubical pants graph} $\mathcal{CP}(S)$. This is obtained from the pants graph $\mathcal{P}(S)$ by adding edges (which may have length greater than one). Hence, for the two pants decompositions $P_X$ and $P_Y$, their distance in $\mathcal{CP}(S)$ is at most their distance in $\mathcal{P}(S)$. In Lemma 4.1 of \cite{cavendishparlier}, it is shown that there is an absolute constant $C$ such that 
\[d_{\mathrm{WP}}(N(P_X),N(P_Y)) \leq C \, d_{\mathcal{CP}(S)}(P_X, P_Y).\]
So, 
\begin{align*}
d_{\mathrm{WP}}(X,Y) &\leq d_{\mathrm{WP}}(N(P_X),N(P_Y)) + O(|\chi(S)|) \\
& \leq C \, d_{\mathcal{CP}(S)}(P_X, P_Y) + O(|\chi(S)|) \\
& \leq C \, d_{\mathcal{P}(S)}(P_X, P_Y) + O(|\chi(S)|) \\
& \leq O(|\chi(S)|^2) \hspace{1mm} ( 1 + \log (i(P_X, P_Y)) ),
\end{align*}
where the latter inequality is from Theorem \ref{thm:upper bound for distance in pants graph}.
\end{proof}

\section{Questions}

\begin{question}
	Is there an algorithm that takes as input a compact connected orientable surface $S$ and two pants decompositions $P$ and $P'$, and computes the distance $d_{\mathcal{P}}(P,P')$ in the pants graph, in time that is a polynomial function of $|\chi(S)|$ and $\log(1+i(P, P'))$?
\end{question}

It follows from the work of Irmer \cite{Irmer2015stable} that there is an algorithm that computes the distance in the pants graph. 

A simple argument shows that the distance in the flip graph of one-vertex triangulations is bounded from above by the intersection number; see \cite[Lemma 2.1]{disarlo2019geometry}. This prompts the following question.

\begin{question}
	Is there a universal constant $C>0$ such that for every compact orientable surface $S$ and pants decompositions $P$ and $P'$ of $S$, the distance between $P$ and $P'$ in the pants graph is at most $C i(P,P')$?
\end{question}

\bibliographystyle{abbrv}
\bibliography{references-pants-graph.bib}

\begin{thebibliography}{10}

\bibitem{agol2003small}
I.~Agol.
\newblock Small 3-manifolds of large genus.
\newblock {\em Geometriae Dedicata}, 102(1):53--64, 2003.

\bibitem{agol2006computational}
I.~Agol, J.~Hass, and W.~Thurston.
\newblock The computational complexity of knot genus and spanning area.
\newblock {\em Transactions of the American Mathematical Society},
  358(9):3821--3850, 2006.

\bibitem{baroni2024uniformly}
F.~Baroni.
\newblock Uniformly polynomial-time classification of surface homeomorphisms.
\newblock {\em arXiv preprint arXiv:2402.00231}, 2024.

\bibitem{bell2021simplifying}
M.~C. Bell.
\newblock Simplifying triangulations.
\newblock {\em Discrete \& Computational Geometry}, 66(1):1--11, 2021.

\bibitem{bell}
M.~C. Bell and R.~Webb.
\newblock Personal communication.

\bibitem{bell2016applications}
M.~C. Bell and R.~C. Webb.
\newblock Applications of fast triangulation simplification.
\newblock {\em arXiv preprint arXiv:1605.03514}, 2016.

\bibitem{bers1985inequality}
L.~Bers.
\newblock An inequality for riemann surfaces.
\newblock {\em Differential Geometry and Complex Analysis: A Volume Dedicated
  to the Memory of Harry Ernest Rauch}, pages 87--93, 1985.

\bibitem{brock2003weil}
J.~Brock.
\newblock The {W}eil-{P}etersson metric and volumes of 3-dimensional hyperbolic
  convex cores.
\newblock {\em Journal of the American Mathematical Society}, 16(3):495--535,
  2003.

\bibitem{buser1980riemannsche}
P.~Buser.
\newblock {\em Riemannsche Fl{\"a}chen und L{\"a}ngenspektrum vom
  trigonometrischen Standpunkt aus}.
\newblock PhD thesis, Mathematisches Institut der Univ., 1980.

\bibitem{buser1992symmetric}
P.~Buser and M.~Sepp{\"a}l{\"a}.
\newblock Symmetric pants decompositions of {R}iemann surfaces.
\newblock 1992.

\bibitem{canary2003approximation}
R.~D. Canary, M.~Culler, S.~Hersonsky, and P.~B. Shalen.
\newblock Approximation by maximal cusps in boundaries of deformation spaces of
  {K}leinian groups.
\newblock {\em Journal of Differential Geometry}, 64(1):57--109, 2003.

\bibitem{cavendishparlier}
W.~Cavendish and H.~Parlier.
\newblock Growth of the {W}eil-{P}etersson diameter of moduli space.
\newblock {\em Duke Math. J.}, 161(1):139--171, 2012.

\bibitem{disarlo2019geometry}
V.~Disarlo and H.~Parlier.
\newblock The geometry of flip graphs and mapping class groups.
\newblock {\em Transactions of the American Mathematical Society},
  372(6):3809--3844, 2019.

\bibitem{dunfield2006random}
N.~M. Dunfield and D.~P. Thurston.
\newblock A random tunnel number one 3--manifold does not fiber over the
  circle.
\newblock {\em Geometry \& Topology}, 10(4):2431--2499, 2006.

\bibitem{dynnikov2007complexity}
I.~Dynnikov and B.~Wiest.
\newblock On the complexity of braids.
\newblock {\em Journal of the European Mathematical Society}, 9(4):801--840,
  2007.

\bibitem{erickson2013tracing}
J.~Erickson and A.~Nayyeri.
\newblock Tracing compressed curves in triangulated surfaces.
\newblock {\em Discrete Comput Geom}, 49:823–863, 2013.

\bibitem{Frab2001rank}
B.~Farb, A.~Lubotzky, and Y.~Minsky.
\newblock {Rank-1 phenomena for mapping class groups}.
\newblock {\em Duke Mathematical Journal}, 106(3):581 -- 597, 2001.

\bibitem{harvey1981boundary}
W.~J. Harvey.
\newblock Boundary structure of the modular group.
\newblock In {\em Riemann surfaces and related topics: Proceedings of the 1978
  Stony Brook Conference (State Univ. New York, Stony Brook, NY, 1978)},
  volume~97, pages 245--251, 1981.

\bibitem{hatcher1980presentation}
A.~Hatcher and W.~Thurston.
\newblock A presentation for the mapping class group of a closed orientable
  surface.
\newblock {\em Topology}, 19(3):221--237, 1980.

\bibitem{hempel20013}
J.~Hempel.
\newblock 3-manifolds as viewed from the curve complex.
\newblock {\em Topology}, 40(3):631--657, 2001.

\bibitem{Irmer2015stable}
I.~Irmer.
\newblock Stable lengths on the pants graph are rational.
\newblock {\em New York J. Math.}, 21:1153--1168, 2015.

\bibitem{lackenby2024some}
M.~Lackenby.
\newblock Some fast algorithms for curves in surfaces.
\newblock {\em arXiv preprint arXiv:2401.16056}, 2024.

\bibitem{lackenby2019triangulation}
M.~Lackenby and J.~S. Purcell.
\newblock The triangulation complexity of fibred 3-manifolds.
\newblock {\em arXiv preprint arXiv:1910.10914}, 2019.

\bibitem{lickorish1962representation}
W.~R. Lickorish.
\newblock A representation of orientable combinatorial 3-manifolds.
\newblock {\em Annals of Mathematics}, pages 531--540, 1962.

\bibitem{mosher1988tiling}
L.~Mosher.
\newblock Tiling the projective foliation space of a punctured surface.
\newblock {\em Transactions of the American Mathematical Society}, pages 1--70,
  1988.

\bibitem{parlier}
H.~Parlier.
\newblock A shorter note on shorter pants.
\newblock {\em arXiv preprint arXiv:2304.06973}, 2023.

\bibitem{penner1992combinatorics}
R.~C. Penner and J.~L. Harer.
\newblock {\em Combinatorics of train tracks}.
\newblock Number 125. Princeton University Press, 1992.

\bibitem{sleator1988rotation}
D.~D. Sleator, R.~E. Tarjan, and W.~P. Thurston.
\newblock Rotation distance, triangulations, and hyperbolic geometry.
\newblock {\em Journal of the American Mathematical Society}, 1(3):647--681,
  1988.

\bibitem{taylor2016products}
S.~J. Taylor and A.~Zupan.
\newblock Products of farey graphs are totally geodesic in the pants graph.
\newblock {\em Journal of Topology and Analysis}, 8(02):287--311, 2016.

\bibitem{thurston1979geometry}
W.~P. Thurston.
\newblock {\em The geometry and topology of three-manifolds}.
\newblock Princeton University Princeton, NJ, 1979.

\bibitem{wolpert2008lengthfunctions}
S.~A. Wolpert.
\newblock Behavior of geodesic-length functions on {T}eichm\"{u}ller space.
\newblock {\em J. Differential Geom.}, 79(2):277--334, 2008.

\end{thebibliography}

\end{document}